\documentclass[11pt,a4paper,leqno]{amsart}
\usepackage[T1]{fontenc}
\usepackage[utf8]{inputenc}

\usepackage{lmodern}
\usepackage{microtype}

\usepackage{amsthm}
\usepackage{amsmath}
\usepackage{amsfonts}
\usepackage{amssymb}
\usepackage{esint}
\usepackage[vmargin=3.5cm]{geometry}

\usepackage[unicode=true, pdfusetitle]{hyperref}
\let\C\undefined

\usepackage[abbrev,backrefs]{amsrefs}
\usepackage{cleveref}

\newtheorem{proposition}{Proposition}[section]
\newtheorem{theorem}[proposition]{Theorem}
\newtheorem{lemma}[proposition]{Lemma}

\theoremstyle{definition}

\newtheorem{openproblem}{Open Problem}
\newtheorem{remark}[proposition]{Remark}

\newcommand{\defeq}{\coloneqq}

\newcommand{\Nset}{\mathbb{N}}
\newcommand{\Zset}{\mathbb{Z}}
\newcommand{\Rset}{\mathbb{R}}
\newcommand{\Cset}{\mathbb{C}}
\newcommand{\Sset}{\mathbb{S}}
\newcommand{\Qset}{\mathbb{Q}}
\newcommand{\Bset}{\mathbb{B}}

\newcommand{\dif}{\,\mathrm{d}}
\usepackage{mathtools}

\newcommand{\compose}{\,\circ\,}
\newcommand{\manifold}[1]{\mathcal{#1}}
\newcommand{\lifting}[1]{\smash{\widetilde{#1}}}
\newcommand{\homog}[1]{\smash{\dot{#1}}}
\DeclarePairedDelimiter{\brk}{(}{)}
\DeclarePairedDelimiter{\abs}{\lvert}{\rvert}
\DeclarePairedDelimiter{\norm}{\lVert}{\rVert}

\DeclarePairedDelimiter{\floor}{\lfloor}{\rfloor}
\DeclarePairedDelimiter{\ceil}{\lceil}{\rceil}
\DeclarePairedDelimiterX{\intvc}[2]{[}{]}{#1,#2}
\DeclarePairedDelimiterX{\intvl}[2]{(}{]}{#1,#2}
\DeclarePairedDelimiterX{\intvr}[2]{[}{)}{#1,#2}
\DeclarePairedDelimiterX{\intvo}[2]{(}{)}{#1,#2}
\DeclarePairedDelimiterX{\setcond}[2]{\{}{\}}{#1 \,\delimsize\vert\, #2}

\newcommand{\restr}[1]{\!\upharpoonright_{#1}}

\providecommand{\st}{\,\vert\,}
\newcommand\stSymbol[1][]{%
\nonscript\;#1\vert
\allowbreak
\nonscript\;
\mathopen{}}
\DeclarePairedDelimiterX\set[1]\{\}{%
\renewcommand\st{\stSymbol[\delimsize]}
#1
}

\DeclareMathOperator{\dist}{dist}
\DeclareMathOperator{\supp}{supp}
\DeclareMathOperator{\Aut}{Aut}
\DeclareMathOperator{\inj}{inj}
\DeclareMathOperator{\diam}{diam}

\usepackage{constants}

\renewcommand{\PrintDOI}[1]{%
  \href{http://dx.doi.org/#1}{doi:#1}%
}
\BibSpec{book}{%
    +{}  {\PrintPrimary}                {transition}
    +{,} { \textit}                     {title}
    +{.} { }                            {part}
    +{:} { \textit}                     {subtitle}
    +{,} { \PrintEdition}               {edition}
    +{}  { \PrintEditorsB}              {editor}
    +{,} { \PrintTranslatorsC}          {translator}
    +{,} { \PrintContributions}         {contribution}
    +{,} { }                            {series}
    +{,} { \voltext}                    {volume}
    +{,} { }                            {publisher}
    +{,} { }                            {organization}
    +{,} { }                            {address}
    +{,} { }                            {status}
    +{,} { \PrintDOI}                   {doi}
    +{}  { \parenthesize}               {language}
    +{}  { \PrintTranslation}           {translation}
    +{;} { \PrintReprint}               {reprint}
    +{.} { }                            {note}
    +{.} {}                             {transition}
    +{}  {\SentenceSpace \PrintReviews} {review}
}
\author{Jean Van Schaftingen}

\title{%
Lifting of fractional Sobolev mappings to noncompact covering spaces%
}

\date{July 7, 2023}

\address{
Universit\'e catholique de Louvain, Institut de Recherche en Math\'ematique et Physique, Chemin du Cyclotron 2 bte L7.01.01, 1348 Louvain-la-Neuve, Belgium}
\email{Jean.VanSchaftingen@UCLouvain.be}


\begin{document}

\begin{abstract}
Given compact Riemannian manifolds $\mathcal{M}$ and $\mathcal{N}$, a Riemannian covering $\pi : \smash{\widetilde{\mathcal{N}}} \to \mathcal{N}$  by a noncompact covering space $\smash{\widetilde{\mathcal{N}}}$, $1 < p < \infty$ and $0 < s < 1$, the space of liftings of fractional Sobolev maps in $\smash{\dot{W}^{s, p}} (\mathcal{M}, \mathcal{N})$ is characterized when $sp > 1$ and an optimal nonlinear fractional Sobolev estimate is obtained when moreover $sp \ge \dim \mathcal{M}$.
A nonlinear characterization of the sum of spaces $\smash{\dot{W}^{s, p}} (\mathcal{M}, \mathbb{R}) + \smash{\dot{W}^{1, sp}} (\mathcal{M}, \mathbb{R})$ is also provided.
\end{abstract}

\thanks{The author was supported by the Mandat d'Impulsion Scientifique F.4523.17, ``Topological singularities of Sobolev maps'' of the Fonds de la Recherche Scientifique--FNRS and by the Projet de Recherche T.0229.21 ``Singular Harmonic Maps and Asymptotics of Ginzburg--Landau Relaxations'' of the
Fonds de la Recherche Scientifique--FNRS}

\maketitle

\tableofcontents

\section{Introduction}

Given a covering map \(\pi : \lifting{\manifold{N}} \to \manifold{N}\), that is, a map \(\pi\) such that for every \(y \in \manifold{N}\) there exists some open set \(U \subseteq \manifold{N}\) such that \(y \in U\) and \(\pi^{-1} (U)\) is a disjoint union of open subsets of \(\lifting{\manifold{N}}\) on which \(\pi\) is a homeomorphism, the classical topological lifting theory states that if \(\manifold{M}\) is a simply-connected topological manifold and if \(\pi\) is surjective, then every mapping \(u \in C (\manifold{M}, \manifold{N})\) can be written as \(u = \pi \compose \lifting{u}\) for some map \(\lifting{u} \in  C (\manifold{M}, \lifting{\manifold{N}})\) (see for example \cite{Hatcher_2002}*{prop.\ 1.33}).
For instance the universal covering of the circle \(\pi : \Rset \to \Sset^1\) defined for each \(\lifting{y} \in \Rset\) by \(\pi (\lifting{y}) \defeq e^{i \lifting{y}} \in \Sset^1 \subseteq \Rset^2 \simeq \Cset\) allows one to classify the homotopy classes of maps from the circle \(\Sset^1\) to itself (see for example \cite{Hatcher_2002}*{th.\ 1.7}).

When the manifolds \(\manifold{N}\) and \(\lifting{\manifold{N}}\) are both endowed with a Riemannian metric, we say that \(\pi : \lifting{\manifold{N}} \to \manifold{N}\) is a \emph{Riemannian covering} whenever it is a covering and it is a local isometry, that is, it preserves the metric tensor. In fact if \(\manifold{N}\) is a Riemannian manifold and \(\pi\) is a topological covering map, there exists a unique Riemannian metric on \(\lifting{\manifold{N}}\) such that \(\pi : \lifting{\manifold{N}} \to \manifold{N}\) is a Riemannian covering  (see \citelist{\cite{Lee_2018}*{prop.\ 2.31}\cite{Gallot_Hullin_Lafontaine_2004}*{2.A.4}}).

Given a Riemannian covering \(\pi : \lifting{\manifold{N}} \to \manifold{N}\), a Riemannian manifold \(\manifold{M}\), \(s \in \intvl{0}{1}\) and \(p \in \intvr{1}{\infty}\), the \emph{lifting problem in Sobolev spaces} amounts to determine whether each mapping \(u \in \homog{W}^{s, p}(\manifold{M}, \manifold{N})\) can be written as \(u = \pi \compose \lifting{u}\) on \(\manifold{M}\), for some map \(\lifting{u} \in \homog{W}^{s, p}(\manifold{M}, \lifting{\manifold{N}})\) \citelist{\cite{Bourgain_Brezis_Mironescu_2000}\cite{Bethuel_Chiron_2007}}.

\medskip

When \(s = 1\), the space \(\homog{W}^{1, p}(\manifold{M}, \manifold{N})\) is the \emph{homogeneous first-order Sobolev space} defined --- if the Riemannian manifold \(\manifold{N}\) is assumed without loss of generality in view of Nash's embedding theorem   \cite{Nash_1956} to be isometrically embedded into some Euclidean space \(\Rset^\nu\) --- as 
\begin{equation*}
    \homog{W}^{1, p} (\manifold{M}, \manifold{N})
    \defeq 
    \set*{u : \manifold{M} \to \manifold{N} \st u \text{ is weakly differentiable and } \int_{\manifold{M}} \abs{D u}^p < \infty}.
\end{equation*}
If the domain manifold \(\manifold{M}\) is simply-connected, then the first-order Sobolev spaces in which each map admits a lifting have been characterized for the universal covering of the circle \(\pi : \Rset \to \Sset^1\) by Bourgain, Brezis \& Mironescu \cite{Bourgain_Brezis_Mironescu_2000}*{th.\ 3} and for a general Riemannian covering map \(\pi : \lifting{\manifold{N}} \to \manifold{N}\) by Bethuel \& Chiron \cite{Bethuel_Chiron_2007}*{th.\ 1} (see also \cite{Brezis_Mironescu_2021}*{th.\ 1.1}): 
if the covering \(\pi\) is surjective and not injective,
every map \(u \in \homog{W}^{1, p} (\manifold{M}, \manifold{N})\) can be written as \(u = \pi \compose \lifting{u}\) for some mapping \(\lifting{u} \in \homog{W}^{1, p}(\manifold{M}, \lifting{\manifold{N}})\) if and only if \(p \ge \min \set{2, \dim \manifold{M}}\); moreover one has then 
\begin{equation}
\label{eq_dingooWohai3veorooz4maeb}\int_{\manifold{M}} \abs{D \lifting{u}}^p
=
\int_{\manifold{M}} \abs{D u}^p;
\end{equation}
once the existence of the lifting \(\lifting{u}\) is known, the identity \eqref{eq_dingooWohai3veorooz4maeb} follows directly from the chain rule for the Sobolev functions since the covering map \(\pi\) is a local Riemannian isometry so that \(\abs{D \lifting{u}} = \abs{D u}\) almost everywhere on \(\manifold{M}\).

\medskip

In the fractional case \(0 < s < 1\), the corresponding \emph{homogeneous fractional Sobolev--Slobodeckiĭ space} \(\homog{W}^{s, p}(\manifold{M}, \manifold{N})\) can be defined through the finiteness of the Gagliardo fractional energy as 
\begin{equation}
\homog{W}^{s, p} (\manifold{M}, \manifold{N})
\defeq  \set[\Bigg]{u : \manifold{M} \to \manifold{N} \st \norm{u}_{\homog{W}^{s, p}(\manifold{M})}^p \defeq \smashoperator{\iint_{\manifold{M} \times \manifold{M}}} \frac{d_{\manifold{N}} \brk{u (y), u (x)}^p}{d_{\manifold{M}} (y, x)^{m + sp}} \dif y \dif x < \infty},
\end{equation}
where \(d_{\manifold{M}}\) and \(d_{\manifold{N}}\) respectively denote the geodesic distances on the connected Riemannian manifolds \(\manifold{M}\) and \(\manifold{N}\) and where \(m \defeq \dim \manifold{M}\).

When \(sp < 1\), by the works of Bourgain, Brezis \& Mironescu for the universal covering of the circle \(\pi : \Rset \to \Sset^1\) \citelist{\cite{Bourgain_Brezis_Mironescu_2000}*{th.\ 2}\cite{Brezis_Mironescu_2021}*{th.\ 5.1 \& 5.2}} and of Bethuel \& Chiron \cite{Bethuel_Chiron_2007}*{th.\ 3}, 
every map \(u \in \homog{W}^{s, p} (\manifold{M}, \manifold{N})\) can be written as \(u = \pi \compose \lifting{u}\) with \(\lifting{u} \in \homog{W}^{s, p}(\manifold{M}, \lifting{\manifold{N}})\)  and one has then the lifting estimate
\begin{equation}
\label{eq_baeJ8oocoorohshaijooj9Lo}
\smashoperator{\iint_{\manifold{M} \times \manifold{M}}} \frac{d_{\lifting{\manifold{N}}} \brk{\lifting{u} (y), \lifting{u} (x)}^p}{d_{\manifold{M}} (y, x)^{m + sp}} \dif y \dif x
\le C \smashoperator{
\iint_{\manifold{M} \times \manifold{M}}} \frac{d_{\manifold{N}} \brk{u (y), u (x)}^p}{d_{\manifold{M}} (y, x)^{m + sp}} \dif y \dif x.
\end{equation}
In this régime \(sp < 1\), fractional Sobolev maps are quite rough mappings, and the possibility of jumps leaves much room for the construction of the lifting, which is quite challenging because of the highly nonunique character of the lifting.

When the covering space \(\lifting{\manifold{N}}\) is \emph{compact}, fractional Sobolev spaces \(\homog{W}^{s, p} (\manifold{M}, \manifold{N})\) for which any map admits a lifting have been characterized in the works of Bethuel \& Chiron \cite{Bethuel_Chiron_2007}*{th.\ 3}, and of Mironescu \& Van Schaftingen \cite{Mironescu_VanSchaftingen_2021}: if the covering \(\pi\) is surjective and not injective, every map \(u \in \homog{W}^{s, p} (\manifold{M}, \manifold{N})\) can be written as \(u = \pi \compose \lifting{u}\) for some \(\lifting{u} \in \homog{W}^{s, p}(\manifold{M}, \lifting{\manifold{N}})\) if and only if \(sp \ge \min \set{2, \dim \manifold{M}}\). 
Moreover, if \(sp > 1\), the estimate \eqref{eq_baeJ8oocoorohshaijooj9Lo} holds for any \(u \in \homog{W}^{s, p} (\manifold{M}, \manifold{N})\) that can be written as \(u = \pi \compose \lifting{u}\) with \(\lifting{u} \in \homog{W}^{s, p}(\manifold{M}, \lifting{\manifold{N}})\); this crucial estimate is a consequence of a reverse oscillation estimate, combined with the observation that the diameter of the covering space \(\lifting{\manifold{N}}\) can be bounded by a multiple of \(\inj(\manifold{N})\), the injectivity radius of the manifold \(\manifold{N}\).

\medskip

When the covering space \(\lifting{\manifold{N}}\) is \emph{not compact}, one encounters an \emph{analytical obstruction} for \(1 \le sp < m\): there exist maps in \(\smash{\homog{W}^{s, p}} (\manifold{M}, \manifold{N})\) that are smooth except at a single point and that cannot be written \(u = \pi \compose \lifting{u}\) for some map \(\lifting{u} \in \homog{W}^{s, p} (\manifold{M}, \lifting{\manifold{N}})\) \citelist{\cite{Bourgain_Brezis_Mironescu_2000}*{th.\ 2}\cite{Bethuel_Chiron_2007}*{th.\ 3}}.
This does not end the story, as one can still try to describe the functional space of liftings.

In the case of the universal covering of the circle \(\pi : \Rset \to \Sset^1\), the liftings have been characterized in a sequence of works by Bourgain, Brezis, Mironescu and Nguyen \citelist{\cite{Mironescu_2008}\cite{Mironescu_2010_Decomposition}\cite{Mironescu_preprint}\cite{Brezis_Mironescu_2021}\cite{Bourgain_Brezis_2003}\cite{Nguyen_2008}\cite{Mironescu_2010}}: 

\begin{theorem}
\label{theorem_lifting_circle_sum}
Let \(\manifold{M}\) be compact Riemannian manifold, let \(m \defeq  \dim \manifold{M}\), let \(s \in \intvo{0}{1}\) and let \(p \in \intvo{1}{\infty}\). 
If \(\manifold{M}\) is simply-connected and if \(sp \ge 2\), then there exists a constant \(C \in\intvo{0}{\infty}\) such that every map \(u \in \homog{W}^{s, p} (\manifold{M}, \Sset^1)\) can be written as \(u = \pi \compose \lifting{u}\) on \(\manifold{M}\) with \(\lifting{u} = \lifting{v} + \lifting{w}\) where the functions \(\lifting{v} \in \homog{W}^{s, p} (\manifold{M}, \Rset)\) and \(\lifting{w} \in \homog{W}^{1, sp} (\manifold{M}, \Rset)\) satisfy the estimate
\begin{equation}
\label{eq_aJeefah6Hee4Chaowoh1aade}
 \smashoperator{\iint_{\manifold{M} \times \manifold{M}}} \frac{\abs*{\lifting{v} (y) - \lifting{v} (x)}^p}{d_{\manifold{M}} (y, x)^{m + sp}} \dif y \dif x
 + \int_{\manifold{M}} \abs{D \lifting{w}}^{sp} 
 \le  C \smashoperator{\iint_{\manifold{M} \times \manifold{M}}} \frac{\abs{u (y) - u (x)}^p}{d_{\manifold{M}} (y, x)^{m + sp}} \dif y \dif x.
\end{equation}
\end{theorem}

In other words, \cref{theorem_lifting_circle_sum} states that any map \(u \in \homog{W}^{s, p} (\manifold{M}, \Sset^1)\) has a lifting \(\lifting{u} \in  \homog{W}^{s, p} (\manifold{M}, \Rset) + \homog{W}^{1, sp} (\manifold{M}, \Rset)\).

Conversely to \cref{theorem_lifting_circle_sum}, in view of the fractional Gagliardo--Nirenberg interpolation inequality (see for example \citelist{\cite{Berzis_Mironescu_2001}*{cor.\ 3.2}\cite{Runst_1986}*{lem.\ 2.1}\cite{Brezis_Mironescu_2018}}), if \(\lifting{u} = \lifting{v} + \lifting{w}\) with \(\lifting{v} \in \homog{W}^{s, p} (\manifold{M}, \Rset)\) and \(\lifting{w} \in \homog{W}^{1, sp} (\manifold{M}, \Rset)\), then \(u \defeq \pi \compose \lifting{u} \in \homog{W}^{s, p} (\manifold{M}, \Sset^1)\),
with inequality \eqref{eq_aJeefah6Hee4Chaowoh1aade} reversed.
\Cref{theorem_lifting_circle_sum} characterizes thus completely the lifting space of \(\homog{W}^{s, p} (\manifold{M}, \Sset^1)\) for \(sp \ge 2\) as the sum of linear spaces \(\homog{W}^{s, p} (\manifold{M}, \Rset) + \homog{W}^{1, sp} (\manifold{M}, \Rset)\).

The first goal of the present work is to obtain a counterpart of \cref{theorem_lifting_circle_sum} for a general covering map \(\pi : \lifting{\manifold{N}} \to \manifold{N}\) when the covering space \(\lifting{\manifold{N}}\) is not compact. 
This endeavour is delicate from its very beginning, since \(\homog{W}^{s, p} (\manifold{M}, \lifting{\manifold{N}}) + \homog{W}^{1, s p} (\manifold{M}, \lifting{\manifold{N}})\) has no straightforward definition or generalization when the covering space \(\lifting{\manifold{N}}\) is not a linear space.
We characterize the lifting space as follows.

\begin{theorem}
\label{theorem_lifting_truncated_noncompact}
Let \(\manifold{M}\) and let \(\manifold{N}\) be a compact Riemannian manifold, let \(m \defeq  \dim \manifold{M}\), let \(\pi : \lifting{\manifold{N}} \to \manifold{N}\) be a surjective Riemannian covering map, let \(s \in \intvo{0}{1}\) and let \(p \in \intvo{1}{\infty}\). 
If \(\manifold{M}\) is simply-connected and if \(sp \ge 2\), then there exists a constant \(C \in \intvo{0}{\infty}\) such that for every map \(u \in \homog{W}^{s, p} (\manifold{M}, \manifold{N})\) there exists a measurable map \(\lifting{u} : \manifold{M} \to \lifting{\manifold{N}}\) satisfying \(\pi \compose \lifting{u} = u\) almost everywhere on \(\manifold{M}\) and 
\begin{equation}
\label{eq_oGeuzoo2Aeyaifaeto0uelei}
 \smashoperator{\iint_{\manifold{M} \times \manifold{M}}} \frac{d_{\lifting{\manifold{N}}} (\lifting{u} (y), \lifting{u} (x))^p \wedge 1 }{d_{\manifold{M}} (y, x)^{m + sp}} \dif y \dif x
 \le  C \smashoperator{\iint_{\manifold{M} \times \manifold{M}}} \frac{d_{\manifold{N}} (u (y), u (x))^p}{d_{\manifold{M}} (y, x)^{m + sp}} \dif y \dif x.
\end{equation}
\end{theorem}

The integrand in the left-hand side \eqref{eq_oGeuzoo2Aeyaifaeto0uelei} only differs from the classical Gagliardo energy in the right-hand side by truncating through the minimum \(\wedge\)  operation the value of the distance at \(1\); in terms of metric space, this can be interpreted as taking on the covering space \(\lifting{\manifold{N}}\) a bounded distance for which the covering map \(\pi\) is an isometry at small scales.

The characterization of liftings of \cref{theorem_lifting_truncated_noncompact} is sharp, in the sense that if for some mapping \(\lifting{u}: \manifold{M} \to \lifting{\manifold{N}}\) the left-hand side of \eqref{eq_oGeuzoo2Aeyaifaeto0uelei} is finite, then by the local isometry property of liftings one has \(u = \pi \compose \lifting{u} \in \homog{W}^{s, p} (\manifold{M}, \manifold{N})\) together with the estimate \eqref{eq_oGeuzoo2Aeyaifaeto0uelei} reversed.

The core of the proof of \cref{theorem_lifting_truncated_noncompact} is the reverse oscillation estimate of Mironescu \& Van Schaftingen \cite{Mironescu_VanSchaftingen_2021}*{lem.\ 3.1} (see \cref{proposition_spgt1_estimate} below), combined with the approximation of maps that are smooth outside a finite union of manifolds of dimension \(\ceil{m - sp - 1}\) by Brezis \& Mironescu \cite{Brezis_Mironescu_2015}, a suitable variant of the fractional Rellich compactness theorem under a boundedness assumption on the left-hand side of \eqref{eq_oGeuzoo2Aeyaifaeto0uelei} (see \cref{prop_compactness} below) and, at a more technical level, the equivalent characterization of the lifting space (see \cref{proposition_equivalent_norms} below).

The lifting in the space of functions such that the left-hand side of \eqref{eq_oGeuzoo2Aeyaifaeto0uelei} is finite enjoys a uniqueness property.
In order to state this, we define the space
\begin{equation}
 X (\manifold{M}, \lifting{\manifold{N}})
 \defeq 
\set[\Bigg]{\lifting{u} : \manifold{M} \to \lifting{\manifold{N}} \st 
 \lifting{u} \text{ is measurable and } \smashoperator{\iint_{\substack{x, y \in \manifold{M}\\[.2em] d_{\lifting{\manifold{N}}} (\lifting{u} (x), \lifting{u} (y))\ge \inj(\manifold{N})/2}}} \frac{1}{d_{\manifold{M}} (x, y)^{m + 1}} \dif y \dif x < \infty};
\end{equation}
the latter space contains mappings for which the left-hand side of \eqref{eq_oGeuzoo2Aeyaifaeto0uelei}  is finite (see \cref{proposition_nonlinear_sum_in_X} below) and the 
uniqueness of the lifting then follows from the next proposition.

\begin{proposition}
\label{proposition_lifting_unique}
Let \(\manifold{M}\) be a compact Riemannian manifold and let \(\pi : \lifting{\manifold{N}} \to \manifold{N} \) be a Riemannian covering.
If \(\manifold{M}\) is connected, if \(\lifting{u}_0, \lifting{u}_1 \in X (\manifold{M},\lifting{\manifold{N}})\) and if \(\pi \compose \lifting{u}_0 = \pi \compose \lifting{u}_1\) almost everywhere on \(\manifold{M}\),
then either \(\lifting{u}_0 = \lifting{u}_1\) almost everywhere on \(\manifold{M}\) or \(\lifting{u}_0 \ne \lifting{u}_1\) almost everywhere on \(\manifold{M}\).
\end{proposition}

When \(1 < sp \le 2\) or when the manifold \(\manifold{M}\) is not simply-connected, topological obstructions can exclude the existence of a lifting; it turns out however that when a lifting exists in \(X (\manifold{M}, \lifting{\manifold{N}})\), then such a lifting has to satisfy the estimate of \cref{theorem_lifting_truncated_noncompact}.

\begin{theorem}
\label{theorem_lifting_truncated_noncompact_apriori}
Let \(\manifold{M}\) and let \(\manifold{N}\) be a compact Riemannian manifold, let \(m \defeq  \dim \manifold{M}\), let \(\pi : \lifting{\manifold{N}} \to \manifold{N}\) be a Riemannian covering map, let \(s \in \intvo{0}{1}\) and let \(p \in \intvo{1}{\infty}\). 
If \(sp > 1\), then there exists a constant \(C \in \intvo{0}{\infty}\) such that if
\(\lifting{u} \in X (\manifold{M}, \lifting{\manifold{N}})\) and if \(u \defeq \pi \compose \lifting{u} \in \homog{W}^{s, p} (\manifold{M}, \manifold{N})\), then 
\begin{equation*}
 \smashoperator{\iint_{\manifold{M} \times \manifold{M}}} \frac{d_{\lifting{\manifold{N}}} (\lifting{u} (y), \lifting{u} (x))^p \wedge 1 }{d_{\manifold{M}} (y, x)^{m + sp}} \dif y \dif x
 \le  C \smashoperator{\iint_{\manifold{M} \times \manifold{M}}} \frac{d_{\manifold{N}} (u (y), u (x))^p}{d_{\manifold{M}} (y, x)^{m + sp}} \dif y \dif x.
\end{equation*}
\end{theorem}

The restriction to \(sp > 1\) is essential in \cref{theorem_lifting_truncated_noncompact_apriori} in both the compact and noncompact cases:  if \(sp = 1\) and if \(\pi\) is surjective and not injective, then there is no estimate on the lifting \cite{Mironescu_VanSchaftingen_2021}*{lem.\ 5.1}.

\Cref{theorem_lifting_truncated_noncompact,theorem_lifting_truncated_noncompact_apriori} motivate studying the quantity on the left-hand side of \eqref{eq_oGeuzoo2Aeyaifaeto0uelei}, which turns out be equivalent to a wide family of similar quantities.

\begin{proposition}
\label{proposition_equivalent_norms}
Let \(\manifold{M}\) be a compact Riemannian manifold with \(m \defeq  \dim \manifold{M}\), let \(\lifting{\manifold{N}}\) be a Riemannian manifold, let \(s \in \intvo{0}{1}\), let \(p \in \intvo{1}{\infty}\) and let \(q_0, q_1 \in \intvr{0}{\infty}\).
If \(q_0 \vee q_1 \vee 1 < sp\), then there exists a constant \(C \in \intvo{0}{\infty}\) such that every measurable mapping \(\lifting{u} : \manifold{M} \to \lifting{\manifold{N}}\) satisfies
\begin{multline*}
\smashoperator[r]{\iint_{\manifold{M} \times \manifold{M}}}
\frac{d_{\lifting{\manifold{N}}} (\lifting{u} (y), \lifting{u} (x))^p \wedge d_{\lifting{\manifold{N}}} (\lifting{u} (y), \lifting{u} (x))^{q_0}
}{d_\manifold{M} (y, x)^{m + sp}} \dif y \dif x \\[-1.5em]
\le 
C \smashoperator{\iint_{\manifold{M} \times \manifold{M}}}
\frac{d_{\lifting{\manifold{N}}} (\lifting{u} (y), \lifting{u} (x))^p \wedge d_{\lifting{\manifold{N}}} (\lifting{u} (y), \lifting{u} (x))^{q_1}}{d_\manifold{M} (y, x)^{m + sp}}
\dif y \dif x .
\end{multline*}
\end{proposition}

In \cref{proposition_equivalent_norms}, the case \(q_0 \le q_1\) is trivial.
The estimate of \cref{proposition_equivalent_norms} generalizes similar estimates obtained in the context of estimates of homotopy classes with \(sp = \dim \manifold{M}\) \cite{VanSchaftingen_2020}*{\S 5}.

As a consequence of \cref{theorem_lifting_circle_sum}, \cref{theorem_lifting_truncated_noncompact} and \cref{proposition_equivalent_norms}, we obtain the following nonlinear characterization of the linear sum of Sobolev spaces.

\begin{theorem}
\label{theorem_nonlinear_sum}
Let \(\manifold{M}\) be compact Riemannian manifold, let \(m \defeq  \dim \manifold{M}\), let \(s \in \intvo{0}{1}\) and let \(p \in \intvo{1}{\infty}\). If \(sp > 1\) and if \(0 < q < sp\),
then 
\begin{multline}
\label{eq_eahep6eet2yoo2ohBeePee9w}
\set[\bigg]{f : \manifold{M} \to \Rset \st \smashoperator[r]{\iint_{\manifold{M} \times \manifold{M}}}
\frac{\abs{f (y) - f(x)}^p \wedge \abs{f (y) - f (x)}^{q}
}{d_\manifold{M} (y, x)^{m + sp}} \dif y \dif x  < \infty}\\[-.8em]
= \homog{W}^{s, p} (\manifold{M}, \Rset) + \homog{W}^{1, sp} (\manifold{M}, \Rset) . 
\end{multline}
Moreover, the quantities 
\begin{equation*}
 \smashoperator{\iint_{\manifold{M} \times \manifold{M}}}
\frac{\abs{f (y) - f(x)}^p \wedge \abs{f (y) - f (x)}^{q}
}{d_\manifold{M} (y, x)^{m + sp}} \dif y \dif x  
\end{equation*}
and 
\begin{equation}
 \inf_{\substack{g \in \homog{W}^{s, p} (\manifold{M}, \Rset)\\ h \in \homog{W}^{1, sp} (\manifold{M}, \Rset)\\
 f = g + h}} \; \smashoperator[r]{\iint_{\manifold{M} \times \manifold{M}}}
\frac{\abs{g (y) - g (x)}^p
}{d_\manifold{M} (y, x)^{m + sp}} \dif y \dif x
+ \int_{\manifold{M}} \abs{Dh}^{sp}
\end{equation}
are equivalent in the sense that each of them is bounded by constant multiple of the other.
\end{theorem}

\Cref{theorem_nonlinear_sum} complements the characterization of sums of fractional Sobolev spaces by Rodiac \& Van Schaftingen
\cite{Rodiac_VanSchaftingen_2021}, which states that if \(q > sp\)
\begin{multline}
\label{eq_tie7thooD8ochui0ush9do8r}
\set[\bigg]{f : \manifold{M} \to \Rset \st\smashoperator[r]{\iint_{\manifold{M} \times \manifold{M}}}
\frac{\abs{f (y) - f(x)}^p \wedge \abs{f (y) - f (x)}^{q}}{d_\manifold{M} (y, x)^{m + sp}} \dif y \dif x  < \infty}\\[-.8em]
= \homog{W}^{s, p} (\manifold{M}, \Rset) + \homog{W}^{\frac{sp}{q}, q} (\manifold{M}, \Rset) . 
\end{multline}

We give proof of \cref{theorem_nonlinear_sum} relying on the characterization of liftings of mappings into the circle \cref{theorem_lifting_circle_sum}; it would be enlightening to have a direct proof.

\begin{openproblem}
Give a direct proof of \cref{theorem_nonlinear_sum}.
\end{openproblem}

In the case \(q = sp\), it turns that the identifications \eqref{eq_eahep6eet2yoo2ohBeePee9w} and \eqref{eq_tie7thooD8ochui0ush9do8r} fail (see \cref{proposition_counterexample_critical}), leading to the following question:

\begin{openproblem}
Given a compact Riemannian manifold \(\manifold{M}\) with \(m = \dim \manifold{M}\), \(p \in \intvr{1}{\infty}\) and \(s \in \intvo{0}{1}\) such that \(sp>1\), characterize the set 
\begin{equation}
\label{eq_ahngeeyohVeikeiboh5ocee7}
X^{s, p} (\manifold{M}, \Rset) \defeq 
\set[\bigg]{f : \manifold{M} \to \Rset \st \smashoperator[r]{\iint_{\manifold{M} \times \manifold{M}}}
\frac{\abs{f (y) - f (x)}^p \wedge \abs{f (y) - f (x)}^{sp}}{d_\manifold{M} (y, x)^{m + sp}} \dif y \dif x  < \infty}.
\end{equation}
\end{openproblem}

We have some information about what the space \(X^{s, p} (\manifold{M}, \Rset)\) could be:
by \cref{theorem_nonlinear_sum} and by \eqref{eq_tie7thooD8ochui0ush9do8r}, we have 
\begin{equation}
\label{eq_aiteip7yiepeivie5ju8joiK}
\bigcup_{s < \theta < 1}
(\homog{W}^{s, p} (\manifold{M}, \Rset) + \homog{W}^{s/\theta, \theta p} (\manifold{M}, \Rset))
\subseteq 
  X^{s, p}(\manifold{M}, \Rset) 
  \subseteq 
  \homog{W}^{s, p} (\manifold{M}, \Rset) + \homog{W}^{1, sp} (\manifold{M}, \Rset),
\end{equation}
whereas by \cref{proposition_counterexample_critical} below, we have
\begin{equation}
\label{eq_ia4oDaeZaXahPoojahthie6a}
 \homog{W}^{1, sp} (\manifold{M}, \Rset)\not \subseteq X^{s, p} (\manifold{M}, \Rset),
\end{equation}
so that the second inclusion in \eqref{eq_aiteip7yiepeivie5ju8joiK} cannot be an equality.

\medskip

The second goal of the present work is to investigate \emph{estimates} for the lifting when \(sp \ge m\). In this case every map \(u \in \homog{W}^{s, p} (\manifold{M}, \manifold{N})\)  can be written as \(\pi \compose \lifting{u}\) with \(\lifting{u} \in \homog{W}^{s, p} (\manifold{M}, \lifting{\manifold{N}})\)  \citelist{\cite{Bourgain_Brezis_Mironescu_2000}*{th.\ 2}\cite{Bethuel_Chiron_2007}*{th.\ 3}} (see also \cite{Brezis_Mironescu_2021}*{th.\ 5.1 \& 5.2}).
When \(sp = 1 = m\), it is known that there is no estimate on the lifting when the covering map \(\pi\) is surjective and not injective \citelist{\cite{Bourgain_Brezis_Mironescu_2000}*{rem.\ 3}\cite{Mironescu_VanSchaftingen_2021}*{lem.\ 5.1}} (see also \cite{Brezis_Mironescu_2021}*{prop.\ 9.2}).
If the covering space \(\lifting{\manifold{N}}\) is not compact, then it is also known that there cannot be any estimate of the form \eqref{eq_baeJ8oocoorohshaijooj9Lo} (see \cite{Mironescu_Molnar_2015}*{prop.\ 5.7} for the universal covering of the circle \(\pi : \Rset \to \Sset^1\)).

For the universal covering of the circle \(\pi : \Rset \to \Sset^1\), Merlet and Mironescu \& Molnar have obtained the following nonlinear estimate \citelist{\cite{Merlet_2006}*{th.\ 1.1}\cite{Mironescu_Molnar_2015}*{th.\ 5.4}} (see also \cite{Brezis_Mironescu_2021}*{th.\ 9.6}).

\begin{theorem}
\label{theorem_estimate_lifting_circle}
Let \(\manifold{M}\) be a compact Riemannian manifold, let \(m \defeq  \dim \manifold{M}\), let \(s \in \intvo{0}{1}\) and let \(p \in \intvo{1}{\infty}\).
If \(sp \ge m\) and \(sp > 1\), then  there exists a constant \(C \in \intvo{0}{\infty}\) such that if \(\lifting{u} \in \homog{W}^{s, p} (\manifold{M}, \Rset)\) and if \(u \defeq e^{i \lifting{u}}\), 
we have
\begin{multline*}
\smashoperator{
  \iint_{\manifold{M} \times \manifold{M}}
  } 
  \frac
    {\abs{\lifting{u} (y) - \lifting{u} (x)}^p}
    {d_{\manifold{M}} (y, x)^{m + sp}} 
    \dif y 
    \dif x
\\[-1.5em]
\le 
C 
\brk[\Bigg]{
  \;
  \smashoperator[r]{
\iint_{\manifold{M} \times \manifold{M}}
} 
  \frac
    {\abs{u (y)- u (x)}^p}
    {d_{\manifold{M}} (y, x)^{m + sp}} 
    \dif y 
    \dif x
  + 
    \brk[\Bigg]{
    \;
    \smashoperator[r]{
      \iint_{\manifold{M} \times \manifold{M}}
      } 
    \frac
      {\abs*{u (y)- u (x)}^p}
      {d_{\manifold{M}} (y, x)^{m + sp}} 
      \dif y 
      \dif x
      }^\frac{1}{s}
      \;}.
\end{multline*} 
\end{theorem}

We generalize \cref{theorem_estimate_lifting_circle} to a general covering \(\pi : \lifting{\manifold{N}} \to \manifold{N}\).
\begin{theorem}
\label{theorem_estimate_lifting_noncompact}
Let \(\manifold{M}\) and \(\manifold{N}\) be compact Riemannian manifolds, let \(m \defeq  \dim \manifold{M}\), let \(\pi : \lifting{\manifold{N}} \to \manifold{N}\) be a Riemannian covering map, let \(s \in \intvo{0}{1}\) and let \(p \in \intvo{1}{\infty}\). 
If \(sp \ge m\) and \(sp > 1\), then there exists a constant \(C \in \intvo{0}{\infty}\) such that if 
\(\lifting{u} \in X (\manifold{M}, \lifting{\manifold{N}})\) and if \(u \defeq \pi \compose \lifting{u} \in \homog{W}^{s, p} (\manifold{M}, \manifold{N})\), we have \(\lifting{u} \in \homog{W}^{s, p} (\manifold{M}, \lifting{\manifold{N}})\) and 
\begin{multline}
\label{eq_Pei1Chu7Amie6ohz0phi6cai}
\smashoperator{\iint_{\manifold{M} \times \manifold{M}}} \frac{d_{\lifting{\manifold{N}}} \brk{\lifting{u} (y), \lifting{u} (x)}^p}{d_{\manifold{M}} (y, x)^{m + sp}} \dif y \dif x\\[-1em]
\le C 
\brk[\Bigg]{\;\smashoperator[r]{
\iint_{\manifold{M} \times \manifold{M}}} \frac{d_{\manifold{N}} (u (y), u (x))^p}{d_{\manifold{M}} (y, x)^{m + sp}} \dif y \dif x
+ \brk[\Bigg]{\;\smashoperator[r]{\iint_{\manifold{M} \times \manifold{M}}} \frac{d_{\manifold{N}} (u (y), u (x))^p}{d_{\manifold{M}} (y, x)^{m + sp}} \dif y \dif x}^\frac{1}{s}\;}.
\end{multline}
\end{theorem}

\Cref{theorem_estimate_lifting_circle} can be proved by combining the estimate \eqref{eq_aJeefah6Hee4Chaowoh1aade}  on the linear decomposition of the lifting with a fractional Sobolev embedding \cite{Mironescu_Molnar_2015}; the latter embedding turns out to be a consequence of \cref{theorem_estimate_lifting_noncompact} (see \cref{remark_Sobolev_embedding} below).
Since the decomposition of the lifting into a sum  \eqref{eq_aJeefah6Hee4Chaowoh1aade} does not subsist for a general covering space \(\lifting{\manifold{N}}\), we give a direct proof of \cref{theorem_estimate_lifting_noncompact}; the structure of the proof with weak-type estimates on some level sets of differences is akin to the proof of Marcinkiewicz’s real interpolation theorem and Sobolev's embedding theorem by interpolation (see for example \cite{Stein_1970}*{ch.\ I th.\ 5}).

\medskip

As a consequence of \cref{theorem_estimate_lifting_circle} and of the classical extension of traces in the fractional space \(\homog{W}^{1 - 1/p, p} (\manifold{M}, \Rset)\) into \(\homog{W}^{1, p} (\manifold{M} \times \intvo{0}{1}, \Rset)\), one gets the following extension estimate: if \(p \ge \dim \manifold{M} + 1\), then there exists a constant \(C \in \intvo{0}{\infty}\) such that every map \(u \in \homog{W}^{1-1/p, p} (\manifold{M}, \Sset^1)\) is the trace on \(\manifold{M} \times \set{0}\) of a mapping \(U \in \homog{W}^{1, p} (\manifold{M} \times \intvo{0}{1}, \Sset^1)\) satisfying the estimate 
\begin{multline}
\label{eq_eedebewaiwuh4aiFoh9faidi}
\int_{\manifold{M} \times \intvo{0}{1}}
\abs{D U}^p \\[-1em]
\le C\, \brk[\Bigg]{\;\smashoperator[r]{
\iint_{\manifold{M} \times \manifold{M}}} \frac{\abs{u (y)- u (x)}^p}{d_{\manifold{M}} (y, x)^{m + sp}} \dif y \dif x
+ \brk[\bigg]{\;\smashoperator[r]{\iint_{\manifold{M} \times \manifold{M}}} \frac{\abs*{u (y)- u (x)}^p}{d_{\manifold{M}} (y, x)^{m + sp}} \dif y \dif x}^\frac{p}{p - 1}\;}.
\end{multline}

For a general target manifold \(\manifold{N}\), it is known that if \(p \ge \dim \manifold{M} + 1\), every every map \(u \in \homog{W}^{1-1/p, p} (\manifold{M}, \manifold{N})\) is the trace of a mapping \(U \in \homog{W}^{1, p} (\manifold{M} \times \intvo{0}{1}, \manifold{N})\) \cite{Bethuel_Demengel_1995}*{th.\ 1}.
When \(p > \dim \manifold{M} + 1\), a compactness argument shows that the extension \(U\) can be taken to remain in a bounded set of \(\homog{W}^{1, p} (\manifold{M} \times \intvo{0}{1}, \manifold{N})\) when the trace \(u\)  remains bounded in \(\homog{W}^{1-1/p, p} (\manifold{M}, \manifold{N})\) (see for example \cite{Petrache_VanSchaftingen_2017}*{th.\ 4}).
When \(p = \dim \manifold{M} + 1\) and  \(\pi_{p - 1} (\manifold{N})\not \simeq \set{0}\), such a boundedness cannot hold \citelist{\cite{Petrache_Riviere_2014}*{prop.\ 2.8}\cite{Mironescu_VanSchaftingen_2021_Trace}*{th.\ 1.10}}; one still gets then estimates when the mapping \(u\) has a small fractional Sobolev energy and weak-type estimates in general \citelist{\cite{Petrache_Riviere_2014}\cite{Petrache_VanSchaftingen_2017}}.

In the particular case where \(\pi_{1} (\manifold{N}) \simeq \dotsb \simeq \pi_{\floor{p - 1}} (\manifold{N}) \simeq \set{0}\), where \(\floor{r} \in \Zset\) denotes the integer part of \(r \in \Rset\), Hardt \& Lin \cite{Hardt_Lin_1987}*{th.\ 6.2} have proved  that there exists a constant \(C \in \intvo{0}{\infty}\) such that 
 every map \(u \in \homog{W}^{1-1/p, p} (\manifold{M}, \manifold{N})\) is the trace of a mapping \(U \in \homog{W}^{1, p} (\manifold{M} \times \intvo{0}{1}, \manifold{N})\) satisfying the estimate 
\begin{equation}
\label{eq_lesot8Ahza8eeRee9wu0evei}
\smashoperator[r]{\int_{\manifold{M} \times \intvo{0}{1}}}
\abs{D U}^p 
\le C \smashoperator{
\iint_{\manifold{M} \times \manifold{M}}} \frac{d_{\manifold{N}}( u (y), u (x))^p}{d_{\manifold{M}} (y, x)^{m + sp}} \dif y \dif x.
\end{equation}
The surjectivity of the trace with the linear estimate \eqref{eq_lesot8Ahza8eeRee9wu0evei} fails when the homotopy group \(\pi_{\floor{p - 1}}(\manifold{N})\) is nontrivial or when one of the homotopy groups \(\pi_{1}(\manifold{N}), \dotsc, \pi_{\floor{p - 2}}(\manifold{N})\) is infinite \citelist{\cite{Hardt_Lin_1987}*{\S 6.3}\cite{Bethuel_Demengel_1995}*{th.\ 4}\cite{Bethuel_2014}*{prop.\ 1.13}\cite{Mironescu_VanSchaftingen_2021_Trace}*{th.\ 1.10}}.

The estimates \eqref{eq_eedebewaiwuh4aiFoh9faidi} and \eqref{eq_lesot8Ahza8eeRee9wu0evei} raise naturally the following question.

\begin{openproblem}
 Given compact Riemannian manifolds \(\manifold{M}\) and \(\manifold{N}\) and \(p \ge \dim \manifold{M} + 1\), is there a  constant \(C \in \intvo{0}{\infty}\) such that every map \(u \in \homog{W}^{1-1/p, p} (\manifold{M}, \manifold{N})\) is the trace on \(\manifold{M} \times \set{0}\) of a mapping \(U \in \homog{W}^{1, p} (\manifold{M} \times \intvo{0}{1}, \manifold{N})\) satisfying the estimate 
\begin{multline}
\smashoperator[r]{\int_{\manifold{M} \times \intvo{0}{1}}}
\abs{D U}^p \le C \brk[\Bigg]{\;\smashoperator[r]{
\iint_{\manifold{M} \times \manifold{M}}} \frac{d_{\manifold{N}} (u (y), u (x))^p}{d_{\manifold{M}} (y, x)^{m + sp}} \dif y \dif x
+ \brk[\bigg]{\;\smashoperator[r]{\iint_{\manifold{M} \times \manifold{M}}} \frac{d_{\manifold{N}} (u (y), u (x))^p}{d_{\manifold{M}} (y, x)^{m + sp}} \dif y \dif x}^\frac{p}{p - 1}\;}?
\end{multline}
\end{openproblem}

In the case where the fundamental group \(\pi_1(\manifold{N})\) is infinite and where \(\pi_2 (\manifold{N}) \simeq \dotsb \simeq \smash{\pi_{\floor{p - 1}}} (\manifold{N}) \simeq \set{0}\), although  \cref{theorem_estimate_lifting_noncompact} provides a lifting in \(\homog{W}^{1 - 1/p, p}(\manifold{M}, \lifting{\manifold{N}})\), a universal covering space \(\lifting{\manifold{N}}\) fails to be compact so that Hardt \& Lin’s theorem on the extension of traces \cite{Hardt_Lin_1987}*{th.\ 6.2} is not applicable.

\section{Characterizations of the lifting space and related estimates}

\subsection{A priori estimate for regular liftings}
We begin by proving an priori estimate on the lifting that will be the main analytical tool for the construction and estimate of liftings in \cref{theorem_lifting_truncated_noncompact,theorem_lifting_truncated_noncompact_apriori}.
Given a convex open set \(\Omega \subseteq \Rset^m\), we define the space of mappings that are essentially continuous on almost every segment of \(\Omega\)
\begin{equation}
\label{eq_definition_Y}
\begin{split}
 Y(\Omega, \lifting{\manifold{N}}) 
 \defeq
 \bigl\{ \lifting{u} : \Omega \to \lifting{\manifold{N}}  \,\st {}&{} \text{for almost every } x, y \in \Omega, \text{ there exists } \lifting{u}_{x, y} \in C (\intvc{x}{y}, \lifting{\manifold{N}})\\ 
 & \text{such that } \lifting{u}_{x, y}(x) = \lifting{u} (x),\, \lifting{u}_{x, y}(y) = \lifting{u} (y),\\
 &\text{and } \lifting{u}_{x, y} = \lifting{u}\vert_{\intvc{x}{y}} \text{ almost everywhere on }\intvc{x}{y}
 \bigr\},
 \end{split}
\end{equation}
for which we prove the following a priori estimate.

\begin{proposition}
\label{proposition_spgt1_estimate}
Let \(m \in \Nset\setminus \set{0}\), let \(s \in \intvo{0}{1}\) and let \(p \in \intvo{1}{\infty}\). 
If \(sp > 1\), then there exists a constant \(C \in \intvo{0}{\infty}\) such that if \(\Omega \subseteq \Rset^m\) is open and convex, if \(\lifting{u} \in Y (\Omega, \lifting{\manifold{N}})\) and if \(u \defeq \pi \compose \lifting{u} \in \homog{W}^{s, p}(\Omega, \manifold{N})\), then 
\begin{equation}
\label{eq_hao4yeiKieju4Riel5AxaeTh}
\smashoperator[l]{\iint_{\Omega \times \Omega}}
\frac{d_{\lifting{\manifold{N}}} (\lifting{u} (y), \lifting{u} (x))^p \wedge 1}{\abs{y - x}^{m +sp}} \dif y \dif x 
\le 
C \smashoperator{\iint_{\Omega \times \Omega}}
\frac{d_{\manifold{N}} (u (y), u (x))^p}{\abs{y - x}^{m + sp}} \dif y \dif x .
\end{equation}
\end{proposition}

\Cref{proposition_spgt1_estimate} was initially stated and proved in the case where the covering space \(\lifting{\manifold{N}}\) is compact \cite{Mironescu_VanSchaftingen_2021}, where it was an essential tool in the construction of liftings; 
the same argument also yields reverse superposition estimates in fractional Sobolev spaces \cite{VanSchaftingen_2022}.
We perform here a straighforward adaptation of the proof to the case where the covering space \(\lifting{\manifold{N}}\) is not compact.

As in the proof in the compact case \cite{Mironescu_VanSchaftingen_2021}, the main analytic ingredient of the proof of \cref{proposition_spgt1_estimate} is the following estimate on Gagliardo seminorms on segments:

\begin{lemma}
\label{lemma_fractional_integration}
Let \(m \in \Nset\setminus \set{0}\), let \(s, \sigma \in \intvo{0}{1}\) and let \(p \in \intvo{1}{\infty}\). 
If the set \(\Omega \subseteq \Rset^m\) is open and convex,
if \(0  < \sigma < s\) and if the mapping \(u :\Omega \to \manifold{N}\) is measurable, then 
\begin{multline}
\label{eq_ooXurang8cheic6thiebi6vu}
\iint\limits_{\Omega \times \Omega}
\brk[\bigg]{\iint\limits_{\intvc{0}{1}\times \intvc{0}{1}}\frac{d_{\manifold{N}} (u ((1-t) x + t y), u((1 - r)x + r y))^p}{\abs{t - r}^{1 + \sigma p}\abs{y - x}^{m + sp}} \dif r \dif t} \dif y \dif x\\
\le \frac{8}{(2(s - \sigma)p + 1)^2 - 1} \iint\limits_{\Omega \times \Omega}  \frac{d_{\manifold{N}} (u (y), u(x))^p}{\abs{x - y}^{m + sp}}  \dif y \dif x.
\end{multline}
\end{lemma}

It will appear in the proof of \cref{lemma_fractional_integration} that the constant in the inequality \eqref{eq_ooXurang8cheic6thiebi6vu} is sharp: equality holds in  \eqref{eq_ooXurang8cheic6thiebi6vu} if \(\Omega = \Rset^m\). The left-hand side of   \eqref{eq_ooXurang8cheic6thiebi6vu} cannot be bounded for \(\sigma = s\).

\begin{proof}[Proof of \cref{lemma_fractional_integration}]
We apply the change of variable \((z, w) = ((1-t) x + t y, (1 - r) x + r y)\) in the integral on the left-hand side of \eqref{eq_ooXurang8cheic6thiebi6vu}, 
and we obtain,
since \(z - w = (t - r) (y - x)\) and \(\det (\begin{smallmatrix} 1 - t & t \\ 1 - r & r \end{smallmatrix}) = -(t - r)\),
\begin{multline}
\label{eq_phee6OoZaichohxae3Cheduk}
\iint\limits_{\Omega \times \Omega} 
\brk[\bigg]{
\iint\limits_{\intvc{0}{1}\times \intvc{0}{1}}
\frac{d_{\manifold{N}} (u ((1-t) x + t y), u((1 - r) x + r y))^p}{\abs{t - r}^{1 + \sigma p}\abs{y - x}^{m + sp}} \dif t \dif r} \dif y \dif x\\
= 
\iint\limits_{\Omega \times \Omega} 
\iint\limits_{\Sigma_{z, w}} 
\frac
  {d_{\manifold{N}} (u (z), u(w))^p}
  {\abs{t - r}^{1  - (s - \sigma) p} \abs{z - w}^{m + sp}} 
\dif t 
\dif r 
\dif z 
\dif w
,
\end{multline}
where we have defined for each \(z, w \in \Omega\) the set
\begin{equation*}
    \Sigma_{z, w} 
  \defeq 
    \bigl\{ 
      (r, t) \in \intvc{0}{1} \times \intvc{0}{1} 
      \st \tfrac{rz - tw}{r - t} \in \Omega 
      \text{ and } 
      \tfrac{(1 - r)z - (1 -t)w}{t - r} \in \Omega 
    \bigr\}. 
\end{equation*}
We observe that, since \(s > \sigma\), we have by domain-monotonicity of the integral and by direct computation for each \(z, w \in \Omega\) 
\begin{equation}
\label{eq_JoopahViop3tiene3Le6beic}
\begin{split}
 \iint\limits_{\Sigma_{z, w}} \frac{1}{\abs{t - r}^{1 - (s - \sigma)p}} \dif t \dif r
 &\le \int_0^1 \int_0^1 \frac{1}{\abs{t - r}^{1 - (s - \sigma)p}} \dif t \dif r\\
 & = \frac{1}{(s - \sigma)p} \int_{0}^1 \abs{1 - r}^{(s - \sigma)p} + \abs{r}^{(s - \sigma)p}\dif r\\
 &
 = \frac{8}{(2(s - \sigma)p + 1)^2 - 1}
< \infty,
\end{split}
\end{equation}
and the conclusion \eqref{eq_ooXurang8cheic6thiebi6vu} follows from the identity \eqref{eq_phee6OoZaichohxae3Cheduk} and the estimate \eqref{eq_JoopahViop3tiene3Le6beic}.
\end{proof}

Our second tool is the following elementary geometric result on covering space.

\begin{lemma}
\label{lemma_small_isometry}
Let \(\pi : {\lifting{\manifold{N}}} \to \manifold{N} \) be a Riemannian covering map. If the manifold \(\manifold{N}\) has a positive injectivity radius \(\inj (\manifold{N}) > 0\), then for every \(\lifting{x}, \lifting{y} \in {\lifting{\manifold{N}}}\) such that \(d_{{\lifting{\manifold{N}}}} (\lifting{x}, \lifting{y}) \le \inj (\manifold{N})\), one has \( d_{{\lifting{\manifold{N}}}}(\lifting{x}, \lifting{y}) = d_{\manifold{N}}(\pi (\lifting{x}), \pi (\lifting{y}))\).
\end{lemma}

The proof of \cref{lemma_small_isometry} follows from the definition of injectivity radius \(\inj (\manifold{N})\) and from the lifting of geodesics (see for example \cite{Mironescu_VanSchaftingen_2021}*{lem.\ 2.1}).

\begin{proof}[Proof of \cref{proposition_spgt1_estimate}]
We first assume that the set \(\Omega \subseteq \Rset^m\) is open and convex.
By convexity of \(\Omega\) and by definition of \(Y (\Omega, \lifting{\manifold{N}})\) in \eqref{eq_definition_Y}, for almost every \(x, y \in \Omega\), we have \(\intvc{x}{y} \subset \Omega\), the restriction \(\lifting{u}\restr{\intvc{x}{y}}\) of \(\lifting{u}\) to the segment \(\intvc{x}{y}\) satisfies \(\lifting{u}\restr{\intvc{x}{y}} = \lifting{u}_{x, y}\) almost everywhere on \(\intvc{x}{y}\), \(\lifting{u}_{x, y}(x) = \lifting{u} (x)\) and \(\lifting{u}_{x, y}(y) = \lifting{u} (y)\),
with \(\lifting{u}_{x, y} \in C (\intvc{x}{y}, \lifting{\manifold{N}})\).
By the intermediate value theorem, there exists \(z \in \intvc{x}{y}\) such that 
\begin{equation*}
 \inj (\manifold{N}) \wedge d_{\lifting{\manifold{N}}} (\lifting{u} (y), \lifting{u}(x))=
 \inj (\manifold{N}) \wedge d_{\lifting{\manifold{N}}} (\lifting{u}_{x, y} (y), \lifting{u}_{x, y} (x))=
 d_{\lifting{\manifold{N}}} (\lifting{u}_{x, y} (z), \lifting{u}_{x, y}(y));
\end{equation*}
by \cref{lemma_small_isometry}, we have thus  
\begin{equation*}
 \inj (\manifold{N}) \wedge d_{\lifting{\manifold{N}}} (\lifting{u} (y), \lifting{u}(x))
 = 
 d_{\manifold{N}} (u_{x, y} (z), u_{x, y}(y)),
\end{equation*}
with \(u_{x, y} \defeq \pi \compose \lifting{u}_{x, y} \in C(\intvc{x}{y}, \manifold{N})\).
We have thus proved that
\begin{equation}
\label{eq_oo5Eakai6eiyei0eengoa1ie}
 \inj (\manifold{N}) \wedge d_{\lifting{\manifold{N}}} (\lifting{u} (y), \lifting{u}(x))
 \le \sup_{z \in \intvc{x}{y}} d_{\manifold{N}} (u_{x, y} (z), u_{x, y} (y)).
\end{equation}
Fixing \(\sigma \in \intvo{0}{1}\) such that \(1/p < \sigma < s\), we deduce from the one-dimensional fractional Morrey--Sobolev embedding (see for example \cite{Leoni_2023}*{th.\ 2.8}) and from \eqref{eq_oo5Eakai6eiyei0eengoa1ie} that
\begin{equation}
\label{eq_xukaejaewufaej2Eu3Sohzuw}
\begin{split}
  \inj (\manifold{N})^p \wedge {}& d_{\lifting{\manifold{N}}} (\lifting{u} (y), \lifting{u}(x))^p\\
  &\le \Cl{cst_pohbuij1ahgahWudohd8laim} \smashoperator{\iint_{\intvc{0}{1}\times \intvc{0}{1}}} \frac{d_{\manifold{N}} (u_{x, y} ((1-t) x + t y), u_{x, y} ((1 - r) x + r y))^p}{\abs{t - r}^{1 + \sigma p}} \dif t \dif r\\
  &= \Cr{cst_pohbuij1ahgahWudohd8laim} \smashoperator{\iint_{\intvc{0}{1}\times \intvc{0}{1}}} \frac{d_{\manifold{N}} (u ((1-t) x + t y), u ((1 - r) x + r y))^p}{\abs{t - r}^{1 + \sigma p}} \dif t \dif r,
\end{split}
\end{equation}
since \(\sigma p > 1\) and \(u_{x, y} = u\restr{\intvc{x}{y}}\) almost everywhere on \(\intvc{x}{y}\).
The conclusion follows then by integration of \eqref{eq_xukaejaewufaej2Eu3Sohzuw} thanks to \cref{lemma_fractional_integration}.
\end{proof}

\subsection{Variations on the lower exponent}
We exhibit a whole family of characterizations of the space appearing in the description of liftings of \cref{theorem_lifting_truncated_noncompact}; our analysis follows and extends the results obtained for \(m = sp\) in the context of homotopy estimates \cite{VanSchaftingen_2020}*{\S 5}.
The results of the present section are valid under the quite general assumption that the target \(\manifold{E}\) is any metric space.

\begin{proposition}[Exponent improvement]
\label{proposition_exponent_improvement}
Let \(\manifold{M}\) be a Riemannian manifold, let \(\manifold{E}\) be a metric space, let \(s \in \intvo{0}{1}\), let \(p \in \intvo{1}{\infty}\) and let \(q_0, q_1 \in \intvo{0}{\infty}\).
If \(sp >  1 \vee q_0 \vee q_1\), then there exists a constant \(C \in \intvo{0}{\infty}\) such for every measurable map \(f : \manifold{M} \to \manifold{E}\) one has
\begin{multline}
\label{eq_theeshee1Ung5ovu2buoph3I}
\smashoperator[r]{\iint_{\manifold{M} \times \manifold{M}}}
\frac{d_{\manifold{E}} (f (y), f (x))^p \wedge d_{\manifold{E}} (f (y), f (x))^{q_1}
}{d_\manifold{M} (y, x)^{m + sp}} \dif y \dif x \\
\le 
C \smashoperator{\iint_{\manifold{M} \times \manifold{M}}}
\frac{d_{\manifold{E}} (f (y), f (x))^p \wedge  d_{\manifold{E}} (f (y), f (x))^{q_0}}{d_\manifold{M} (y, x)^{m + sp}}
\dif y \dif x ,
\end{multline}
with \(m \defeq  \dim \manifold{M}\).
\end{proposition}

The main tool to prove \cref{proposition_exponent_improvement} is the following estimate which was already known in the special case \(\gamma = m\) \cite{VanSchaftingen_2020}*{prop.\ 5.5}.

\begin{proposition}
  \label{proposition_p_q_scaling_Omega}
  Let \(q_0, q_1 \in [0, + \infty)\), let \(\eta \in (0, 1)\),
  and let \(\gamma \in \intvo{0}{\infty}\).
  If \(q_1 < \gamma\) and if  
  either \(q_0 \ge 1\) or \(\gamma > 1\), then there exists a constant \(C \in \intvo{0}{\infty}\) 
  such that 
  for every \(m \in \Nset \setminus \set{0}\), for every
  convex open set \(\Omega \subseteq \Rset^m\) and for every measurable map \(f : \Omega \to \manifold{E}\),
  one has 
  \begin{equation}
  \label{eq_Adui5veiraethahki0dah4mu}
   \smashoperator[r]{\iint_{
      \substack{
        (x, y) \in \Omega \times \Omega\\
        d_{\manifold{E}} (f (y), f (x)) \ge \lambda}
    }}
    \frac
    {\brk[\big]{
      d_{\manifold{E}} (f (y), f (x))
      -
      \lambda
      }^{q_1}}
    {\abs{y - x}^{m + \gamma}}
    \dif y
    \dif x
    \le
    C 
    \lambda^{q_1 - q_0}
    \smashoperator{\iint_{
      \substack{
        (x, y) \in \Omega \times \Omega\\
        d_{\manifold{E}} (f (y), f (x)) \ge \eta \lambda}
    }}
    \frac{
      \bigl(
      d_{\manifold{E}} (f (y), f (x))
      -
      \eta \lambda
      \bigr)^{q_0}}
    {\abs{y - x}^{m + \gamma}}
    \dif y
    \dif x
    .
  \end{equation}
\end{proposition}

In the particular case \(q_1 \le q_0\), one has the pointwise estimate
\(
\brk{
      t
      -
      \lambda
      }^{q_1}
      \le \brk{
      t
      -
      \eta \lambda
      }^{q_0}/((1 - \eta)\lambda)^{q_0 - q_1}
\),
and \eqref{eq_Adui5veiraethahki0dah4mu} follows immediately by integration.

\Cref{proposition_p_q_scaling_Omega} is reminiscent of an estimate of Nguyen that appears in characterizations of first-order Sobolev spaces \cite{Nguyen_2011}*{th.\ 1 (a)}.

Our first tool to prove \cref{proposition_p_q_scaling_Omega} in general, is the following scaling inequality (when \(\gamma = m\) see \cite{VanSchaftingen_2020}*{prop.\ 5.1}).

\begin{lemma}
  \label{lemma_fract_gap_integral_growth}
  For every \(m \in \Nset \setminus \set{0}\),
  for every convex open set \(\Omega \subset \Rset^m\), for every measurable map \(f : \Omega \to \manifold{E}\),
  for every \(q \in \intvr{0}{\infty}\) and for every \(\gamma \in \Rset\),
  if \(\lambda_0 < \lambda_1\) one has
  \begin{multline}
  \label{eq_ahSh0Fuchei0eiGhoo5bohwe}
    \smashoperator[r]{\iint_{
      \substack{
        (x, y) \in \Omega \times \Omega\\
        d_{\manifold{E}} (f (y), f (x)) \ge \lambda_1}
    }}
    \frac
    {\bigl(d_{\manifold{E}} (f (y), f (x)) - \lambda_1\bigr)^q}
    {\abs{y - x}^{m + \gamma}}
    \dif y
    \dif x
    \\[-1em]
    \le 
    2^{(\gamma - 1 - (q - 1)_+)_+}
    \biggl(
    \frac
    {\lambda_1}
    {\lambda_0}
    \biggr)^{(q - 1)_+ - \gamma + 1}
    \smashoperator{
    \iint_{
      \substack{
        (x, y) \in \Omega \times \Omega\\
        d_{\manifold{E}} (f (y), f (x)) \ge \lambda_0}
    }
    }
    \frac
    {\bigl(d_{\manifold{E}} (f (y), f (x)) - \lambda_0\bigr)^q}
    {\abs{y - x}^{m + \gamma}}
    \dif y
    \dif x
    .
  \end{multline}
\end{lemma}

\begin{proof}%
[Proof of \cref{lemma_fract_gap_integral_growth}]
  Since the set \(\Omega\) is convex, for every \(x, y \in \Omega\), we have \(\frac{x + y}{2} \in \Omega\) and thus by the triangle inequality
  \begin{equation*}
   d_{\manifold{E}} (f (y), f (x)) - \lambda_1
   \le d_{\manifold{E}} \brk[\big]{f (y), f (\tfrac{x + y}{2})} - \tfrac{\lambda_1}{2}
   + d_{\manifold{E}} \brk[\big]{f (\tfrac{x + y}{2}), f (x)} - \tfrac{\lambda_1}{2},
  \end{equation*}
so that 
  \begin{equation}
  \label{eq_quailee0Reethie1eivifeeZ}
    \begin{split}
    \smashoperator[r]{
      \iint_{
        \substack{
          (x, y) \in \Omega \times \Omega\\
          d_{\manifold{E}} (f (y), f (x)) \ge \lambda_1}}
      }
      \frac
      {\bigl(d_{\manifold{E}} (f (y), f (x)) - \lambda_1\bigr)^q}
      {\abs{y - x}^{m + \gamma}}
      \dif y
      \dif x
      &\le
      2^{(q - 1)_+}
      \smashoperator{
      \iint_{
        \substack{
          (x, y) \in \Omega \times \Omega\\
          d_{\manifold{E}} (f (y), f (\frac{x + y}{2})) \ge \frac{\lambda_1}{2}
        }
      }
      }
      \frac
      {\bigl(d_{\manifold{E}} (f (y), f (\frac{x + y}{2})) - \frac{\lambda_1}{2}\bigr)^q}
      {\abs{y - x}^{m + \gamma}}
      \dif y
      \dif x\\
      &\qquad 
      +
      2^{(q - 1)_+}
      \smashoperator{
      \iint_{
        \substack{
          (x, y) \in \Omega \times \Omega\\
          d_{\manifold{E}} (f (\frac{x + y}{2}), f (x)) \ge \frac{\lambda_1}{2}
        }        
      }
      }
      \frac
      {\bigl(d_{\manifold{E}} (f (\frac{x + y}{2}), f (x)) - \frac{\lambda_1}{2}\bigr)^q}
      {\abs{y - x}^{m + \gamma}}
      \dif y
      \dif x.
    \end{split}
  \end{equation}
Therefore by symmetry between both terms in the right-hand side of \eqref{eq_quailee0Reethie1eivifeeZ} under exchange of the variables \(x\) and \(y\) in the integral, we have
  \begin{equation}
    \label{ineq_fract_double_gap_tgl}
    \begin{split}
      \smashoperator[r]{\iint_{
        \substack{
          (x, y) \in \Omega \times \Omega\\
          d_{\manifold{E}} (f (y), f (x)) \ge \lambda_1}}}
      &
      \frac
      {\bigl(d_{\manifold{E}} (f (y), f (x)) - \lambda_1\bigr)^q}
      {\abs{y - x}^{m + \gamma}}
      \dif y
      \dif x
      = 
      2^{(q - 1)_+ + 1}
      \smashoperator{
      \iint_{
        \substack{
          (x, y) \in \Omega \times \Omega\\
          d_{\manifold{E}} (f (\frac{x + y}{2}), f (x)) \ge \frac{\lambda_1}{2}
        }
      }
      }
      \frac
      {\bigl(d_{\manifold{E}} (f (\frac{x + y}{2}), f (x)) - \frac{\lambda_1}{2}\bigr)^q}
      {\abs{y - x}^{m + \gamma}}
      \dif y
      \dif x
       .
    \end{split}
  \end{equation}
By the change of variable \(y = 2 z - x\), we have \(\abs{y - x} = 2 \abs{z - x}\) and thus
  \begin{equation}
    \label{ineq_fract_double_gap_cov}
    \begin{split}
      \smashoperator[r]{\iint_{
        \substack{
          (x, y) \in \Omega \times \Omega\\
          d_{\manifold{E}} (f (\frac{x + y}{2}), f (x)) \ge \frac{\lambda_1}{2}
        }
      }}
      \frac
      {
        \bigl(
        d_{\manifold{E}} (f (\frac{x + y}{2}), f (x)) 
        -
        \frac{\lambda_1}{2}
        \bigr)^q
      }
      {\abs{y - x}^{m + \gamma}}
      \dif y
      \dif x
      &
      =
      \frac
      {1}
      {2^\gamma}
      \int_{\Omega}
      \brk[\bigg]{
      \int_{\Sigma_x}
      \frac
      {
        \bigl(
        d_{\manifold{E}} (f (z), f (x))
        - 
        \frac{\lambda_1}{2}
        \bigr)^q
      }
      {\abs{z - x}^{m + \gamma}}  
      \dif z
      }
      \dif x
      \\
      &
      \le
      \frac
      {1}
      {2^\gamma}
      \smashoperator{
      \iint_{
        \substack{
          (x, y) \in \Omega \times \Omega\\
          d_{\manifold{E}} (f (y), f (x)) \ge \frac{\lambda_1}{2}
        }
      }}
      \frac
      {
        \bigl(
        d_{\manifold{E}} (f (y), f (x)) - \frac{\lambda_1}{2}
        \bigr)^q
      }
      {\abs{y - x}^{m + \gamma}}  
      \dif y
      \dif x
       ,
    \end{split}
  \end{equation}
  where for every \(x \in \Omega\), the set \(\Sigma_x\) is defined as
  \begin{equation*}
    \Sigma_x 
    \defeq 
    \set[\big]{
    z \in \Omega 
    \st 
    2 z - x \in \Omega 
    \text{ and }
    d_{\manifold{E}} (f (z), f (x)) \ge \tfrac{\lambda_1}{2}
    }
    .
  \end{equation*}
  By \eqref{ineq_fract_double_gap_tgl} and \eqref{ineq_fract_double_gap_cov}, we deduce that for every \(\lambda_1 > 0\),
  \begin{multline}
    \label{ineq_fract_double_gap}
    \smashoperator[r]{\iint_{
      \substack{
        (x, y) \in \Omega \times \Omega\\
        d_{\manifold{E}} (f (y), f (x)) \ge \lambda_1
      }        
    }}
    \frac
    {\bigl(d_{\manifold{E}} (f (y), f (x)) - \lambda_1\bigr)^q}
    {\abs{y - x}^{m + \gamma}}
    \dif y
    \dif x
    \\[-2em]
    \le
    2^{(q - 1)_+ - (\gamma - 1)}
    \smashoperator{
    \iint_{
      \substack{
        (x, y) \in \Omega \times \Omega\\
        d_{\manifold{E}} (f (y), f (x)) \ge \frac{\lambda_1}{2}
      }        
    }}
    \frac
    {
      \bigl(
      d_{\manifold{E}} (f (y), f (x)) - \frac{\lambda_1}{2}
      \bigr)^q
    }
    {\abs{y - x}^{m + \gamma}}
    \dif y
    \dif x
    .    
  \end{multline}
Iterating the estimate \eqref{ineq_fract_double_gap}, 
we deduce that for every nonnegative integer \(\ell \in \Nset\),
  \begin{multline}
  \label{eq_ThooyieGahyeekoo6Oomeala}
    \smashoperator[r]{\iint_{
      \substack{
        (x, y) \in \Omega \times \Omega\\
        d_{\manifold{E}} (f (y), f (x)) \ge \lambda_1
      }        
    }}
    \frac
    {\bigl(d_{\manifold{E}} (f (y), f (x)) - \lambda_1\bigr)^q}
    {\abs{y - x}^{m + \gamma}}
    \dif y
    \dif x\\[-2em]
    \le
    2^{\ell ((q - 1)_+ - (\gamma - 1))}
    \smashoperator{
    \iint_{
      \substack{
        (x, y) \in \Omega \times \Omega\\
        d_{\manifold{E}} (f (y), f (x)) \ge \frac{\lambda_1}{2^\ell}
      }        
    }}
    \frac
    {(d_{\manifold{E}} (f (y), f (x)) - \frac{\lambda_1}{2^\ell} )^q}
    {\abs{y - x}^{m + \gamma}}
    \dif y
    \dif x
    .    
  \end{multline}
  If \(\lambda_0 \in (0, \lambda_1)\), we let \(\ell \in \Nset\) in \eqref{eq_ThooyieGahyeekoo6Oomeala}
  be defined by the condition \(2^{-(\ell + 1)} \lambda_1 \le \lambda_0 < 2^{-\ell} \lambda_1\) and
  we conclude that \eqref{eq_ahSh0Fuchei0eiGhoo5bohwe} holds.
\end{proof}

Our second proof for the proof of \cref{proposition_p_q_scaling_Omega} is the next elementary integral inequality \cite{VanSchaftingen_2020}*{lem.\ 5.6}.

\begin{lemma}
  [Integral estimate of truncated powers]
  \label{lemma_ineq_p_q_scaling}
  For every \(q_0, q_1 \in \intvr{0}{\infty}\) and every \(\eta \in \intvo{0}{1}\), 
  there exists a constant \(C > 0\) such that for every \(t \in \intvr{1}{\infty}\),
  \begin{equation*}
    (t - 1)^{q_1}
    \le 
    C
    \int_{\eta}^t
    \frac
    {(t - r)^{q_0}}
    {r^{1 + q_0 - q_1}}
    \dif r
     .
  \end{equation*}
\end{lemma}

\begin{proof}%
  [Proof of \cref{proposition_p_q_scaling_Omega}]
  \resetconstant
  Applying \cref{lemma_ineq_p_q_scaling} with \(t \defeq d_{\manifold{E}} (f (y), f (x))/\lambda\) at each \(x, y \in \Omega\), integrating the result and interchanging the integrals, we have
  \begin{multline}
    \label{ineq_pqsO_1}
    \smashoperator[r]{\iint_{
      \substack{
        (x, y) \in \Omega \times \Omega\\
        d_{\manifold{E}} (f (y), f (x)) \ge \lambda}
    }}
    \frac
    {\brk[\big]{
      d_{\manifold{E}} \brk{f (y), f (x)}
      -
      \lambda 
      }^{q_1}}
    {\abs{y - x}^{m + \gamma}}
    \dif y
    \dif x
    \\[-2em]
    \le
    \C
    \lambda^{q_1 - q_0}
    \int_{\eta}^{\infty}
    \smashoperator{
    \iint_{
      \substack{
        (x, y) \in \Omega \times \Omega\\
        d_{\manifold{E}} (f (y), f (x)) \ge r\lambda}
    }}
    \frac
    {\brk[\big]{
      d_{\manifold{E}} (f (y), f (x))
      -
      r \lambda
      }^{q_0}}
    {r^{1 + q_0 - q_1} \, \abs{y - x}^{m + \gamma}}
    \dif y
    \dif x 
    \dif r
    .
  \end{multline}
  Since the set \(\Omega \subseteq \Rset^m\) is convex, by \cref{lemma_fract_gap_integral_growth}, we have for every \(r \in (\eta, \infty)\), 
  \begin{multline}
    \label{ineq_pqsO_2}
    \smashoperator{\iint_{
      \substack{
        (x, y) \in \Omega \times \Omega\\
        d_{\manifold{E}} (f (y), f (x)) \ge r \lambda}
    }}
    \frac
    {\brk[big]{
      d_{\manifold{E}} (f (y), f (x))
      -
      r \lambda
      }^{q_0}}
    {\abs{y - x}^{m + \gamma}}
    \dif y
    \dif x
    \\[-2em]
    \le
    \C
    \,
    \frac{1}{r^{(\gamma - 1) - (q_0 - 1)_+}}
    \smashoperator{
    \iint_{
      \substack{
        (x, y) \in \Omega \times \Omega\\
        d_{\manifold{E}} (f (y), f (x)) \ge \eta \lambda }
    }}
    \frac
    {\bigl(
      d_{\manifold{E}} (f (y), f (x))
      -
      \eta \lambda
      \bigr)^{q_0}}
    {\abs{y - x}^{m + \gamma}}
    \dif y
    \dif x 
    .
  \end{multline}
Combining the estimates \eqref{ineq_pqsO_1} and \eqref{ineq_pqsO_2}, 
  we deduce that 
  \begin{multline}
  \label{eq_ohz7gieBoo4Sheenieseich1}
   \smashoperator[r]{\iint_{
      \substack{
        (x, y) \in \Omega \times \Omega\\
        d_{\manifold{E}} (f (y), f (x)) \ge \lambda}
    }}
    \frac
    {\bigl(
      d_{\manifold{E}} (f (y), f (x))
      -
      \lambda 
      \bigr)^{q_1}}
    {\abs{y - x}^{m + \gamma}}
    \dif y
    \dif x
    \\
    \le 
    \C 
    \lambda^{q_1 - q_0}
    \,
    \int_{\eta}^{\infty}
    \frac
      {1}
      {r^{\gamma + 1 - (1 - q_0)_+  - q_1}}
    \dif r
    \smashoperator{
    \iint_{
      \substack{
        (x, y) \in \Omega \times \Omega\\
        d_{\manifold{E}} (f (y), f (x)) \ge \eta \lambda}
    }
    }
    \frac
    {\bigl(
      d_{\manifold{E}} (f (y), f (x))
      - \eta \lambda
      \bigr)^{q_1}}
    {\abs{y - x}^{m + \gamma}}
    \dif y
    \dif x ,
  \end{multline}
  since \(q_0 - 1 - (q_0 - 1)_+ = - (1 - q_0)_+ \).
  If \(
  q_1
  <
  \gamma - (1 - q_0)_+
  \), 
  then
  \begin{equation*}
      \int_{\eta}^{\infty}
      \frac{1}{r^{\gamma + 1 - (1 - q_0)_+  - q_1}}
      \dif r
      =
      \frac{1}{(\gamma - (1 - q_0)_+ - q_1) \, \eta^{\gamma - (1 - q_0)_+ - q_1} } < \infty,
  \end{equation*}
  and 
  the estimate \eqref{eq_Adui5veiraethahki0dah4mu} follows from \eqref{eq_ohz7gieBoo4Sheenieseich1}.
 
  If \(q_0 \ge 1\), then we have proved the estimate for \(q_1 < \gamma\).
  Otherwise, \(q_0 < 1\), and we have proved the estimate 
  \eqref{eq_Adui5veiraethahki0dah4mu} for \(q_1 < q_0 + (\gamma - 1)\). 
  Iterating finitely many times the estimate we reach the interval \(q_1 \in [0, 1]\) and the conclusion \eqref{eq_Adui5veiraethahki0dah4mu} then follows for \(q_1 < \gamma\).
\end{proof}

We are now in position to prove \cref{proposition_exponent_improvement}.

\begin{proof}%
  [Proof of \cref{proposition_exponent_improvement}]%
  \resetconstant  
Since the case \(q_1 \le q_0\) follows from the fact that for every \(t \in \intvo{0}{\infty}\), we have \(t^p \wedge t^{q_0} \le t^p \wedge t^{q_1}\),
we consider the case \(q_1 > q_0\).
Letting \(\Omega \subseteq \Rset^m\) be a convex open set and the mapping \(f : \Omega \to \manifold{E}\) be measurable, and defining the set
\begin{equation*}
A \defeq\set{(x,y) \in \Omega \times \Omega \st d_{\manifold{E}} \brk{f (y), f (x)} \ge 1},
\end{equation*}
we decompose, since \(q_1 < sp < p\), the integral in the left-hand side of \eqref{eq_theeshee1Ung5ovu2buoph3I} as 
\begin{multline}
  \label{eq_Boekeo8gu6auceef9mu8eph4}
  \smashoperator[r]{
    \iint_{\Omega \times \Omega}
  } 
  \frac
    {d_{\manifold{E}} \brk{f (y), f (x)}^p \wedge d_{\manifold{E}} \brk{f (y), f (x)}^{q_1}}
    {\abs{y - x}^{m + sp}}
    \dif y 
    \dif x
\\[-1em]
=  
  \smashoperator{\iint_{\Omega \times \Omega \setminus A}} \frac
    {d_{\manifold{E}} \brk{f (y), f (x)}^p}
    {\abs{y - x}^{m + sp}}
    \dif y 
    \dif x
  + 
  \smashoperator{\iint_{A}} 
  \frac
    {d_{\manifold{E}} \brk{f (y), f (x)}^{q_1}}
    {\abs{y - x}^{m + sp}}
    \dif y 
    \dif x.
\end{multline}
On the one hand, we have immediately
\begin{equation}
\label{eq_euSheekeeQue1aiFu3thahgh}
 \smashoperator[r]{\iint_{\Omega \times \Omega \setminus A}} \frac{d_{\manifold{E}} \brk{f (y), f (x)}^p}{\abs{y - x}^{m + sp}}\dif y \dif x \le \smashoperator{\iint_{\Omega \times \Omega}} \frac{d_{\manifold{E}} \brk{f (y), f (x)}^p \wedge d_{\manifold{E}} \brk{f (y), f (x)}^{q_0}}{\abs{y - x}^{m + sp}}\dif y \dif x.
\end{equation}
On the other hand, by \cref{proposition_p_q_scaling_Omega}, since \(sp > 1\) and \(q_1 < sp\), we have
\begin{equation}
\label{eq_eipiecha1aloh6ICh1ohheeZ}
\begin{split}
\iint\limits_{A} \frac{d_{\manifold{E}} \brk{f (y), f (x)}^{q_1}}{\abs{y - x}^{m + sp}}\dif y \dif x
&\le 2^{q_1} \smashoperator{\iint_{\Omega\times \Omega}} \frac{(d_{\manifold{E}} \brk[\big]{f (y), f (x)} - \frac{1}{2})^{q_1}_+}{\abs{y - x}^{m + sp}}\dif y \dif x\\
&\le \C \smashoperator{\iint_{\substack{(x, y) \in \Omega\times \Omega\\ d_{\manifold{E}} \brk{f (y), f(x)} \ge \frac{1}{3}}}} \frac{(d_{\manifold{E}} \brk[\big]{f (y), f (x)} - \frac{1}{3})^{q_0}_+}{\abs{y - x}^{m + sp}}\dif y \dif x,
\end{split}
\end{equation}
and it follows thus from \eqref{eq_Boekeo8gu6auceef9mu8eph4}, \eqref{eq_euSheekeeQue1aiFu3thahgh} and \eqref{eq_eipiecha1aloh6ICh1ohheeZ} that 
\begin{multline}
\label{eq_rupain9gocau5Oer2zoch7if}
\smashoperator{\iint_{\Omega \times \Omega}} \frac{d_{\manifold{E}} \brk{f (y), f (x)}^p \wedge d_{\manifold{E}} \brk{f (y), f (x)}^{q_1}}{\abs{y - x}^{m + sp}}\dif y \dif x\\[-1em]
\le \C  \smashoperator{
\iint_{\Omega \times \Omega}} \frac{d_{\manifold{E}} \brk{f (y), f (x)}^p \wedge d_{\manifold{E}} \brk[\big]{f (y), f (x)}^{q_0}}{\abs{y - x}^{m + sp}} \dif y \dif x, 
\end{multline}
since \(q_0 < p\).
The announced conclusion \eqref{eq_theeshee1Ung5ovu2buoph3I} follows then from \eqref{eq_rupain9gocau5Oer2zoch7if} and the covering of \cref{lemma_manifold_diagonal_covering}.
\end{proof}

Thanks to \cref{proposition_exponent_improvement}, we can now prove \cref{proposition_equivalent_norms}.

\begin{proof}[Proof of \cref{proposition_equivalent_norms}]
This follows from \cref{proposition_exponent_improvement},
with \(\manifold{E} = \lifting{\manifold{N}}\).
\end{proof}

\subsection{Compactness in the space of liftings}
Given the estimate on the lifting of \cref{proposition_spgt1_estimate} on a set which is dense in the fractional Sobolev space \(\homog{W}^{s, p}(\manifold{M}, \manifold{N})\) \cite{Brezis_Mironescu_2015}, a classical approach to prove the existence of a lifting would be to consider the limit of the liftings of an approximating sequence.
In order to perform this, we need a compactness result on sets for which the left-hand side of \eqref{eq_hao4yeiKieju4Riel5AxaeTh} is uniformly bounded.

\begin{proposition}
\label{prop_compactness}
Let \(\manifold{M}\) be a Riemannian manifold with finite volume, let \(\manifold{E}\) be a metric space, let \(0 \le q \le p\) and let \(0 < s < 1\).
Assume that every bounded subset of \(\manifold{E}\) is totally bounded.
If \(\mathcal{S}\) is a set of measurable functions from \(\manifold{M}\) to \(\manifold{E}\) such that
\begin{equation}
\label{eq_eeQuaoPheej4xa9yich4Xuy4}
 \sup_{f \in \mathcal{S}}
 \smashoperator[r]{\iint_{\manifold{M} \times \manifold{M}}}
\frac{d_{\manifold{E}} (f (y), f (x))^p \wedge d_{\manifold{E}} (f (y), f (x))^q
}{d_\manifold{M} (y, x)^{m + sp}} \dif y \dif x < \infty,
\end{equation}
with \(m \defeq  \dim \manifold{M}\),
and such that
\begin{equation}
\label{eq_sohCh0Oosoo4xaexeK2Aecee}
\inf_{f, g \in \mathcal{S}} 
\int_{\mathcal{M}}\frac{1}{1 + d_{\manifold{E}} (g, f)}
 > 0,
\end{equation}
then the set \(\mathcal{S}\) is totally bounded for the distance
\begin{equation}
\label{eq_oht5queo0koiShijefohnguT}
 d_{\mu} (f, g) \defeq \int_{\manifold{M}} \frac{d_{\manifold{E}} (f, g)}{1 + d_{\manifold{E}} (f, g)}.
\end{equation}
\end{proposition}

Although the case \(p = q = 0\) is covered in \cref{prop_compactness},
it is not particulary interesting since in view of \cref{lemma_potential_char}, the mapping \(f\) should be constant on every connected component of \(\manifold{M}\). 

If the metric space \(\mathcal{E}\) is complete, the assumption that any of its subsets is totally bounded is equivalent to \(\mathcal{E}\) having the Bolzano--Weierstraß property or to \(\mathcal{E}\) being a proper space.

The convergence with respect to the distance \(d_\mu\) defined in \eqref{eq_oht5queo0koiShijefohnguT} is the convergence in measure. 
We first remark that this distance can be controlled on finite-measure sets by a quantity reminiscent of the integrand in \eqref{eq_eeQuaoPheej4xa9yich4Xuy4}.

\begin{lemma}
\label{lemma_measure_holder}
Let \(\mu\) be a measure on \(\Omega\) and let \(\manifold{E}\) be a metric space.
If \(0 \le q \le p\) and if the mappings \(f, g : \Omega \to \manifold{E}\) are measurable, then 
\begin{equation}
\label{eq_opai2mu8vo4ohp2phaThao0j}
 \int_{\manifold{M}} \frac{d_{\manifold{E}} (f, g)}{1 + d_{\manifold{E}} (f, g)}\dif \mu
 \le \mu (\Omega)^{(1-1/p)_+} \brk[\bigg]{\int_{\Omega} d_{\manifold{E}} (f, g)^p \wedge  d_{\manifold{E}} (f, g)^q\dif \mu}^{\frac{1}{p} \wedge 1}.
\end{equation}
\end{lemma}
\begin{proof}
When \(0 \le p \le 1\), \eqref{eq_opai2mu8vo4ohp2phaThao0j} follows from the fact that for every \(t \in \intvr{0}{\infty}\) one has 
\begin{equation*}
  t/(1 + t)\le t \wedge 1 \le t^p \wedge t^q, 
\end{equation*}
whereas when \(p > 1\)  \eqref{eq_opai2mu8vo4ohp2phaThao0j} follows from the fact that 
\begin{equation*}
t/(1 + t)\le t \wedge 1 \le t \wedge t^{q/p} 
\end{equation*}
 and Hölder's inequality.
\end{proof}

The proof of \cref{prop_compactness} will rely on the following inequality.
\begin{lemma}
\label{lemma_p_q_a_b}
If \(p, q \in \intvr{0}{\infty}\), then for every \(\ell \in \Nset\) and \(a_1, \dotsc, a_\ell \in \intvr{0}{\infty}\), we have 
\begin{equation*}
 \brk[\bigg]{\sum_{i = 1}^\ell a_i}^p \wedge \brk[\bigg]{\sum_{i = 1}^\ell a_i}^q
 \le \max_{i \in \set{1, \dotsc, \ell}}
 (\ell a_i)^p \wedge (\ell a_i)^q.
\end{equation*}
\end{lemma}
\begin{proof}
Without loss of generality, we can assume that for each \(i \in \set{1, \dotsc, \ell}\) one has \(a_i \le a_1\), so that \(\sum_{i = 1}^\ell a_i \le \ell a_1\) and 
\begin{equation*}
\brk[\bigg]{\sum_{i = 1}^\ell a_i}^p 
\wedge 
\brk[\bigg]{\sum_{i = 1}^\ell a_i}^q
 \le (\ell a_1)^p \wedge (\ell a_1)^q = \max_{i \in \set{1, \dotsc, \ell}}
  (\ell a_i)^p \wedge (\ell a_i)^q.\qedhere
\end{equation*}
\end{proof}

\begin{proof}[Proof of \cref{prop_compactness}]
\resetconstant
By the finiteness of the volume and a local charts argument, we assume that \(\Omega = \Qset^m \defeq [0, 1]^m\).
For every \(k \in \Nset \setminus \{0\}\), we subdivide the cube \(\Qset^m\) in a set \(\mathcal{Q}^k\) of \(k^m\) cubes of edge-length \(1/k\). 
Given \(f \in \mathcal{S}\), we define the map \(f_k : \Qset^m \to \manifold{E} \) in such a way that \(f_k\) is constant on each cube \(Q \in \mathcal{Q}^k\) and for every \(x \in Q\in \mathcal{Q}^k\)
\begin{multline}
\label{eq_ap9Ei0eeKusoomiem4WiSave}
  \int_{Q} d_{\manifold{E}} (f (x), f_k (x))^p \wedge d_{\manifold{E}} (f (x), f_k (x))^q\dif x\\
   \le k^m \smashoperator{\iint_{Q \times Q}} d_{\manifold{E}} (f (y), f (x))^p \wedge d_{\manifold{E}} (f (y), f (x))^q\dif y \dif x.
\end{multline}
It follows immediately from \eqref{eq_ap9Ei0eeKusoomiem4WiSave} 
that 
\begin{multline}
\label{eq_aehidopiev7Ahcoo4vai5ki7}
  \int_{\Qset^m} d_{\manifold{E}} (f (x), f_k (x))^p \wedge d_{\manifold{E}} (f (x), f_k (x))^q\dif x\\
   \le \frac{m^\frac{m + sp}{2}}{k^{sp}} \smashoperator{\iint_{\Qset^m \times \Qset^m}} \frac{d_{\manifold{E}} (f (y), f (x))^p \wedge d_{\manifold{E}} (f (y), f (x))^q}{\abs{y - x}^{m + sp}}\dif y \dif x.
\end{multline}
and thus by \cref{lemma_measure_holder} that
\begin{equation}
\label{eq_jis0Ien9xeshai3ohrohd5ei}
\begin{split}
  \int_{\Qset^m}
  \frac{d_{\manifold{E}} (f (x), f_k (x))}{1 + d_{\manifold{E}} (f (x), f_k (x))} & \dif x \\[-1em]
  &\le 
  \brk[\bigg]{\frac{m^\frac{m + sp}{2}}{k^{sp}} \smashoperator{\iint_{\Qset^m \times \Qset^m}} \frac{d_{\manifold{E}} (f (x), f (y))^p \wedge d_{\manifold{E}} (f (x), f (y))^q}{\abs{y - x}^{m + sp}} \dif y \dif x}^{1 \wedge \frac{1}{p}}.
\end{split}
\end{equation}
The assumption \eqref{eq_eeQuaoPheej4xa9yich4Xuy4} and the estimate \eqref{eq_jis0Ien9xeshai3ohrohd5ei} imply that for \(k \in \Nset \setminus \set{0}\) large enough the set
\( \mathcal{S}
\)
is contained in an arbitrarily small neighbourhood of the set of the mappings \(f_k\). Since for every \(k \in \Nset\),  the set of mappings taking constant value on each \(Q \in \mathcal{Q}^k\) is bi-Lipschitz equivalent to the manifold \(\manifold{E}^{k^m}\) and since bounded subsets of \(\mathcal{E}\) are totally bounded,
it remains to prove that for any \(k \in \Nset\), the mappings \(f_k\) are contained in a bounded set.

For every \(\lambda \in \intvo{0}{\infty}\) and \(f,g : \manifold{M} \to \manifold{E}\), we have 
\begin{equation*}
 \int_{\Qset}\frac{1}{1 + d_{\manifold{E}} (g(x), f(x))}\dif x 
 \le \abs{\set{x \in \Qset^m \st d_{\manifold{E}} (g (x), f (x))\le \lambda }} + \frac{1}{1 + \lambda},
\end{equation*}
and therefore by our assumption \eqref{eq_sohCh0Oosoo4xaexeK2Aecee}, there exists \(\lambda \in \intvo{0}{\infty}\) and \(\eta \in \intvo{0}{\infty}\) such that for every \(f, g \in \mathcal{S}\),
\begin{equation}
\label{eq_tei7ahShoh3iceun4ja8Chac}
 \abs{\set{x \in \Qset^m \st d_{\manifold{E}} (g (x), f (x))\le \lambda }}\ge \eta.
\end{equation}
For every \(x, y \in \Qset^m\), we have by the triangle inequality,
\begin{equation*}
\begin{split}
 d_{\manifold{E}} (g_k (x), f_k (x))
 \le {}&{} d_{\manifold{E}} (g_k (x), g (x)) + d_{\manifold{E}} (g (x), g (y))\\
 &+ d_{\manifold{E}} (g (y), f (y))
 + d_{\manifold{E}} (f (y), f (x)) 
 + d_{\manifold{E}} (f (x), f_k (x)),
 \end{split}
\end{equation*}
and thus by \cref{lemma_p_q_a_b}
\begin{equation}
\label{eq_RahshohTh3Sha5Aiv0ooraeK}
\begin{split}
\int_{\Qset^m }  d_{\manifold{E}} (g_k (x), f_k (x))^p {}&{}\wedge d_{\manifold{E}} (g_k (x), f_k (x))^q\dif x \\
&\le 
\int_{\Qset^m} \brk{5 d_{\manifold{E}} (g_k (x), g (x))}^p \wedge \brk{5d_{\manifold{E}} (g_k (x), g (x))}^q \dif x\\
&\qquad + \int_{\Qset^m} \fint_{A} \brk{5d_{\manifold{E}} (g (y), g (x))}^p \wedge \brk{5 d_{\manifold{E}} (g (y), g (x))}^q \dif y \dif x\\
&\qquad + 
\fint_{A} \brk{5 d_{\manifold{E}} (g (y), f (y))}^p \wedge \brk{5 d_{\manifold{E}} (g (y), f (y))}^q \dif y \\
&\qquad + \int_{\Qset^m} \fint_{A} \brk{5 d_{\manifold{E}} (f (y), f (x))}^p \wedge \brk{5 d_{\manifold{E}} (f (y), f (x))}^q \dif y \dif x\\
&\qquad + \int_{\Qset^m} \brk{5 d_{\manifold{E}} (f_k (x), f (x))}^p \wedge \brk{5 d_{\manifold{E}} (f_k (x), f (x))}^q \dif x,
\end{split}
\end{equation}
with
\begin{equation}
\label{eq_xanauZoebaeciekeejei9Pee}
 A \defeq \set[\big]{x \in \Qset^m \st d_{\manifold{E}} (g (x), f (x))\le \lambda }.
\end{equation}
Inserting 
\eqref{eq_eeQuaoPheej4xa9yich4Xuy4},  \eqref{eq_aehidopiev7Ahcoo4vai5ki7} and  \eqref{eq_tei7ahShoh3iceun4ja8Chac}
in \eqref{eq_RahshohTh3Sha5Aiv0ooraeK} combined with \eqref{eq_xanauZoebaeciekeejei9Pee} and with \cref{lemma_measure_holder}, 
we get 
\begin{equation*}
\int_{\Qset^m } \frac{d_{\manifold{E}} (g_k (x), f_k (x))}{1 + d_{\manifold{E}} (g_k (x), f_k (x))}\dif x 
\le \C \brk[\Big]{\frac{1}{k^{sp}} + \frac{1}{\eta}  + \lambda^p \wedge \lambda^q}^{\frac{1}{p} \wedge 1},
\end{equation*}
and the announced boundedness follows.
\end{proof}

\subsection{Existence of a lifting}
The last tool we will use to prove \cref{theorem_lifting_truncated_noncompact} is the existence of local charts that cover the product. 

\begin{lemma}
\label{lemma_manifold_diagonal_covering}
If \(\manifold{M}\) is a connected compact manifold with \(m \defeq \dim \manifold{M}\), then there exists open sets \(V_1, \dotsc, V_\ell \subset \manifold{M}\) such that for each \(i \in \{1, \dotsc, \ell\}\), the set \(\Bar{V}_i\) is diffeomorphic to the closed ball \(\Bar{\mathbb{B}}_1 \subset \Rset^m\) and such that 
\begin{equation}
    \manifold{M} \times \manifold{M} 
  \subseteq 
    \bigcup_{i = 1}^\ell 
      V_i \times V_i .
\end{equation}
\end{lemma}
\begin{proof}
Since the manifold \(\manifold{M}\) is connected, every doubleton \(\set{x, y} \subset \manifold{M}\)
is contained in an open set \(V \subseteq \manifold{M}\) such that \(\Bar{V}\) is diffeomorphic to the closed ball \(\Bar{B}_1 \subset \Rset^m\).
In particular \((x, y) \in V \times V\). 
We conclude by compactness of \(\manifold{M} \times \manifold{M}\).
\end{proof}

The proof of \cref{theorem_lifting_truncated_noncompact} will rely on the notion of normal covering. 
A covering map \(\pi : \lifting{\manifold{N}} \to \manifold{N}\) is \emph{normal} (or regular) whenever for every \(\lifting{y} \in \lifting{\manifold{N}}\)
we have 
\begin{equation}
\label{eq_oYa4oaRaiph3Uel9xee2uiM2}
 \pi^{-1} \brk{\set{\pi(\lifting{y})}}
 = \set{\tau(\lifting{y}) \st \tau \in \operatorname{Aut} (\pi)}
\end{equation}
where the group of deck transformations (or group of covering transformations or Galois group) of the covering \(\pi\) is the group
\begin{equation}
\label{eq_definition_automorphism_group}
  \Aut (\pi) 
  = \bigl\{ \tau : {\lifting{\manifold{N}}} \to {\lifting{\manifold{N}}} \st \tau \text{ is a homeomorphism and } \pi \compose \tau = \pi \bigr\}
\end{equation}
endowed with the composition operation \citelist{\cite{Hatcher_2002}*{\S 1.3}\cite{Spanier_1966}*{ch. 2 \S 6}}.
When \(\pi\) is a Riemannian covering, \(\pi\) is a local isometry and any \(\tau \in \Aut (\pi)\) is a global isometry of \(\lifting{\manifold{N}}\).

If \(\pi : \lifting{\manifold{N}} \to \manifold{N}\) is a universal covering, that is when \(\pi\) is surjective and \(\lifting{\manifold{N}}\) is simply-connected, then \(\pi\) is normal.

We proceed to the proof of existence of a lifting.

\begin{proof}[Proof of \cref{theorem_lifting_truncated_noncompact}]
\resetconstant
We first assume that \(\pi : \lifting{\manifold{N}} \to \manifold{N}\) is a normal covering of \(\manifold{N}\). 

Given a map \(u \in \homog{W}^{s, p}(\manifold{M},\manifold{N})\), by Brezis \& Mironescu's approximation result for fractional Sobolev mappings \cite{Brezis_Mironescu_2015}, there exists a sequence \((u_j)_{j \in \Nset}\) in \(\mathcal{R}^0_{m - 2}
(\manifold{M}, \manifold{N}) \cap \homog{W}^{s, p} (\manifold{M}, \manifold{N})\) that converges strongly to \(u\) in \( \homog{W}^{s, p}(\manifold{M},\manifold{N})\), where \(\mathcal{R}^0_{k}(\manifold{M}, \manifold{N})\) denotes for \(k\in \set{0, \dotsc, m - 1}\) the set of maps from a manifold \(\manifold{M}\) to a manifold \(\manifold{N}\) that are continuous outside a finite union of \(k\)--dimensional submanifolds with boundary of \(\manifold{M}\).

For every \(j \in \Nset\), the mapping \(u_j\) is continuous outside an \((m-2)\)--dimensional subset \(\Sigma_j \subset \manifold{M}\).
Since the manifold \(\manifold{M}\) is simply-connected, the set \(\manifold{M} \setminus \Sigma_j\) is also simply-connected and there exists \(\lifting{u}_j \in C (\manifold{M} \setminus \Sigma_j, \lifting{\manifold{N}})\) such that \(\pi \compose \lifting{u}_j=  \smash{u_j\restr{\manifold{M} \setminus \Sigma_j}}\), where \(\smash{u_j\restr{\manifold{M} \setminus \Sigma_j}}\) is the restriction of \(u_j\) to the set \(\manifold{M} \setminus \Sigma_j\). In particular we have \(\lifting{u}_j \in \mathcal{R}^0_{m - 2} (\manifold{M}, \lifting{\manifold{N}})\).
Since for every convex open set \(\Omega \subseteq \Rset^m\) we have \(\mathcal{R}^0_{m - 2} (\Omega, \lifting{\manifold{N}}) \subseteq Y (\Omega, \lifting{\manifold{N}})\) and since \(sp > 1\), by the a priori estimate on the lifting (\cref{proposition_spgt1_estimate}), by the diagonal covering (\cref{lemma_manifold_diagonal_covering}) and by \cref{proposition_exponent_improvement},
we have 
\begin{equation}
\label{eq_bia6Eitah8aeseThudohnaib}
 \sup_{j \in \Nset}
 \smashoperator[r]{\iint_{\manifold{M} \times \manifold{M}}} \frac{d_{\lifting{\manifold{N}}} (\lifting{u}_j (x), \lifting{u}_j (y))^p \wedge d_{\lifting{\manifold{N}}} (\lifting{u}_j (x), \lifting{u}_j (y))}{d_{\manifold{M}} (y, x)^{m + sp}} \dif y \dif x < \infty.
\end{equation}
By \eqref{eq_bia6Eitah8aeseThudohnaib}, there exists thus \(\lambda \in \intvo{0}{\infty}\) such that for every \(j \in \Nset\),
there exists \(x_j \in \manifold{M}\) for which if we set 
\begin{equation}
\label{eq_AiSoolejeif2Aiy5phohRaet}
 A_j \defeq \set{x \in \manifold{M} \st d_{\lifting{\manifold{N}}} (\lifting{u}_j (x), \lifting{u}_j (x_j)) \le \lambda }
\end{equation}
we have then
\begin{equation}
\label{eq_cu0ienae1HaphoRishodicie}
 \abs{A_j} \ge \tfrac{2}{3} \abs{\manifold{M}}.
\end{equation}

Since the manifold \(\manifold{N}\) is compact and since the covering \(\pi\) is normal, there exists an open bounded set \(\lifting{W} \subseteq \lifting{\manifold{N}}\) such that \(\pi (\lifting{W}) = \manifold{N}\) and
\begin{equation}
\label{eq_thoeL0ohS7lasoo3we4xov7e}
 \lifting{\manifold{N}} = \bigcup_{\tau \in \operatorname{Aut} (\pi) } \tau^{-1} (\lifting{W}),
\end{equation}
in view of \eqref{eq_oYa4oaRaiph3Uel9xee2uiM2}.
By \eqref{eq_thoeL0ohS7lasoo3we4xov7e}, for every \(j \in \Nset\), exists thus \(\tau_j \in \operatorname{Aut} (\pi) \) such that \(\tau_j (\lifting{u}_j (x_j)) \in \lifting{W}\). Without loss of generality we assume that for each \(j \in \Nset\) we have \(\tau_j= \operatorname{id}_{\lifting{\manifold{N}}}\), so that \(\lifting{u}_j (x_j) \in \lifting{W}\).

We deduce from \eqref{eq_AiSoolejeif2Aiy5phohRaet} that for every \(i, j \in \Nset\) and every \(x \in A_{i,j} \defeq A_i \cap A_j\)
\begin{equation}
\label{eq_BieNgie0thohkie1uG4Axei4}
\begin{split}
 d_{\lifting{\manifold{N}}} (\lifting{u}_j (x), \lifting{u}_i (x))
 &\le  d_{\lifting{\manifold{N}}} (\lifting{u}_j (x), \lifting{u}_j(x_j))
 +  d_{\lifting{\manifold{N}}} (\lifting{u}_j (x_j), \lifting{u}_i (x_i))
 +  d_{\lifting{\manifold{N}}} (\lifting{u}_i (x_i), \lifting{u}_i (x))\\
 &\le 2 \lambda  + \diam (\lifting{W});
 \end{split}
\end{equation}
by \eqref{eq_cu0ienae1HaphoRishodicie}, 
we have 
\begin{equation}
\label{eq_oyo1eepaRienipien5guba2A}
 \abs{A_{i,j}} = \abs{A_i \cap A_j} = 
 \abs{A_i} + \abs{A_j} - \abs{A_i \cup A_j}
 \ge \tfrac{2}{3} \abs{\manifold{M}} + \tfrac{2}{3} \abs{\manifold{M}} - \abs{\manifold{M}}
 = \tfrac{1}{3}\abs{\manifold{M}}.
\end{equation}
Therefore, we have by \eqref{eq_BieNgie0thohkie1uG4Axei4} and \eqref{eq_oyo1eepaRienipien5guba2A}
\begin{equation}
  \smashoperator[r]{\iint_{\manifold{M} \times \manifold{M}}} \frac{1}{1 + d_{\lifting{\manifold{N}}} (\lifting{u}_j, \lifting{u}_i )}
  \ge \frac{\abs{A_{i,j}}}{1 + 2 \lambda + \vphantom{\hat{W}}\diam (\lifting{W})}
  \ge \frac{\abs{\manifold{M}}}{3(1 + 2 \lambda + \vphantom{\hat{W}}\diam(\lifting{W}))},
\end{equation}
and it follows from \cref{proposition_spgt1_estimate}, from \cref{prop_compactness} and from the completeness of the manifold \(\lifting{\manifold{N}}\) that, up to a subsequence, the sequence \((\lifting{u}_j)_{j \in \Nset}\) converges almost everywhere on \(\manifold{M}\) to some mapping \(\lifting{u} : \manifold{M} \to \lifting{\manifold{N}}\);
we also have \(\pi \compose \lifting{u} = \lim_{j\to \infty} \pi \compose \lifting{u}_j = \lim_{j \to \infty} u_j = u\) almost everywhere; by Fatou's lemma, by the a priori estimate on the lifting (\cref{proposition_spgt1_estimate}) and by the diagonal covering (\cref{lemma_manifold_diagonal_covering}) we have
\begin{equation*}
\begin{split}
\smashoperator[r]{\iint_{\manifold{M} \times \manifold{M}}}
\frac{d_{\lifting{\manifold{N}}} (\lifting{u} (y), \lifting{u} (x))^p \wedge 1}{d_\manifold{M} (y, x)^{m +sp}} \dif y \dif x 
& \le \liminf_{j \to \infty} \smashoperator{\iint_{\manifold{M} \times \manifold{M}}}
\frac{d_{\lifting{\manifold{N}}} (\lifting{u}_j (y), \lifting{u}_j (x))^p \wedge 1}{d_\manifold{M} (y, x)^{m +sp}} \dif y \dif x \\
&\le \Cl{cst_LohTheifeiNooth9Eepeedei} \liminf_{j \to \infty} \smashoperator{\iint_{\manifold{M} \times \manifold{M}}}
\frac{d_{\manifold{N}} (u_j (y), u_j (x))^p}{d_\manifold{M} (y, x)^{m + sp}} \dif y \dif x \\
& = \Cr{cst_LohTheifeiNooth9Eepeedei} \smashoperator{\iint_{\manifold{M} \times \manifold{M}}}
\frac{d_{\manifold{N}} (u (y), u (x))^p}{d_\manifold{M} (y, x)^{m + sp}} \dif y \dif x,
\end{split}
\end{equation*}
which proves the statement and the estimate \eqref{eq_oGeuzoo2Aeyaifaeto0uelei} when \(\pi : \lifting{\manifold{N}} \to \manifold{N}\) is a normal covering. 

If \(\pi :  \lifting{\manifold{N}} \to \manifold{N}\) is not a normal covering, we choose 
\(\pi_* :  \lifting{\manifold{N}}_* \to \lifting{\manifold{N}}\) to be a universal covering of \(\lifting{\manifold{N}}\), so that in particular
\(\pi \compose \pi_* : \lifting{\manifold{N}}_* \to \manifold{N}\) is a universal covering of \(\manifold{N}\) and thus also a normal covering. Applying the first part of the proof, we get a mapping \(\lifting{u}_* \in \homog{W}^{s, p} (\manifold{M}, \lifting{\manifold{N}}_*)\) such that \(\pi \compose \pi_* \compose \lifting{u}_* = u\) on \(\manifold{M}\); setting \(\lifting{u} \defeq \pi_* \compose \lifting{u}_*\), we reach the conclusion in the general case.
\end{proof}

As a byproduct of the proof of \cref{theorem_lifting_truncated_noncompact}, we get under the weaker condition \(sp > 1\) the existence of a lifting with an estimate for maps that are continuous outside a submanifold of codimension \(2\).

\begin{theorem}
\label{theorem_lifting_R_Wsp}
Let \(\manifold{M}\) and let \(\manifold{N}\) be a compact Riemannian manifold, let \(m \defeq  \dim \manifold{M}\), let \(\pi : \lifting{\manifold{N}} \to \manifold{N}\) be a surjective Riemannian covering map, let \(s \in \intvo{0}{1}\) and let \(p \in \intvo{1}{\infty}\). 
If \(\manifold{M}\) is simply-connected and if \(sp > 1\), then there exists a constant \(C \in \intvo{0}{\infty}\) such that for every map \(u \in \mathcal{R}^0_{m - 2}
(\manifold{M}, \manifold{N}) \cap \homog{W}^{s, p} (\manifold{M}, \manifold{N})\) there exists a measurable map \(\lifting{u} : \manifold{M} \to \lifting{\manifold{N}}\) such that \(\pi \compose \lifting{u} = u\) almost everywhere on \(\manifold{M}\) and \eqref{eq_oGeuzoo2Aeyaifaeto0uelei} holds.
\end{theorem}

As a consequence of \cref{theorem_lifting_R_Wsp} and \cref{prop_compactness}, 
any map which is the almost everywhere limit of a sequence of maps \((u_j)_{j \in \Nset}\) in \(\mathcal{R}^0_{m - 2} (\manifold{M}, \manifold{N}) \cap \homog{W}^{s, p} (\manifold{M}, \manifold{N})\) that is bounded in \(\homog{W}^{s, p} (\manifold{M}, \manifold{N})\) has a lifting \(\lifting{u} : \manifold{M} \to \lifting{\manifold{N}}\) satisfying
\begin{equation}
 \smashoperator{\iint_{\manifold{M} \times \manifold{M}}} \frac{d_{\lifting{\manifold{N}}} (\lifting{u} (y), \lifting{u} (x))^p \wedge 1 }{d_{\manifold{M}} (y, x)^{m + sp}} \dif y \dif x
 \le \liminf_{j \to \infty}   C \smashoperator{\iint_{\manifold{M} \times \manifold{M}}} \frac{d_{\manifold{N}} (u_j (y), u_j (x))^p}{d_{\manifold{M}} (y, x)^{m + sp}} \dif y \dif x.
\end{equation}

When \(1 < sp < 2\), the assumption that the domain \(\manifold{M}\) is simply-connected in \cref{theorem_lifting_truncated_noncompact} and \cref{theorem_lifting_R_Wsp} can be replaced by a smallness assumption on the map to be lifted. 

\begin{theorem}
\label{theorem_lifting_non_simply_connected}
Let \(\manifold{M}\) and let \(\manifold{N}\) be a compact Riemannian manifold, let \(m \defeq  \dim \manifold{M}\), let \(\pi : \lifting{\manifold{N}} \to \manifold{N}\) be a surjective Riemannian covering map, let \(s \in \intvo{0}{1}\) and let \(p \in \intvo{1}{\infty}\). 
If \(sp > 1\), then there exists constants \(\varepsilon, C \in \intvo{0}{\infty}\) such for every map \(u \in \homog{W}^{s, p} (\manifold{M}, \manifold{N})\)
satisfying 
\begin{equation}
\label{eq_Eliushamoo3goozu2uy0quah}
  \smashoperator[r]{\iint_{\manifold{M} \times \manifold{M}}} \frac{d_{\manifold{N}} (u (y), u (x))^p}{d_{\manifold{M}} (y, x)^{m + sp}} \dif y \dif x \le \varepsilon,
\end{equation}
and satisfying also 
\(u \in \mathcal{R}^0_{m - 2}
(\manifold{M}, \manifold{N})\) when \(1 < sp < 2\), there exists a measurable map \(\lifting{u} : \manifold{M} \to \lifting{\manifold{N}}\) such that \(\pi \compose \lifting{u} = u\) almost everywhere on \(\manifold{M}\) and \eqref{eq_oGeuzoo2Aeyaifaeto0uelei} holds.
\end{theorem}

When \(\pi: \Rset\to \Sset^1\) is the universal covering of the circle, \cref{theorem_lifting_non_simply_connected} is a reformulation of a result of Brezis \& Mironescu \cite{Brezis_Mironescu_2021}*{th.\ 14.5 \& \S 14.6.2}.

\begin{proof}[Proof of \cref{theorem_lifting_non_simply_connected}]
We follow the proof of \cref{theorem_lifting_truncated_noncompact}, noting that \(\pi_{1} (\mathcal{M})\) has finitely many generators, so that if \(\varepsilon \in \intvo{0}{\infty}\) is taken small enough, the smallness assumption \eqref{eq_Eliushamoo3goozu2uy0quah} implies that \(u_j\) has a lifting on a finite set of loops generating \(\pi_1 (\manifold{N})\) and not intersecting the singular set of \(u_j\) and hence \(u_j\) has a lifting outside its singular set.
\end{proof}

\subsection{Uniqueness of the lifting}

The lifting given by \cref{theorem_lifting_truncated_noncompact} turns out to be essentially unique, as it is well established for the lifting in fractional Sobolev spaces \citelist{\cite{Bourgain_Brezis_Mironescu_2000}*{Appendix B}\cite{Bethuel_Chiron_2007}*{lem.\ A.4}}.

The main analytical tool is the following result of Bourgain, Brezis \& Mironescu \citelist{\cite{Brezis_2002}\cite{Bourgain_Brezis_Mironescu_2001}*{Appendix B}} (see also \citelist{\cite{Brezis_Mironescu_2021}*{cor.\ 6.4}\cite{DeMarco_Mariconda_Solimini_2008}\cite{RanjbarMotlagh_2020}}).

\begin{lemma}
\label{lemma_potential_char}
Let \(\manifold{M}\) be a connected Riemannian manifold with \(m \defeq \dim \manifold{M}\).
If the set \(A \subseteq \manifold{M}\) is measurable and if
\begin{equation*}
 \int_{A} \int_{\manifold{M} \setminus A}
 \frac{1}{\dist_{\manifold{M}} (y, x)^{m + 1}} \dif y \dif x<\infty,
\end{equation*}
then either \(\abs{A}=0\) or \(\abs{\manifold{M} \setminus A} = 0\).
\end{lemma}

\begin{proof}[Proof of \cref{proposition_lifting_unique}]
We define the set 
\begin{equation}
\label{eq_tee6euGaeKizah5Aeng6pohT}
 A \defeq \set{x \in \manifold{M} \st \lifting{u}_1 (x) = \lifting{u}_0 (x)}.
\end{equation}
We observe that if \(x \in A\) and \(y \in \manifold{M} \setminus A\), then by \cref{lemma_small_isometry} and by the triangle inequality
\begin{equation*}
\begin{split}
 \inj(\manifold{N}) &\le  d_{\lifting{\manifold{N}}} (\lifting{u}_1 (y), \lifting{u}_0 (y))\\
 &\le  d_{\lifting{\manifold{N}}} (\lifting{u}_1 (y), \lifting{u}_1 (x)) +  d_{\lifting{\manifold{N}}} (\lifting{u}_1 (x), \lifting{u}_0 (x)) + d_{\lifting{\manifold{N}}} (\lifting{u}_0 (x), \lifting{u}_0 (y))\\
 &= d_{\lifting{\manifold{N}}} (\lifting{u}_1 (y), \lifting{u}_1 (x)) +  d_{\lifting{\manifold{N}}} (\lifting{u}_0 (x), \lifting{u}_0 (y)),
\end{split}
\end{equation*}
and thus either \(d_{\lifting{\manifold{N}}} (\lifting{u}_0 (x), \lifting{u}_0 (y)) \ge
\inj(\manifold{N})/2\) or \(d_{\lifting{\manifold{N}}} (\lifting{u}_1 (x), \lifting{u}_1 (y)) \ge \inj(\manifold{N})/2\),
and thus 
\begin{equation*}
 \int_{A} \int_{\manifold{M} \setminus A} \frac{1}{d_{\manifold{M}} (x, y)^{m + 1}}\dif y \dif x
 \le \sum_{j \in \{0, 1\}}\hspace{1em}\smashoperator{\iint_{\substack{x, y \in \manifold{M}\\ d_{\lifting{\manifold{N}}} (\lifting{u}_j (x), \lifting{u}_j (y))\ge \inj(\manifold{N})/2}}}  \frac{1}{d_{\manifold{M}} (x, y)^{m + 1}} \dif y \dif x  < \infty.
\end{equation*}
It follows then from \cref{lemma_potential_char} that either \(\abs{A} = 0\) or \(\abs{\manifold{M} \setminus A} = 0\) and the conclusion follows from the definition of \(A\) in \eqref{eq_tee6euGaeKizah5Aeng6pohT}.
\end{proof}

The space \(X (\manifold{M}, \lifting{\manifold{N}})\) contains all the functions such that the left-hand side of \eqref{eq_oGeuzoo2Aeyaifaeto0uelei} in \cref{theorem_lifting_truncated_noncompact} is finite.

\begin{proposition}
\label{proposition_nonlinear_sum_in_X}
 Let \(\manifold{M}\) be a compact Riemannian manifold, let \(\pi : \lifting{\manifold{N}} \to \manifold{N} \) be a Riemannian covering, let \(s \in \intvo{0}{1}\), let \(p \in \intvo{1}{\infty}\) and let \(q \in \intvr{0}{\infty}\). If \(sp > 1\) and if the mapping \(u : \manifold{M} \to \lifting{\manifold{N}}\) is measurable and satisfies 
\begin{equation*}
  \smashoperator[r]{\iint_{\manifold{M} \times \manifold{M}}} \frac{d_{\lifting{\manifold{N}}} (\lifting{u} (x), \lifting{u} (y))^p \wedge d_{\lifting{\manifold{N}}} (\lifting{u} (x), \lifting{u} (y))^q}{d_{\manifold{M}} (y, x)^{m + sp}} \dif y \dif x < \infty,
\end{equation*}
with \(m \defeq  \dim \manifold{M}\), then \(\lifting{u} \in X (\manifold{M}, \lifting{\manifold{N}})\).
\end{proposition}

\begin{proof}
We have
\begin{multline}
\label{eq_quiePhuo7diethai1xat6ier}
  \smashoperator{\iint_{\substack{x, y \in \manifold{M}\\ d_{\lifting{\manifold{N}}} (\lifting{u} (x), \lifting{u} (y))\ge \inj(\manifold{N})/2}}} \frac{(\inj(\manifold{N})/2)^p \wedge (\inj(\manifold{N})/2)^q}{d_{\manifold{M}} (x, y)^{m + 1}} \dif y \dif x\\ 
 \le 
 \diam(\manifold{M})^{sp - 1}\smashoperator{\iint_{\manifold{M} \times \manifold{M}}} \frac{d_{\lifting{\manifold{N}}} (\lifting{u} (x), \lifting{u} (y))^p \wedge d_{\lifting{\manifold{N}}} (\lifting{u} (x), \lifting{u} (y))^q}{d_{\manifold{M}} (y, x)^{m + sp}} \dif y \dif x < \infty.
\end{multline}
\end{proof}

Classical fractional uniqueness results for \(u_0, u_1 \in \homog{W}^{s, p} (\manifold{M}, \lifting{\manifold{N}})\) with \(0 < s < 1\) \citelist{\cite{Bourgain_Brezis_Mironescu_2001}\cite{Bethuel_Chiron_2007}} can be recovered from \cref{proposition_lifting_unique} and \cref{proposition_nonlinear_sum_in_X}.

The uniqueness property of the lifting also allows one to write any lifting in term of a fixed lifting over a normal covering.

\begin{proposition}
\label{proposition_uniqueness_automorphism}
Let \(\manifold{M}\) be a Riemannian manifold, let \(\pi_\sharp : \lifting{\manifold{N}}_\sharp \to \lifting{\manifold{N}}_\flat\) and \(\pi_\flat : \lifting{\manifold{N}}_\flat \to \manifold{N}\) be Riemannian coverings,
and let \(\lifting{u}_\sharp \in X (\manifold{M}, \lifting{\manifold{N}}_\sharp)\) and \(\lifting{u}_\flat \in X (\manifold{M},\lifting{\manifold{N}}_\flat)\).
If \(\manifold{M}\) is connected, if the covering \(\pi_{\sharp}\) is surjective, if the covering \(\pi_\flat \compose \pi_{\sharp}\) is normal, 
and if \(\pi_{\flat} \compose  \lifting{u}_\flat= \pi_{\flat} \compose \pi_{\sharp} \compose \lifting{u}_\sharp\) almost everywhere on \(\manifold{M}\),
then there exists \(\tau \in \Aut(\pi_\flat \compose \pi_{\sharp})\) such that 
\(\lifting{u}_\flat = \pi_\sharp \compose \tau \compose \lifting{u}_\sharp\) almost everywhere on \(\manifold{M}\).
\end{proposition}
\begin{proof}
Since the covering \(\pi_{\sharp}\) is surjective, for every \(x \in \manifold{M}\),
there exists \(\lifting{y}_\sharp \in \lifting{\manifold{N}}_\sharp\) such that \(\pi_{\sharp} (\lifting{y}_\sharp) = \lifting{u}_{\flat} (x)\).
For almost every \(x \in \manifold{M}\), since \(\pi_{\flat} (\pi_{\sharp} (\lifting{y}_\sharp)) = \pi_\flat (\lifting{u}_\flat (x))= \pi_{\flat} (\pi_{\sharp} (\lifting{u}_{\sharp}(x)))\) and since the covering \(\pi_\flat \compose \pi_\sharp\) is normal, there exists \(\tau \in \Aut(\pi_{\flat} \compose \pi_\sharp)\) such that \(\lifting{y}_\sharp = \tau(\lifting{u}_\sharp (x))\) and thus \(\lifting{u}_\flat(x) = \pi_{\sharp} (\tau (\lifting{u}_\sharp (x)))\).
Hence we have 
\begin{equation}
 \manifold{M}
 = \smashoperator{\bigcup_{\tau \in \Aut(\pi_{\flat} \compose \pi_{\sharp})} }
 \set{x \in \manifold{M} \st \lifting{u}_\flat(x) = \pi_{\sharp} (\tau (\lifting{u}_\sharp (x)))} \cup E,
\end{equation}
where \(E \subseteq \manifold{M}\) satisfies \(\abs{E} = 0\).
Since the set \(\Aut (\pi_\flat \compose \pi_\sharp)\) is countable, there exists \(\tau \in \Aut(\pi_\flat \compose \pi_{\sharp})\) such that 
\( \lifting{u}_\flat = \pi_{\sharp} \compose \tau \compose \lifting{u}_\sharp\) on a set of positive measure of \(\manifold{M}\) and the identity then holds outside a null set by the uniqueness of lifting (\cref{proposition_lifting_unique}) since \(\manifold{M}\) is connected.
\end{proof}

As a consequence of \cref{proposition_uniqueness_automorphism}, we get that a lifting in \(X(\manifold{M}, \lifting{\manifold{N}})\) of a continuous map is necessarily essentially continuous.

\begin{proposition}
\label{proposition_lifting_X_continuous}
Let \(\manifold{M}\) be a Riemannian manifold, let \(\pi : \lifting{\manifold{N}} \to \lifting{\manifold{N}}\)
by a Riemannian covering.
If \(\lifting{u} \in X (\manifold{M}, \lifting{\manifold{N}})\) and if \(u = \pi \compose \lifting{u}\) is continuous, then there exists \(\lifting{v} \in C (\manifold{M}, \lifting{\manifold{N}})\) such that \(\lifting{v} = \lifting{u}\) almost everywhere on \(\manifold{M}\).
\end{proposition}
\begin{proof}
We first assume that the manifold \(\manifold{M}\) is simply-connected.
We apply \cref{proposition_uniqueness_automorphism} with \(\pi_\flat = \pi : \lifting{\manifold{N}} \to \manifold{N}\), \(\pi_{\sharp} : \lifting{\manifold{N}}_* \to \lifting{\manifold{N}}\) a universal covering and \(\lifting{v} \in C (\manifold{M}, \lifting{\manifold{N}}_*)\) such that \(\pi \compose \lifting{v} = \pi \compose \lifting{u}\). The conclusion then follows from \cref{proposition_uniqueness_automorphism}.

In the general case, we cover the manifold \(\manifold{M}\) by simply-connected open sets \(U_j \subseteq \manifold{M}\), with \(j \in J\).
By the first part of the proof, for every \(j \in J\), there exists a mapping \(\lifting{v}_j \in C (U_j, \lifting{\manifold{N}})\) such that \(\lifting{u} = \lifting{v}_j\) almost everywhere in \(U_j\).
For every \(j, \ell \in J\), it follows in view of the continuity of the mappings \(\lifting{v}_j\) and \(\lifting{v}_\ell\) that \(\lifting{v}_j = \lifting{v}_\ell\) everywhere in \(U_j \cap U_\ell\). 
Therefore the map \(\lifting{v}\) can be defined in such a way that for every \(j \in J\) its restriction \(\lifting{v} \restr{U_j}\) to the set \(U_j\) satisfies \(\lifting{v} \restr{U_j} = \lifting{v}_j\) and that \(\lifting{v}\) is continuous on \(\manifold{M}\).
\end{proof}

\subsection{A priori estimate on the lifting}
\Cref{theorem_lifting_truncated_noncompact_apriori} will be proved as a consequence of \cref{proposition_X_lifting_of_Wsp_in_Y}, once one notices that liftings in \(X(\Omega, \lifting{\manifold{N}})\) of maps in \(\homog{W}^{s, p} (\Omega, \manifold{N})\) with \(sp > 1\) turn out to be in \(Y(\Omega, \lifting{\manifold{N}})\).

\begin{proposition}
\label{proposition_X_lifting_of_Wsp_in_Y}
Let \(m \in \Nset \setminus \set{0}\), let \(\Omega \subseteq \Rset^m\) be open and convex, let \(\manifold{N}\) be a compact Riemannian manifold, let \(\pi : \lifting{\manifold{N}} \to \manifold{N}\) be a Riemannian covering map, let \(s \in \intvo{0}{1}\) and let \(p \in \intvo{1}{\infty}\). 
If \(sp > 1\), if \(\lifting{u} \in X (\Omega, \lifting{\manifold{N}})\)
and if \(u \defeq \pi \compose \lifting{u} \in \homog{W}^{s, p} (\Omega, \manifold{N})\),
then \(\lifting{u} \in Y (\Omega, \lifting{\manifold{N}})\).
\end{proposition}

In order to prove \cref{proposition_X_lifting_of_Wsp_in_Y}, we will use the following consequence of Fubini's theorem, which implements the rotation method on the space \(X(\Omega,\lifting{\manifold{N}})\).

\begin{lemma}
\label{lemma_Fubini_X}
For every \(m \in \Nset \setminus \set{0}\), there exists a constant \(C \in \intvo{0}{\infty}\) such that for every convex open set \(\Omega \subset \Rset^m\), every metric space \(\mathcal{E}\), every \(\delta \in \intvo{0}{\infty}\) and every measurable function \(f : \Omega \to \manifold{E}\), we have
\begin{equation}
\smashoperator{\iint_{\substack{x, y \in \Omega\\ d_{\manifold{E}} (f (y), f (x))\ge \delta}}} \frac{1}{\abs{y -  x}^{m + 1}} \dif y \dif x
= C 
\int_{\Sset^{m - 1}} 
\int_{w^\perp} \smashoperator{\iint_{\substack{x, y \in \Omega \cap (z + \Rset w )\\ d_{\manifold{E}} (f (y), f (x))\ge \delta}}} \frac{1}{\abs{y - x}^{2}} \dif y \dif x \dif z \dif w .
\end{equation}

\end{lemma}

\begin{proof}[Proof of \cref{proposition_X_lifting_of_Wsp_in_Y}]
Since \(sp > 1\), by Fubini's theorem and the fractional Morrey--Sobolev embedding, for every straight line \(L \subseteq \Rset^m\), there exists a mapping \(u_L \in C (\Omega \cap L, \lifting{\manifold{N}})\) such that 
\(u \restr{\Omega \cap L} = u_L = \pi \compose \lifting{u} \restr{\Omega \cap L}\) almost everywhere in \(\Omega \cap L\).
Similarly, by \cref{lemma_Fubini_X}, we have \(\lifting{u}\restr{\Omega \cap L} \in X (\Omega \cap L, \manifold{N})\). By \cref{proposition_lifting_X_continuous}, there exists a mapping \(\lifting{u}_L \in C (\Omega \cap L, \lifting{\manifold{N}})\) such that \(\lifting{u}\restr{\Omega \cap L} = \lifting{u}_L\) almost everywhere on \(\Omega \cap L\).
It follows thus by definition \eqref{eq_definition_Y} that \(\lifting{u} \in Y (\Omega, \lifting{\manifold{N}})\). 
\end{proof}

\begin{proof}[Proof of \cref{theorem_lifting_truncated_noncompact_apriori}]
By \cref{proposition_X_lifting_of_Wsp_in_Y}, the a priori estimate \cref{proposition_spgt1_estimate} holds on any local chart.
We reach the conclusion by the covering of \cref{lemma_manifold_diagonal_covering}.
\end{proof}

\section{Relationship to linear Sobolev spaces}

\subsection{Characterization as a sum of Sobolev spaces}
Our proof of \cref{theorem_nonlinear_sum} that characterizes the space of liftings appearing in \cref{theorem_lifting_truncated_noncompact}  and \cref{proposition_equivalent_norms} will use the following density result.

\begin{proposition}
\label{proposition_dense_truncated}
Let \(m \in \Nset \setminus \set{0}\), let \(s \in \intvo{0}{1}\) and let \(p \in \intvr{1}{\infty}\).
If \(U \subseteq \Rset^m\) is open and if \(f : U \to \Rset\) is a measurable function satisfying
\begin{equation}
\label{eq_ash6Iehahs7quie3thebi0ea}
  \smashoperator{\iint_{U \times U}}
 \frac{\abs{f (y) - f (x)}^p \wedge \abs{f (y) - f (x)}}{\abs{y - x}^{m + sp}}\dif y \dif x  < \infty,
\end{equation}
then for every set \(\Omega \subseteq U\) such that \(\dist(\Omega, \Rset^m \setminus U) > 0\), 
there exists a sequence \((f_j)_{j \in \Nset}\) in \(C^\infty (\Bar{\Omega}, \Rset)\) such that \(f_j \to f\) almost everywhere in \(\Omega\) as \(j \to \infty\) and 
\begin{equation}
\label{eq_tuph2EiDiecoVung9uosha5a}
\sup_{j \in \Nset}
\smashoperator{\iint_{\Omega \times \Omega}}
 \frac{\abs{f_j (y) - f_j (x)}^p \wedge \abs{f_j (y) - f_j (x)}}{\abs{y - x}^{m + sp}}\dif y \dif x < \infty.
\end{equation}
\end{proposition}

\begin{proof}
We define the function \(\Phi : \Rset \to \Rset\) for each \(t \in \Rset\) by 
\begin{equation}
\label{eq_yooSoo9IaVaoWeeZieco9si0}
 \Phi (t)\defeq
 \begin{cases}
  \abs{t}^p &\text{if \(\abs{t} \le 1\)},\\
  1 + p \brk{\abs{t} - 1} &\text{if \(\abs{t} \ge 1\)}.
 \end{cases}
\end{equation}
We observe that the function \(\Phi\) is convex and that it satisfies for every \(t \in \Rset\)
\begin{equation}
\label{eq_saw8ues2tai8siWahsoh2eem}
 \abs{t}^p \wedge \abs{t} \le \Phi (t) \le \abs{t}^p \wedge (p \abs{t})
 \le p (\abs{t}^p \wedge \abs{t}).
\end{equation}
We fix a function \(\eta \in C^\infty_c (\Rset^m, \Rset)\) such that \(\eta \ge 0\) and \(\int_{\Rset^m} \eta = 1\).
Since the condition \eqref{eq_ash6Iehahs7quie3thebi0ea} implies the local integrability of the function \(f\), there exists a sequence \((\delta_j)_{j \in \Nset}\) in \(\intvo{0}{\infty}\) such that the function \(f_j : \smash{\Bar{\Omega}} \to \Rset\) defined for each \(x \in \Omega\) by 
\begin{equation}
 f_j (x) \defeq \int_{\Rset^m} \eta (z) f (x - \delta_j z) \dif z,
\end{equation}
is well-defined and \(f_j \to f\) almost everywhere in \(\Omega\) as \(j \to \infty\).
Moreover, since the function \(\Phi\) is convex, we have for every \(j \in \Nset\),
\begin{equation}
\label{eq_Needai8echei0neef1eghohZ}
 \smashoperator{\iint_{\Omega \times \Omega}}
 \frac{\Phi (f_j (y) - f_j (x))}{\abs{y - x}^{m + sp}}\dif y \dif x
 \le  \smashoperator{\iint_{U \times U}}
 \frac{\Phi (f (y) - f (x))}{\abs{y - x}^{m + sp}}\dif y \dif x,
\end{equation}
and \eqref{eq_tuph2EiDiecoVung9uosha5a} follows from \eqref{eq_saw8ues2tai8siWahsoh2eem} and \eqref{eq_Needai8echei0neef1eghohZ}.
\end{proof}

We now prove the characterization of the sum \(W^{s, p} (\manifold{M}, \Rset) + W^{1, sp} (\manifold{M}, \Rset)\).

\begin{proof}%
[Proof of \cref{theorem_nonlinear_sum}]%
\resetconstant%
In order to prove the inclusion \(\subseteq\) in \eqref{eq_eahep6eet2yoo2ohBeePee9w}, 
we first assume that the set \(\Omega \subseteq \Rset^m\) is bounded and open with a smooth boundary \(\partial \Omega\), that \(\Omega \subseteq U\), with \(\dist(\Omega, \Rset^m \setminus U) > 0\) for some open set \(U \supseteq \Bar{\Omega}\), and that, in view of \cref{proposition_equivalent_norms}, 
\begin{multline}
\label{eq_icae0Aujohh5Uo6nahle4eiZ}
  \smashoperator{
    \iint_{U \times U}
  }
  \frac
    {\abs{f (x) - f (y)}^p \wedge \abs{f (x) - f (x)}}
    {\abs{y - x}^{m + sp}}
    \dif y 
    \dif x
 \\[-1em]
  \le 
    \smashoperator{\iint_{U \times U}}
    \frac
      {\abs{f (x) - f (y)}^p \wedge \abs{f (x) - f (x)}^q}
      {\abs{y - x}^{m + sp}}
      \dif y 
      \dif x 
      < 
      \infty
      .
\end{multline}
By \cref{proposition_dense_truncated} and by \eqref{eq_icae0Aujohh5Uo6nahle4eiZ}, there exists a sequence \((f_j)_{j \in \Nset}\) in \(C^\infty (\Bar{\Omega}, \Rset)\) such that \(f_j \to f\) almost everywhere in \(\Omega\)
and 
\begin{equation}
\label{eq_uMae7bei1ahNeece9Ohno4Uo}
  \sup_{j \in \Nset} \smashoperator{\iint_{\Omega \times \Omega}}
 \frac{\abs{f_j (y) - f_j (x)}^p \wedge 1}{\abs{y -x}^{m + sp}}\dif y \dif x \le 
 \C \smashoperator{\iint_{U \times U}}
 \frac{\abs{f (y) - f (x)}^p \wedge \abs{f (y) - f (x)}}{\abs{y - x}^{m + sp}}\dif y \dif x < \infty.
\end{equation}
For every \(j \in \Nset\),
setting \(u_j \defeq e^{i f_j}\), we have  
\begin{equation}
\label{eq_othiu9veix2aib9YohDiwoob}
  \smashoperator{\iint_{\Omega \times \Omega}}
 \frac{\abs{u_j (y) - u_j (x)}^p}{\abs{y - x}^{m + sp}} \dif y \dif x 
 \le 
 \C \smashoperator{\iint_{\Omega \times \Omega}}
 \frac{\abs{f_j (x) - f_j (y)}^p \wedge 1}{\abs{y - x}  ^{m + sp}}\dif y \dif x < \infty.
\end{equation}
Since \(sp > 1\), by the lifting in the sum of Sobolev spaces \citelist{\cite{Brezis_Mironescu_2021}*{th.\ 8.8}\cite{Mironescu_preprint}*{th.\ 2}} (see also \citelist{\cite{Mironescu_2010_Decomposition}\cite{Nguyen_2008}\cite{Mironescu_2008}\cite{Bourgain_Brezis_Mironescu_2001}}), we can write
\(f_j = g_j + h_j\) with the functions \(g_j \in \homog{W}^{s, p} (\Omega, \Rset)\) and \(h_j \in \homog{W}^{1, sp} (\Omega, \Rset)\) satisfying the estimates
\begin{gather} 
\label{eq_OK5EL0Gaphaer7vie1oozoon}
\smashoperator{\iint_{\Omega \times \Omega}}
 \frac{\abs{g_j (y) - g_j (x)}^p}{\abs{y - x}^{m + sp}}\dif y \dif x
\le \C
 \smashoperator{\iint_{\Omega \times \Omega}}
 \frac{\abs{u_j (x) - u_j (y)}^p}{\abs{y - x}^{m + sp}}\dif y \dif x,
 \\
\label{eq_oopheeSh2iequael1hagei8u}
 \int_{\Omega} \abs{D h_j}^{sp}
\le \C
\smashoperator{\iint_{\Omega \times \Omega}}
 \frac{\abs{u_j (x) - u_j (y)}^p}{\abs{y - x}^{m + sp}}\dif y \dif x,\\
\intertext{and the conditions}
\label{eq_Pool4xainajiepaoJ3shohti}
\int_{\Omega} g_j = \int_{\Omega} h_j = \frac{1}{2} \int_{\Omega} f_j .
\end{gather}
Up to a subsequence, we can assume that \(g_j \to g\) and \(h_j \to h\) almost everywhere in \(\Omega\) as \(j \to \infty\), with the functions \(g \in \homog{W}^{s, p} (\Omega, \Rset)\) and \(h \in \homog{W}^{1, s p} (\Omega, \Rset)\) satisfying in view of \eqref{eq_uMae7bei1ahNeece9Ohno4Uo}, \eqref{eq_othiu9veix2aib9YohDiwoob}, \eqref{eq_OK5EL0Gaphaer7vie1oozoon} and \eqref{eq_oopheeSh2iequael1hagei8u}
\begin{gather} 
\label{eq_ahWai8ooch0eiwoohohlain2}
\smashoperator{\iint_{\Omega \times \Omega}}
 \frac{\abs{g (y) - g (x)}^p}{\abs{y - x}^{m + sp}}\dif y \dif x
\le 
\C
 \smashoperator{\iint_{U \times U}}
 \frac
  {\abs{f (x) - f (y)}^p \wedge \abs{f (x) - f (x)}}
  {\abs{y - x}^{m + sp}}
 \dif y 
 \dif x,
 \\
\label{eq_xawooDeephuRohmeeg6ox1ge}
 \int_{\Omega} \abs{D h}^{sp}
\le  \C
 \smashoperator{\iint_{U \times U}}
 \frac{\abs{f (x) - f (y)}^p \wedge \abs{f (x) - f (x)}}{\abs{y - x}^{m + sp}}\dif y \dif x,
 \intertext{and in view of \eqref{eq_Pool4xainajiepaoJ3shohti}}
 \label{eq_hoo3chaeles9Xaidoo2AeChu}
 \int_{\Omega} g = \int_{\Omega} h = \frac{1}{2} \int_{\Omega} f .
\end{gather}

In the general case we follow Rodiac \& Van Schaftingen \cite{Rodiac_VanSchaftingen_2021}*{proof of prop.\ 4.1}.  
Since \(\manifold{M}\) is a compact manifold with boundary, there exist \(N \in \Nset\), and for \(k \in \set{1, \dotsc, N}\), a diffeomorphism \(\psi_k : U_k \to \Rset^m\) such that  either \(\psi_k (U_k) = \Bset^m \subset \Rset^m\) or \(\psi_k (U_k) = \Bset^m \cap \Rset^{m - 1} \times \intvr{0}{\infty}\)
and such that \(\manifold{M} =\bigcup_{k=1}^N U_k\). 
We take a partition of unity \((\varphi_k)_{1 \le k \le N}\) associated to the sets \(U_k\), that is, for every \(k \in \set{1, \dotsc, N}\), \(\varphi_k \in C^1 (\manifold{M}, \Rset)\) and \(\varphi_k = 0\) in \(\manifold{M} \setminus U_k\),
and \(\sum_{i = 1}^{N} \varphi_k = 1\) on \(\manifold{M}\).
Given a measurable function \(f : \manifold{M} \to \Rset\), for each \(k \in \set{1, \dotsc, N}\), 
we define the function \(f_k\defeq f \compose \psi_k^{-1} : \psi_k (U_k) \to \Rset\)
to which we apply the first part of the proof which yields functions \(g_k\in  \homog{W}^{s, p} (\psi_k (U_k), \Rset)\) and \(h_k\in  \homog{W}^{1, s p} (\psi_k (U_k), \Rset)\) satisfying \eqref{eq_ahWai8ooch0eiwoohohlain2}, \eqref{eq_xawooDeephuRohmeeg6ox1ge} and \eqref{eq_hoo3chaeles9Xaidoo2AeChu} with \(\Omega = \psi_k (U_k)\).
Defining the functions
\begin{gather*}
 g_* \defeq \sum_{k = 1}^N \varphi_k \brk[\bigg]{g_k \compose \psi_k - \fint_{U_k} g_k\compose \psi_k},
 \\
 h_* 
 \defeq 
 \sum_{k = 1}^N 
 \varphi_k 
 \brk[\bigg]{h_k \compose \psi_k - \fint_{U_k} h_k\compose \psi_k},
\end{gather*}
and the low frequency component
\[
  f_0 \defeq \sum_{k = 1}^N \varphi_k \fint_{U_k} f,
\]
we have 
\(f = f_0 + g_* + h_*\)
on \(\manifold{M}\).
Moreover, since
\[
  f_0 \defeq \sum_{k = 1}^N \varphi_k \brk[\bigg]{\fint_{U_k} f - \fint_{\manifold{M}} f} + \fint_{\manifold{M}} f,
\]
where the last term is constant,
we have 
\[
\norm{D f_0}_{L^\infty(\manifold{M})}
\le 
\C \smashoperator{\iint_{\manifold{M} \times \manifold{M}}} \abs{f (y) - f (x)} \dif y \dif x,
\]
and 
\[
 \Phi (\norm{D f_0}_{L^\infty(\manifold{M})})
\le 
\C \smashoperator{\iint_{\manifold{M} \times \manifold{M}}} \Phi (\abs{f (y) - f (x)})\dif y \dif x,
\]
with the convex function \(\Phi\) defined as in \eqref{eq_yooSoo9IaVaoWeeZieco9si0}.
Since \(1 < sp < p\), by \eqref{eq_saw8ues2tai8siWahsoh2eem}, we have
\begin{equation*}
\begin{split}
 \norm{D f_0}_{L^\infty(\manifold{M})}^p \wedge
 \norm{D f_0}_{L^\infty(\manifold{M})}^{sp}
 &\le \norm{D f_0}_{L^\infty(\manifold{M})}^p \wedge
 \norm{D f_0}_{L^\infty(\manifold{M})}\\
&\le \C \smashoperator{\iint_{\manifold{M} \times \manifold{M}}}
 \frac{\abs{f (y) - f (x)}^p \wedge \abs{f (y) - f (x)}}{d_{\manifold{M}} (y, x)^{m + sp}}\dif y \dif x,
 \end{split}
\end{equation*}
so that 
\begin{multline*}
  \smashoperator{\iint_{\manifold{M} \times \manifold{M}}}
 \frac{\abs{f_0 (x) - f_0 (y)}^p}{d_{\manifold{M}} (y, x)^{m + sp}}\dif y \dif x 
 \wedge \int_{\manifold{M}} \abs{D f_0}^{sp}\\[-1em]
\le \C \smashoperator{\iint_{\manifold{M} \times \manifold{M}}}
 \frac{\abs{f (x) - f (y)}^p \wedge \abs{f (x) - f (x)}}{d_{\manifold{M}} (y, x)^{m + sp}}\dif y \dif x,
\end{multline*}
By either taking \(g \defeq g_*\) and \(h \defeq h_* + f_0\) or \(g \defeq g_* + f_0\) and \(h \defeq h_*\), we finally get in view of \cref{proposition_equivalent_norms},
\begin{equation*}
\smashoperator[r]{\iint_{\manifold{M} \times \manifold{M}}}
\frac{\abs{g (y) - g (x)}^p
}{d_\manifold{M} (y, x)^{m + sp}} \dif y \dif x
+ \int_{\manifold{M}} \abs{Dh}^{sp}
\le \C \smashoperator{\iint_{\manifold{M} \times \manifold{M}}}
\frac{\abs{f (y) - f(x)}^p \wedge \abs{f (y) - f (x)}^{q}
}{d_\manifold{M} (y, x)^{m + sp}} \dif y \dif x ,
\end{equation*}
which gives the first estimate and inclusion.

We now prove the reverse inclusion \(\supseteq\) in \eqref{eq_eahep6eet2yoo2ohBeePee9w}.
If \(f = g+ h\) with \(g \in \homog{W}^{s, p}(\manifold{M}, \Rset)\) and \(h \in \homog{W}^{1, sp}(\manifold{M}, \Rset)\), there exists sequences of smooth maps \((g_j)_{j \in \Nset}\) and \((h_j)_{j \in \Nset}\) in \(C^\infty (\manifold{M}, \Rset)\), such that, as \(j \to \infty\), \(g_j \to g\) in \(\homog{W}^{s, p} (\manifold{M}, \Rset)\) and \(h_j \to h\) in \(\homog{W}^{1, sp} (\manifold{M}, \Rset)\). 
For every \(j \in \Nset\), defining \(f_j \defeq g_j + h_j\) and \(u_j \defeq e^{i f_j}\), we have by the fractional Gagliardo--Nirenberg interpolation inequality (see for example \citelist{\cite{Berzis_Mironescu_2001}*{cor.\ 3.2}\cite{Runst_1986}*{lem.\ 2.1}\cite{Brezis_Mironescu_2018}}), since \(sp > 1\) and \(\abs{e^{i h_j}} \le 1\),
\begin{equation}
\begin{split}
 \smashoperator[r]{\iint_{\manifold{M} \times \manifold{M}}}&
 \frac{\abs{u_j (y) - u_j (x)}^p}{d_{\manifold{M}} (y, x)^{m + sp}}\dif y \dif x\\
 &\le \smashoperator[r]{\iint_{\manifold{M} \times \manifold{M}}}
 \frac{\abs{e^{ig_j (y)} - e^{ig_j (x)}}^p}{d_{\manifold{M}} (y, x)^{m + sp}}\dif y \dif x
 +\smashoperator[r]{\iint_{\manifold{M} \times \manifold{M}}}
 \frac{\abs{e^{i h_j (y)} - e^{i h_j (x)}}^p}{d_{\manifold{M}} (y, x)^{m + sp}}\dif y \dif x
 \\
 &\le \C
 \brk[\Bigg]{ \smashoperator[r]{\iint_{\manifold{M} \times \manifold{M}}}
 \frac{\abs{g_j (y) - g_j (x)}^p}{d_{\manifold{M}} (y, x)^{m + sp}}\dif y \dif x
 + \int_{\manifold{M}} \abs{D h_j}^{sp}}.
 \end{split}
\end{equation}
By \cref{theorem_lifting_truncated_noncompact_apriori}, we have for every \(j \in \Nset\)
\begin{equation}
\label{eq_ohsheej0ahx1GoitheNgahsh}
\smashoperator[r]{\iint_{\manifold{M} \times \manifold{M}}}
 \frac{\abs{f_j (y) - f_j (x)}^p\wedge 1}{d_{\manifold{M}} (y, x)^{m + sp}}\dif y \dif x
 \le
\smashoperator{\iint_{\manifold{M} \times \manifold{M}}}
 \frac{\abs{u_j (y) - u_j (x)}^p}{d_{\manifold{M}} (y, x)^{m + sp}}\dif y \dif x.
\end{equation}
Letting \(j \to \infty\) in \eqref{eq_ohsheej0ahx1GoitheNgahsh} and \eqref{eq_ohsheej0ahx1GoitheNgahsh} and applying \cref{proposition_equivalent_norms}, we get 
\begin{multline*}
\smashoperator[l]{\iint_{\manifold{M} \times \manifold{M}}}
\frac{\abs{f (y) - f(x)}^p \wedge \abs{f (y) - f (x)}^{q}
}{d_\manifold{M} (y, x)^{m + sp}} \dif y \dif x \\[-2em]
\le\C \brk[\Bigg]{\,\smashoperator[r]{\iint_{\manifold{M} \times \manifold{M}}}
\frac{\abs{g (y) - g (x)}^p
}{d_\manifold{M} (y, x)^{m + sp}} \dif y \dif x
+ \int_{\manifold{M}} \abs{Dh}^{sp}},
\end{multline*}
which proves the announced reverse inclusion and estimate.
\end{proof}

\subsection{About the critical lower exponent}
If the function \(f : \Bset^m \to \Rset\) is measurable and if \(q \in \intvr{1}{\infty}\), then it is known that 
\begin{equation}
 \smashoperator{\iint_{\Bset^m \times \Bset^m}} \frac{\abs{f (y) - f (x)}^{q}}{\abs{y - x}^{m + q}} \dif y \dif x = \infty
\end{equation}
unless the function \(f\) is constant \cite{Brezis_2002}*{prop.\ 2}.
Although the integral restricted to a region of large oscillation 
\begin{equation}
\label{eq_cebie5AiSh9iTh8EeSh5Rofo}
 \smashoperator{\iint_{\substack{(x, y) \in \Bset^m \times \Bset^m\\ \abs{f (y) - f (x)} \ge 1}}} \frac{\abs{f (y) - f (x)}^{q}}{\abs{y - x}^{m + q}} \dif y \dif x = \infty
\end{equation}
might be finite for a function of small oscillation, there are still Sobolev functions for which the large oscillation part of the integral \eqref{eq_cebie5AiSh9iTh8EeSh5Rofo} blows up.

\begin{proposition}
\label{proposition_counterexample_critical}
Let \(m \in \Nset \setminus \set{0, 1}\).
If \(1 \le q < m\), then there exists a function \(f \in \homog{W}^{1, q}_0 (\Bset^m, \Rset)\) such that 
\begin{equation}
 \smashoperator{\iint_{\substack{(x, y) \in \Bset^m \times \Bset^m\\ \abs{f (y) - f (x)} \ge 1}}} \frac{\abs{f (y) - f (x)}^{q}}{\abs{y - x}^{m + q}} \dif y \dif x = \infty.
\end{equation}
\end{proposition}
As a consequence of \cref{proposition_counterexample_critical}, if \(1 \le sp < m\), there exists a function \(f \in \homog{W}^{1, sp}_0 (\Bset^m_1, \Rset)\) such that 
\begin{equation*}
 \smashoperator{\iint_{\substack{(x, y) \in \manifold{M} \times \manifold{M}\\ \abs{f (y) - f (x)} \ge 1}}} \frac{\abs{f (y) - f (x)}^{p} \wedge \abs{f (y) - f (x)}^{sp}}{d_{\manifold{M}} (y, x)^{m + sp}} \dif y \dif x 
 = \infty
\end{equation*}
and thus the noninclusion \eqref{eq_ia4oDaeZaXahPoojahthie6a} holds.

\begin{proof}[Proof of \cref{proposition_counterexample_critical}]
\resetconstant
We choose a function \(\psi \in C^\infty_c (\Rset^m, \Rset)\) such that \(\psi (x) = x_1\) when \(x \in B_{1/2}\) and \(\supp \psi \subset B_1\).
For every \(\lambda \in \intvo{2}{\infty}\), we have
\begin{equation}
\label{eq_aereeXahjuul9kieGhaesoo6}
\begin{split}\smashoperator{\iint_{\substack{(x, y) \in \Bset^m \times \Bset^m\\ \abs{\lambda \psi (y) - \lambda \psi (x)} \ge 1}}} \frac{\abs{\lambda \psi (y) - \lambda \psi (x)}^{q}}{\abs{y - x}^{m + q}} \dif y \dif x 
&\ge \lambda^q \smashoperator{\iint_{\substack{(x, y) \in B_{1/2} \times B_{1/2}\\ \lambda \abs{y_1 - x_1} \ge 1}}} \frac{1}{\abs{y - x}^{m}} \dif y \dif x \\
&\ge \Cl{cst_Eet3ahz0rooquei3phoh8aib} \lambda^q \ln (\tfrac{\lambda}{2} - 1),
\end{split}
\end{equation}
for some constant \(\Cr{cst_Eet3ahz0rooquei3phoh8aib} \in \intvo{0}{\infty}\).
We now define for each \(j \in \Nset\) the numbers 
\begin{align}
\label{eq_Xo3ye9aesaev2choph6ohzoy}
\lambda_j & \defeq \lambda_0 2^{j^2}&&\text{ and } & \rho_j &\defeq \brk[\big]{\lambda_j^{q} \ln (\tfrac{\lambda_j}{2} - 1)}^{-1/(m - q)},
\end{align}
with \(\lambda_0 \in \intvo{2}{\infty}\)
 large enough so that there exists a sequence of points \((a_j)_{j \in \Nset}\) for which the closed balls \(\Bar{B}_{\rho_j} (a_j)\) are pairwise disjoint and all contained in \(\Bset^m\) (this is possible since \(q < m\)). 
We define the function \(f: \Bset^m \to \Rset\) for every \(x \in \Bset^m\) by 
\begin{equation*}
 f (x)
 \defeq
 \begin{cases}
  \lambda_j \psi(\frac{x - a_j}{\rho_j}) & \text{if \(x \in B_{\rho_j} (a_j)\)},\\
  0 & \text{otherwise}.
 \end{cases}
\end{equation*}
By the disjointness of the balls, by scaling and by \eqref{eq_Xo3ye9aesaev2choph6ohzoy}, we have
\begin{equation*}
\begin{split}
 \int_{\Bset^m} \abs{D f}^q
 &= \sum_{j \in \Nset}
 \int_{B_{\rho_j} (a_j)} \abs{D f}^q
 = \sum_{j \in \Nset}\lambda_j^q\rho_j^{m - q}
 \int_{\Bset^m} \abs{D \psi}^q\\
 &
 = \sum_{j \in \Nset} \frac{1}{\ln (\lambda_0 2^{j^2 - 1}- 1)} \int_{\Bset^m} \abs{D \psi}^q < \infty,
\end{split}
\end{equation*}
so that \(f \in W^{1, q}_0 (\Bset^m, \Rset)\).
On the other hand by the disjointness of the balls, by scaling, by \eqref{eq_aereeXahjuul9kieGhaesoo6} and by \eqref{eq_Xo3ye9aesaev2choph6ohzoy}, we have 
\begin{equation*}
\begin{split}
 \smashoperator{\iint_{\substack{(x, y) \in \Bset^m \times \Bset^m\\ \abs{f (y) - f (x)} \ge 1}}} \frac{\abs{f (y) - f (x)}^{q}}{\abs{y - x}^{m + q}} \dif y \dif x 
 &\ge \sum_{j \in \Nset} \rho_j^{m - q}
 \smashoperator{\iint_{\substack{(x, y) \in \Bset^m \times \Bset^m\\ \abs{\lambda_j \psi (y) - \lambda_j \psi (x)} \ge 1}}} \frac{\abs{\lambda_j \psi (y) - \lambda_j \psi (x)}^{q}}{\abs{y - x}^{m + q}} \dif y \dif x \\
 &\ge \sum_{j \in \Nset}\Cr{cst_Eet3ahz0rooquei3phoh8aib} = \infty.\qedhere
 \end{split}
\end{equation*}
\end{proof}

\section{Estimate of the lifting in subcritical dimension}

This section is devoted to the proof of \cref{theorem_estimate_lifting_noncompact}.
We first observe that by \cref{theorem_lifting_truncated_noncompact_apriori}, for every \(\delta \in \intvo{0}{\infty}\), 
the map \(\lifting{u} : \manifold{M} \to \lifting{\manifold{N}}\) immediately satisfies the \emph{small-scale estimate}
\begin{equation*}
 \smashoperator{\iint_{\substack{(x, y) \in \manifold{M} \times \manifold{M}\\d_{\lifting{\manifold{N}}} (\lifting{u}(y), \lifting{u}(x))\le \delta}}} \frac{d_{\lifting{\manifold{N}}} (\lifting{u} (y), \lifting{u} (x))^p}{d_{\manifold{M}} (y, x)^{m + sp}} \dif y \dif x
 \le  C \smashoperator{\iint_{\manifold{M} \times \manifold{M}}} \frac{d_{\manifold{N}} (u (y), u (x))^p}{d_{\manifold{M}} (y, x)^{m + sp}} \dif y \dif x,
\end{equation*}
so that it will be sufficient to estimate the \emph{large-scale integral}:
\begin{equation}
 \smashoperator{\iint_{\substack{(x, y) \in \manifold{M} \times \manifold{M}\\d_{\lifting{\manifold{N}}} (\lifting{u}(y), \lifting{u}(x))> \delta}}} \frac{d_{\lifting{\manifold{N}}} (\lifting{u} (y), \lifting{u} (x))^p}{d_{\manifold{M}} (y, x)^{m + sp}} \dif y \dif x.
\end{equation}
We will prove the following counterpart of \cref{proposition_spgt1_estimate} for \emph{large-scale oscillations}.

\begin{proposition}
\label{proposition_large_oscillation_estimate}
Let \(m \in \Nset \setminus \set{0}\), let \(s, s_* \in \intvo{0}{1}\) and let \(p, p_* \in \intvr{1}{\infty}\).
If \(s p > 1\), then there exists a constant \(C \in \intvo{0}{\infty}\) such that if \(\pi : \lifting{\manifold{N}} \to \manifold{N}\) is a Riemannian covering, if \(\Omega \subseteq \Rset^m\) is open and convex, if \(\lifting{u} \in Y (\Omega, \lifting{\manifold{N}})\), if \(u \defeq \pi \compose \lifting{u}\), if \(\delta \le \inj (\manifold{N})\), and if 
\begin{equation}
\label{eq_Foreigie4zaiw2ze0dah8nah}
\frac{1 - s_*}{m} = \frac{1}{sp} - \frac{1}{p_*}
,
\end{equation}
then  
\begin{equation}
\label{eq_aeghula1ao9kiJath8ais2ie}
\smashoperator[r]{\iint_{\substack{(x, y) \in \Omega \times \Omega\\
 d_{\lifting{\manifold{N}}}(\lifting{u} (y), \lifting{u} (x))\ge \delta}}} 
 \frac{d_{\lifting{\manifold{N}}} (\lifting{u} (y), \lifting{u} (x))^{p_*}}{\abs{y - x}^{m + s_*p_*}} \dif y \dif x 
 \\
 \le 
 C
 \brk[\Bigg]{
 \frac{1}{\delta^{(1 - s)p}}
 \smashoperator{\iint_{\Omega\times \Omega}} \frac{d_{\manifold{N}} (u (y), u (x))^{p}}{\abs{y - x}^{m + s p}}  \dif y \dif x
 }^\frac{p_*}{sp}. 
\end{equation}
\end{proposition}

We recall that the space \(Y (\Omega, \lifting{\manifold{N}})\) was defined in \eqref{eq_definition_Y} as the set of maps whose restriction on almost every segment coincides almost everywhere with a continuous function taking the same value at the extremities.

\begin{remark}
\label{remark_Sobolev_embedding}
\Cref{proposition_large_oscillation_estimate} implies a fractional Sobolev embedding: 
for \(s_*\in \intvo{0}{1}\) and \(p, p_* \in \intvo{1}{\infty}\) such that \(1/p_* = 1/p - (1-s_*)/m\), 
letting \(\pi : \Rset \to \Sset^1\) be the universal covering of the circle and choosing \(\lifting{u} \defeq t f\) for \(t > 0\) in \eqref{eq_aeghula1ao9kiJath8ais2ie} with \(\delta = 1\), one gets by the fractional Gagliardo--Nirenberg interpolation inequality, since \(\abs{e^{itf}} \le 1\) in \(\Omega\),
\begin{equation}
\label{eq_teTeizo0Map9ieMediBohjei}
\begin{split}
\smashoperator[r]{\iint_{\substack{(x, y) \in \Omega \times \Omega\\
 \abs{f (y) - f (x)}\ge t }}} \frac{\abs{f (y) - f (x)}^{p_*}}{\abs{y - x}^{m + s_*p_*}} \dif y \dif x 
 &\le 
 \frac{\C}{t^{p_*}}
 \brk[\Bigg]{\;\smashoperator[r]{\iint_{\Omega \times \Omega}} \frac{\abs{e^{i t f (y)} - e^{i t f(x)}}^{p/s_*}}{\abs{y - x}^{m + p}}  \dif y \dif x}^{p_*/p}\\[-1em]
&\le \frac{\Cl{cst_ziezeiph1eiQuo7eiH6Cohdi}}{t^{p_*}}
 \brk[\bigg]{\int_{\Omega} \abs{D e^{i tf}}^p}^{p_*/p}
 = \Cr{cst_ziezeiph1eiQuo7eiH6Cohdi}
 \brk[\bigg]{\int_{\Omega} \abs{D f}^p}^{p_*/p};
\end{split}
\end{equation}
letting \(t \to 0\) in \eqref{eq_teTeizo0Map9ieMediBohjei}, one gets the fractional Sobolev embedding 
\begin{equation*}
\smashoperator[r]{\iint_{\substack{\Omega \times \Omega}}} \frac{\abs{f (y) - f (x)}^{p_*}}{\abs{y - x}^{m + s_*p_*}} \dif y \dif x 
 \le \Cr{cst_ziezeiph1eiQuo7eiH6Cohdi}
 \brk[\bigg]{\int_{\Omega} \abs{D f}^p}^{p_*/p}.
\end{equation*}
\end{remark}

\subsection{One-dimensional estimates}

Our first tool towards the proof of \cref{proposition_large_oscillation_estimate} is the following truncated fractional Morrey--Sobolev inequality.

\begin{lemma}
\label{lemma_truncated_Morrey}
Let \(s \in \intvo{0}{1}\) and let \(p \in \intvr{1}{\infty}\). 
If \(sp > 1\), then there exists a constant \(C \in \intvo{0}{\infty}\) such that if \(I \subseteq \Rset\) is an interval, if \(\manifold{N}\) is a Riemannian manifold, if the mapping \(u : I \to \manifold{N}\) is measurable and if \(\mu \in \intvr{0}{\infty}\), then for almost every \(x, y \in I\), we have
\begin{multline}
\label{eq_aeNg6Aechie0Aivohchaijee}
 d_{\manifold{N}} (u (y), u (x)) 
 \\[-.5em]
 \le C
 \brk[\Bigg]{
\brk[\Bigg]{\ \smashoperator[r]{\iint_{\intvc{x}{y}\times \intvc{x}{y}}} \brk[\bigg]{\frac{d_{\manifold{N}} (u (w), u (v))}{\abs{w - v}^s} - \mu}_+^p \frac{\dif w \dif v}{\abs{w - v}} 
}^\frac{1}{p}  \abs{y - x}^{s  - \frac{1}{p}} + \mu \abs{y - x}^{s} }.
\end{multline}
\end{lemma}
When \(\mu = 0\), the estimate \eqref{eq_aeNg6Aechie0Aivohchaijee} reduces to the fractional Morrey--Sobolev inequality; when \(\mu > 0\), \eqref{eq_aeNg6Aechie0Aivohchaijee} shows that when small values of the difference quotient are removed, one still gets some truncated uniform bound.

The proof will use the following Minkowski inequality for mean oscillations.

\begin{lemma}
\label{lemma_Minkowski_average}
Let  \(m \in \Nset \setminus\set{0}\), let \(p \in \intvr{1}{+\infty}\), let \(\Omega \subseteq \Rset^m\) be measurable and let the mapping \(u : \Omega \to \manifold{N}\) be measurable.
For every \(k \in \Nset\) and measurable sets \(A_0, \dotsc, A_k \subseteq \Omega\) such that for every \(j \in \set{0, \dotsc, k}\), \(\mathcal{L}^d (A_j) > 0\), one has 
\begin{equation}
\label{eq_ceejee4igood8ChiVah7ooxa}
  \brk[\bigg]{\fint_{A_0} \fint_{A_k} d_{\manifold{N}} (u (y), u (x))^p \dif y \dif x}^{\frac{1}{p}}
  \le \sum_{j = 0}^{k - 1} \brk[\bigg]{\fint_{A_j} \fint_{A_{j + 1}} d_{\manifold{N}} (u (y), u (x))^p \dif y \dif x}^\frac{1}{p}.
\end{equation}
\end{lemma}
\begin{proof}
We have by the triangle inequality and by Minkowski's inequality
\begin{equation*}
\begin{split}
 \brk[\bigg]{\fint_{A_0} \fint_{A_k} d_{\manifold{N}} (u (y), u (x))^p \dif y \dif x}^\frac{1}{p}
 & = \brk[\bigg]{\fint_{A_0} \dotsi \fint_{A_k} d_\manifold{N} (u (x_k),u (x_0))^p \dif x_k \, \dotsm \dif x_0}^\frac{1}{p}\\
 & \le \sum_{j = 0}^{k - 1} \brk[\bigg]{\fint_{A_0} \dotsi \fint_{A_k} d_{\manifold{N}} (u (x_{j + 1}), u (x_j))^p \dif x_k \, \dotsm \dif x_0}^\frac{1}{p}\\
 & = \sum_{j = 0}^{k - 1} \brk[\bigg]{\fint_{A_j} \fint_{A_{j + 1}} d_{\manifold{N}}(u (y), u (x))^p \dif y \dif x}^\frac{1}{p},
\end{split}
\end{equation*}
which proves \eqref{eq_ceejee4igood8ChiVah7ooxa}.
\end{proof}

We now prove the truncated fractional Morrey--Sobolev embedding.

\begin{proof}[Proof of \cref{lemma_truncated_Morrey}]
\resetconstant
Since the mapping \(u\) is measurable, we can assume without loss of generality that \(x\) and \(y\) are Lebesgue points of \(u\) and that \(I = \intvo{x}{y}\).
We define for each \(j \in \Nset\) the set \(I^x_j \defeq  x + 2^{-j} (I - x) \subseteq I\).
Since \(x\) is a Lebesgue point of \(u\), we have
\begin{equation}
\label{eq_fi3Aijiap5shoh4yial1un6u}
\lim_{j \to \infty} \fint_{I_j} d_{\manifold{N}} (u (x), u (z))^p \dif z = 0,
\end{equation}
and then by \eqref{eq_fi3Aijiap5shoh4yial1un6u}, by \cref{lemma_Minkowski_average} and by Minkowski's inequality
\begin{equation}
\label{eq_eingo6Vod3Aaxeimeubeisha}
\begin{split}
    \brk[\bigg]{
      \fint_I 
        d_{\manifold{N}} (u (x), u (z))^p 
        \dif z
      }^\frac{1}{p} 
    &
    \le 
      \sum_{j = 0}^{\infty} 
        \brk[\bigg]{
          \fint_{I^x_j}  
          \fint_{I^x_{j + 1}} 
            d_{\manifold{N}} (u (w), u (v))^p 
          \dif w \dif v
          }^\frac{1}{p}
          \\
    &
    \le 
      \sum_{j = 0}^{\infty} 
      \brk[\bigg]{
        \fint_{I^x_j} \fint_{I^x_{j + 1}} 
          \brk[\big]{d_{\manifold{N}} (u (w), u (v)) - \mu \abs{w - v}^s}_+^p 
          \dif w \dif v
        }^\frac{1}{p}
        \\
    & 
    \qquad 
    +
      \sum_{j = 0}^{\infty} 
        \brk[\bigg]{
          \fint_{I^x_j} 
          \fint_{I^x_{j + 1}} 
            \mu^p 
            \abs{w - v}^{sp}
            \dif w 
            \dif v
          }^\frac{1}{p}
          .
\end{split}
\end{equation}
For the first term in the right-hand side of \eqref{eq_eingo6Vod3Aaxeimeubeisha}, we have for every \(j \in \Nset\), since \(sp \ge 1\),
\begin{equation}
\label{eq_Thix0kuwai1gochi3fameeka}
\begin{split}
\fint_{I^x_j}  \fint_{I^x_{j + 1}}& \brk[\big]{d_{\manifold{N}} (u (w), u (v)) - \mu \abs{w - v}^s}_+^p \dif w \dif v\\
    &\le  
    \frac{2 \diam(I)^{sp - 1}}{ 2^{j(sp -1)}} 
    \smashoperator{\iint_{I \times I}}
    \brk[\Big]{ \frac{d_{\manifold{N}} (u (w), u (v))}{\abs{w - v}^s} - \mu}_+^p \frac{\dif w \dif v}{\abs{w - v}},
\end{split}
\end{equation}
while for the second term  in the right-hand side of \eqref{eq_eingo6Vod3Aaxeimeubeisha}, we have for every \(j \in \Nset\),  
\begin{equation}
\label{eq_XeFau1Oothaig8eg3cahBu2e}
\fint_{I^x_j} 
    \fint_{I^x_{j + 1}} \abs{w - v}^{sp} \dif w \dif v
    \le \frac{\C  \diam (I)^{sp}}{2^{j{sp}}}.
\end{equation}
Inserting \eqref{eq_Thix0kuwai1gochi3fameeka} and \eqref{eq_XeFau1Oothaig8eg3cahBu2e} into \eqref{eq_eingo6Vod3Aaxeimeubeisha}, we get, since \(sp > 1\),
\begin{multline}
\label{eq_eiT2Izi4ahciH1ue4yaiquoo}
    \brk[\Big]{\fint_I d_{\manifold{N}} (u (x), u (z))^p \dif z}^\frac{1}{p} \\
    \le 
     \C
    \brk[\Big]{\diam (I)^{sp - 1}
    \smashoperator{\iint_{I \times I}}
    \brk[\bigg]{ \frac{d_{\manifold{N}} (u (w), u (v))}{\abs{w - v}^s} - \mu}_+^p \frac{\dif w \dif v}{\abs{w - v}}
    + \mu^p \diam (I)^{sp}}^\frac{1}{p},
\end{multline}
and the conclusion \eqref{eq_aeNg6Aechie0Aivohchaijee} follows from \eqref{eq_eiT2Izi4ahciH1ue4yaiquoo} and the triangle inequality.
\end{proof}

Next we use \cref{lemma_truncated_Morrey} to estimate the large scale oscillations of a lifting by a truncated fractional Sobolev semi-norm.

\begin{lemma}
\label{lemma_large_scale_1d_estimate}
Let \(s \in \intvo{0}{1}\) and let \(p \in \intvr{1}{\infty}\).
If \(sp>1\), 
then there exists a constant \(C \in \intvo{0}{\infty}\) such that 
if \(I \subseteq \Rset\) is an interval, 
if \(\pi : \lifting{\manifold{N}} \to \manifold{N}\) is a Riemannian covering, 
if \(\lifting{u} \in C (I, \lifting{\manifold{N}})\) 
and if \(u \defeq \pi \compose \lifting{u}\), 
then for almost every \(x, y \in I\), 
every \(\mu \in \intvr{0}{\infty}\)
and every \(\delta \in \intvr{0}{\inj (\manifold{N})}\), one has
\begin{multline}
\label{eq_iesh4ejiotahn2Ieb3oo3oog}
 \brk[\big]{
  d_{\lifting{\manifold{N}}} (\lifting{u} (y), \lifting{u} (x)) -\delta
  }_+^{sp}
 \\
 \le  
 \frac
  {C}
  {\delta^{p (1 - s)}} 
  \brk[\bigg]{
  \smashoperator[r]{
    \iint_{\intvc{x}{y}\times \intvc{x}{y}}}
        \brk[\bigg]{
          \frac
            {d_{\manifold{N}} (u (w), u (v))}
            {\abs{w - v}^s} 
          -
          \mu 
          }_+^p
        \frac
          {\dif w \dif v}
          {\abs{w - v}}
        \abs{y - x}^{sp - 1} 
      + 
        \mu^p 
        \abs{y - x}^{sp}}.
\end{multline}
\end{lemma}

\Cref{lemma_large_scale_1d_estimate} gives a growth estimate corresponding to what the Morrey--Sobolev embedding would give if one had \(\lifting{u} \in \homog{W}^{1, sp} (I, \lifting{\manifold{N}})\).

When \(\mu = 0\), \cref{lemma_large_scale_1d_estimate} shows that on large scale the lifting \(\lifting{u}\) behaves like a Hölder--continuous mapping of exponent \(1 - 1/sp\), which is not as good as the exponent \(s - 1/p\) that the fractional Morrey--Sobolev embedding gives on the original function \(u\);
this generalizes the results obtained for the universal covering of the circle by Merlet \cite{Merlet_2006} and Mironescu \& Molnar \cite{Mironescu_Molnar_2015}*{lemma 8.25}.

\begin{proof}%
[Proof of \cref{lemma_large_scale_1d_estimate}]%
\resetconstant
Let \(\ell \defeq \floor{d_{\lifting{\manifold{N}}} (\lifting{u} (x), \lifting{u} (y))/\delta}\), so that 
\begin{equation}
\label{eq_oe7baphee3voV2eet8chiy8w}
 (d_{\lifting{\manifold{N}}} (\lifting{u} (y), \lifting{u} (x)) -\delta)_+ \le \ell \delta. 
\end{equation}
Since the mapping \(\lifting{u}\) is continuous, by the intermediate value theorem, there exist points \(z_0 = x \le z_1 \le z_2 \le \dotsb \le z_\ell \le y\) such that for every \(i \in \set{1, \dotsc, \ell}\), one has
\(
 d_{\lifting{\manifold{N}}}(\lifting{u} (z_i), \lifting{u} (z_{i - 1})) = \delta\).
 Since \(\delta \le \inj(\manifold{N})\), by \cref{lemma_small_isometry},
 we also have \(d_{\manifold{N}}(u (z_i), u (z_{i - 1})) = \delta\).
Therefore, since \(sp > 1\), it follows from \cref{lemma_truncated_Morrey} that for each \(i \in \set{1, \dotsc,\ell}\),
\begin{equation}
\label{eq_EXah3ooN7ohlain1quiej5az}
\delta \le 
\Cl{cst_sohph6Yainahlu7ain3fei7u} \brk[\Bigg]{\brk[\bigg]{\smashoperator[r]{\iint_{\intvc{z_{i - 1}}{z_i}^2}}
 \brk[\bigg]{\frac{d_{\manifold{N}} (u (z), u (w))}{\abs{z - w}^{s}} - \mu}_+^p
 \frac{\dif z \dif w}{\abs{z - w}}\abs{z_i - z_{i-1}}^{sp - 1}}^\frac{1}{p} + \mu \abs{z_i - z_{i - 1}}^s} .
\end{equation}
Summing \eqref{eq_EXah3ooN7ohlain1quiej5az} we have
\begin{equation}
\label{eq_ibie1ooceeJah6Iu7thooGiu}
\ell 
\le \brk[\bigg]{\frac{\Cr{cst_sohph6Yainahlu7ain3fei7u}}{\delta}}^\frac{1}{s}\sum_{i = 1}^\ell \brk[\Bigg]{\brk[\bigg]{\smashoperator[r]{\iint_{\intvc{z_{i - 1}}{z_i}^2}}\brk[\bigg]{\frac{d_{\manifold{N}} (u (z), u (w))}{\abs{z - w}^{s}} - \mu}_+^p
 \frac{\dif z \dif w}{\abs{z - w}}}^\frac{1}{sp}\abs{z_i - z_{i-1}}^{1 - \frac{1}{sp}} + \mu^\frac{1}{s} \abs{z_i - z_{i - 1}}}.
\end{equation}
Applying the discrete Hölder inequality to the right-hand side of  \eqref{eq_ibie1ooceeJah6Iu7thooGiu}, we get 
\begin{equation}
\label{eq_xi7Ohg9Ahph9ahVaiS7Veegh}
\begin{split}
 \ell &\le\brk[\Bigg]{ \frac{\Cr{cst_sohph6Yainahlu7ain3fei7u}}{\delta}}^\frac{1}{s}\brk[\Bigg]{\brk[\bigg]{\sum_{i = 1}^\ell\smashoperator[r]{\iint_{\intvc{z_{i - 1}}{z_i}^2}}\brk[\bigg]{\frac{d_{\manifold{N}} (u (z), u (w))}{\abs{z - w}^{s}} - \mu}_+^p
 \frac{\dif z \dif w}{\abs{z - w}}}^\frac{1}{sp}
 \brk[\bigg]{\sum_{i = 1}^\ell\abs{z_i - z_{i-1}}}^{1 - \frac{1}{sp}} \\
 &\hspace{8cm}+ \mu^\frac{1}{s} \sum_{i = 1}^\ell\abs{z_i - z_{i-1}}}.
  \end{split}
  \end{equation}
Since \(x \le z_0  \le z_1\le \dotsb \le z_\ell \le y\), the sets \(\intvo{z_{i - 1}}{z_i}^2\) are disjoint subsets of \(\intvc{x}{y}^2\) and we deduce from \eqref{eq_xi7Ohg9Ahph9ahVaiS7Veegh} that 
\begin{equation}
\label{eq_oquovu6iephei0jei7thai4Y}
 \ell \le \brk[\Bigg]{\frac{\Cr{cst_sohph6Yainahlu7ain3fei7u}}{\delta}}^\frac{1}{s}\brk[\bigg]{\brk[\bigg]{\smashoperator[r]{\iint_{\intvc{x}{y}^2}}\brk[\bigg]{\frac{d_{\manifold{N}} (u (z), u (w))}{\abs{z - w}^{s}} - \mu}_+^p
 \frac{\dif z \dif w}{\abs{z - w}}}^\frac{1}{sp}
 \abs{y - x}^{1 - \frac{1}{sp}}+ \mu^\frac{1}{s} \abs{y - x}}.
\end{equation}
Recalling \eqref{eq_oe7baphee3voV2eet8chiy8w}, the conclusion \eqref{eq_iesh4ejiotahn2Ieb3oo3oog} follows from \eqref{eq_oquovu6iephei0jei7thai4Y}.
\end{proof}

\subsection{Mean integral oscillation estimates}
Integrating the estimate of \cref{lemma_large_scale_1d_estimate}, we will obtain the following estimate on truncated mean oscillation by a truncated fractional Sobolev norm.

\begin{lemma}
\label{lemma_oscillation_estimate}
Let \(m \in \Nset \setminus \set{0}\), let \(s \in \intvo{0}{1}\) and let \(p \in \intvo{1}{\infty}\). 
If \(sp > 1\), then there exists a constant \(C \in \intvo{0}{\infty}\) such that 
if the set \(\Omega\subset\Rset^m\) is bounded and convex, 
if \(\pi : \lifting{\manifold{N}} \to \manifold{N}\) is a Riemannian covering map, 
if \(\lifting{u} \in Y (\Omega, \lifting{\manifold{N}})\), 
if \(u \defeq \pi \compose \lifting{u}\), 
if \(\delta \in \intvl{0}{\inj (\manifold{N})}\),
and if \(\mu \in \intvr{0}{\infty}\),
then 
\begin{multline}
\label{eq_ca0nuovae4eevemie5Zae1ne}
 \smashoperator{\iint_{\Omega \times \Omega}}
\brk[\big]{d_{\lifting{\manifold{N}}} (\lifting{u} (y), \lifting{u} (z)) - \delta}_+^{s p}  
 \dif y 
 \dif z
 \\
 \le \frac{\C}{\delta^{(1 - s)p}} \brk[\Bigg]{\frac{\diam (\Omega)^{m + sp}}{ m + sp } \smashoperator{\iint_{\Omega\times \Omega}} \brk[\bigg]{\frac{d_{\manifold{N}} (u (y), u (x))}{\abs{y - x}^{s}} - \mu}_+^p \frac{ \dif y \dif x}{\abs{y - x}^{m}} + \mu^p \diam (\Omega)^{2m + sp}}.
\end{multline}
\end{lemma}

\Cref{lemma_oscillation_estimate} will be deduced from \cref{lemma_oscillation_estimate} and the next integral estimate.

\begin{lemma}
\label{lemma_integration_2}
Let \(m \in \Nset \setminus \set{0}\).
If the set \(\Omega \subseteq \Rset^m\) is open and convex, 
if the function \(F : \Omega \times \Omega \to \intvr{0}{\infty} \) is measurable,
and if \(\gamma > - m\), then 
\begin{equation}
\label{eq_Taxeiyoo8ziephah7bienee3}
 \smashoperator{\iint_{\Omega\times \Omega}}\brk[\Bigg]{\;
 \smashoperator[r]{\iint_{\intvc{x}{y} \times \intvc{x}{y}}} F (w, v)  \dif w \dif v} 
 \frac{\dif y \dif x}{\abs{y - x}^{1-\gamma}}
 \le 
  \frac{2 \diam (\Omega)^{m + \gamma}}{ m + \gamma } \smashoperator{\iint_{\Omega\times \Omega}}
  \frac{F(x, y)}{\abs{y - x}^{m - 1}} \dif y \dif x.
\end{equation}
\end{lemma}
\begin{proof}
We have by definition of integral on a segment
\begin{multline}
\label{eq_shouP1leethous4Rour7aiph}
\iint\limits_{\Omega\times \Omega}\brk[\Bigg]{\;
 \smashoperator[r]{\iint_{\intvc{x}{y}\times \intvc{x}{y}}} F (w, v)  \dif w \dif v} \abs{y - x}^{\gamma - 1} \dif y \dif x\\
 = \smashoperator[r]{\iint_{\Omega\times \Omega}}\;
 \smashoperator[r]{\iint_{\intvc{0}{1}\times \intvc{0}{1}}} F ((1 -t) x + t y, (1-r) x + r y) \abs{y - x}^{\gamma + 1} \dif t \dif r \dif y \dif x.
\end{multline}
By the change of variable \(v = (1-r) x + r y\), \(w = (1 - t) x + t y\) in the right-hand side of \eqref{eq_shouP1leethous4Rour7aiph}, we obtain, since \(\abs{v - w} = \abs{t - r} \abs{y - x}\)
\begin{equation}
\label{eq_Aethajahfo4wal4iTh6naiN7} 
\iint\limits_{\Omega\times \Omega}\brk[\Bigg]{\;
 \smashoperator[r]{\iint_{\intvc{x}{y}\times \intvc{x}{y}}} F (w, v)  \dif w \dif v} \frac{\dif y \dif x}{\abs{y - x}^{1-\gamma}}
 = \smashoperator[l]{\iint\limits_{\Omega\times \Omega}}
 \iint\limits_{\Sigma_{v, w}} \frac{F (w, v) \abs{w - v}^{\gamma + 1}}{\abs{t - r}^{ m + \gamma + 1}} \dif t \dif r \dif w \dif v,
\end{equation}
where we have defined for every \(v, w \in \Omega\) the set
\[
 \Sigma_{v, w} \defeq \set[\bigg]{(t, r) \in \intvc{0}{1} \times \intvc{0}{1} \st
 \frac{r v - tw}{r - t} \in \Omega \text{ and }
 \frac{(1 - r)v - (1 - t)w}{t - r } \in \Omega }.
\]
Since 
\[
 \Sigma_{v, w} \subseteq \set[\bigg]{(t, r ) \in \intvc{0}{1} \times \intvc{0}{1}\st
 \abs{t - r} \ge \frac{\abs{w - v}}{\diam \Omega}},
\]
we have 
\begin{equation}
\label{eq_EiraisieVeibeemia8tah9mi}
 \iint\limits_{\Sigma_{v, w}} \frac{1}{\abs{t - r}^{ m + \gamma + 1}} \dif t \dif r 
 \le \smashoperator[r]{\int_{\abs{s} \ge  \frac{\abs{w - v}}{\diam \Omega}}} \frac{\dif s}{\abs{s}^{ m + \gamma + 1}}
 = \frac{2 \diam (\Omega)^{m + \gamma}}{(m + \gamma)\abs{w - v}^{m + \gamma }}.
\end{equation}
and we deduce from \eqref{eq_Aethajahfo4wal4iTh6naiN7} and \eqref{eq_EiraisieVeibeemia8tah9mi} that \eqref{eq_Taxeiyoo8ziephah7bienee3} holds.
\end{proof}

We proceed now to the proof of \cref{lemma_oscillation_estimate}.

\begin{proof}%
[Proof of \cref{lemma_oscillation_estimate}]%
\resetconstant
We have by \cref{lemma_large_scale_1d_estimate} since \(sp > 1\),
\begin{multline}
\label{eq_poosaeng4dequiaGom4aichi}
 \smashoperator[r]{\iint_{\Omega \times \Omega}} 
 \brk[\big]{
 d_{\lifting{\manifold{N}}} (\lifting{u} (y), \lifting{u} (x)) - \delta}_+^{sp}  
 \dif y 
 \dif x
 \\
 \le 
 \frac{\C}{\delta^{(1 -s)p}}
 \biggl( 
    \smashoperator[r]{\iint_{\Omega \times \Omega }}\;
      \smashoperator[r]{\iint_{\intvc{x}{y} \times \intvc{x}{y}}} \brk[\bigg]{\frac{d_{\manifold{N}} (u (w), u (v))}{\abs{w - v}^{s}} - \mu}_+^p 
      \frac{\dif w \dif v}{\abs{w - v}} \abs{y - x}^{s p - 1}  
    \dif y 
    \dif x
    \\
    + 
    \mu^p 
    \smashoperator{\iint_{\Omega \times \Omega}}  
    \abs{y - x}^{sp} 
    \dif y 
    \dif x 
 \biggr).
\end{multline} 
For the first term in the right-hand side of \eqref{eq_poosaeng4dequiaGom4aichi},
we proceed by \cref{lemma_integration_2} to infer from \eqref{eq_poosaeng4dequiaGom4aichi}, since \(sp > - m\) that 
\begin{equation}
\label{eq_wahqu6ahrai5eizieSheyeeR}
\begin{split}
 \smashoperator[r]{\iint_{\Omega \times \Omega }}\;
 \smashoperator[r]{\iint_{\intvc{x}{y}\times \intvc{x}{y}}}& \brk[\bigg]{\frac{d_{\manifold{N}} (u (w), u (v))}{\abs{w - v}^{s}} - \mu}_+^p \frac{\dif w \dif v}{\abs{w - v}} \abs{y - x}^{s p - 1}  \dif y \dif x\\
 &\le \frac{2 \diam (\Omega)^{m + sp}}{ m + sp } \smashoperator{\iint_{\Omega\times \Omega}} \brk[\bigg]{\frac{d_{\manifold{N}} (u (y), u (x))}{\abs{y - x}^{s}} - \mu}_+^p \frac{ \dif y \dif x}{\abs{y - x}^{m}},
\end{split}
\end{equation}
whereas for the second term in the right-hand side of \eqref{eq_poosaeng4dequiaGom4aichi} we have 
\begin{equation}
\label{eq_iyui7Zaesae3Ooh9iheerae3}
 \smashoperator{\iint_{\Omega \times \Omega}} \abs{y - x}^{sp} \dif y \dif x
 \le \C \diam (\Omega)^{2m + sp}.
\end{equation}
The estimate \eqref{eq_ca0nuovae4eevemie5Zae1ne} then follows from the inequalities \eqref{eq_poosaeng4dequiaGom4aichi}, \eqref{eq_wahqu6ahrai5eizieSheyeeR} and \eqref{eq_iyui7Zaesae3Ooh9iheerae3}.
\end{proof}

\subsection{Integral truncated mean oscillation estimate}
We now obtain an interpolation estimate similar to \cref{proposition_large_oscillation_estimate} on an integral of truncated mean oscillations.

\begin{proposition}
\label{lemma_single_scale_large_oscillation_estimate}
Let \(\manifold{M}\) be a compact Riemannian manifold, 
let \(s, s_* \in \intvo{0}{1}\), 
let \(p, p_* \in \intvr{1}{\infty}\), 
and let \(m \defeq \dim \manifold{M}\).
If \(s p > 1\), 
then 
there exists a constant \(C \in \intvo{0}{\infty}\) such that if \(\pi : \lifting{\manifold{N}} \to \manifold{N}\) is a Riemannian covering, 
if \(\lifting{u} \in Y(\manifold{M}, \lifting{\manifold{N}})\), 
if \(u \defeq \pi \compose \lifting{u}\), 
if \(\delta \le \inj (\manifold{N})\) 
and if 
\begin{equation}
\label{eq_roo9Eep4EeXookohxahk6ieW}
 \frac{1 - s_* }{m}
 = \frac{1}{sp} - \frac{1}{p_*},
\end{equation}
then 
\begin{multline}
  \int_{\Omega} 
  \int_0^{\diam \Omega}
    \brk[\bigg]{
  \frac
    {1}
    {r^{2m + s_*}}
    \smashoperator{
      \iint_{(\Omega \cap B_{r} (x))^2}
    }
    \brk[\big]{
      d_{\lifting{\manifold{N}}} (\lifting{u} (y), \lifting{u} (z))-\delta
      }_+
    \dif y
    \dif z
    }^{p_*}
    \frac
      {\dif r}
      {r} 
    \dif x
    \\[-1em]
 \le 
 C
 \brk[\bigg]{
  \frac{1}{\delta^{(1 - s)p}}
    \smashoperator{
      \iint_{\Omega \times \Omega}
      } 
      \frac
        {d_{\manifold{N}} (u (y), u (z))^{p}}
        {\abs{y - z}^{m + s p}}  
      \dif y
      \dif z
    }^\frac{p_*}{sp}.
\end{multline}
\end{proposition}

The proof of \cref{lemma_single_scale_large_oscillation_estimate} is reminiscent of the proof of the  Marcinkiewicz real interpolation theorem, although the framework here is much more nonlinear.

\begin{proof}[Proof of \cref{lemma_single_scale_large_oscillation_estimate}]
\resetconstant
We have, by the layer cake representation of integrals (Cavalieri's principle),
\begin{multline}
\label{eq_baeWu5quoh6ziexaiph0ieb3}
  \int_{\Omega} 
    \brk[\bigg]{\frac{1}{r^{2m}}
      \iint_{(\Omega \cap B_{r} (x))^2}
        \brk[\big]{
          d_{\lifting{\manifold{N}}} (\lifting{u} (y), \lifting{u} (z))- \delta
          }_+  
      \dif y 
      \dif z
    }^{p_*} 
    \dif x
    \\[-1em]
  = 
    (p_* - 1)
    \int_{0}^\infty 
      \mathcal{L}^m (E^r_\lambda) 
      \lambda^{p^* - 1} 
    \dif \lambda,
\end{multline}
where for each \(\lambda \in \intvo{0}{\infty}\) and \(r \in \intvo{0}{\infty}\) we have defined the set 
\[
 E^r_\lambda
 \defeq 
 \set[\Bigg]
    {x \in \Omega 
    \st  
      \smashoperator[r]{
        \iint_{(\Omega \cap B_{r} (x))^2}
      }
    \brk[\big]
      {
        d_{\lifting{\manifold{N}}} (\lifting{u} (y), \lifting{u} (z))-\delta
      }_+ 
      \dif y 
      \dif z
  \ge 
    \lambda r^{2m} 
    },
\]

On the one hand, fixing \(q \in \intvo{\frac{1}{s}}{p}\) --- which is possible since \(s p > 1\) ---
for each \(x \in E^r_\lambda\) and \(\mu \in \intvr{0}{\infty}\), 
we have by Jensen's inequality and by \cref{lemma_oscillation_estimate}, since \(s q > 1\),
\begin{equation}
\label{eq_oosahjoe6oph4ziu8yaiS2mo}
\begin{split}
  \lambda^{sq}
  &
  \le 
    \brk[\bigg]{
      \frac{1}{r^{2m}} 
      \smashoperator{
        \iint_{(\Omega \cap B_{r} (x))^2}
      }
      \brk[\big]{
        d_{\lifting{\manifold{N}}} (\lifting{u} (y), \lifting{u} (z))- \delta}_+ 
      \dif y 
      \dif w
    }^{s q} 
    \\
  & 
    \le 
    \C 
    \frac{1}{r^{2m}} 
    \smashoperator{
        \iint_{(\Omega \cap B_{r} (x))^2}
      }
      \brk[\big]{
        d_{\lifting{\manifold{N}}} (\lifting{u} (y), \lifting{u} (z))- \delta
        }_+^{s q} 
      \dif y 
      \dif z\\
 &
 \le  
 \C 
 \frac
  {r^{sq  - m}}
  {\delta^{(1 - s)q}} 
  \smashoperator{
    \iint_{(\Omega \cap B_{r} (x))^2}}
    \brk[\bigg]{
      \frac
        {d_{\manifold{N}} (u (y), u (z))}
        {\abs{y - z}^{s}} - \mu }_+^{q} 
      \frac
        {\dif y \dif z}
        {\abs{y - z}^m} 
  + 
    \Cl{cst_Ueghaik5eduzoi9yee4pi3hu} 
      \frac
        {\mu^q r^{sq}}
        {\delta^{(1-s)q}}
 ;
 \end{split}
\end{equation}
If we take now \(\mu\) to be given by
\[
 \mu^r_\lambda \defeq 
 \Cl{cst_booyakie9jae4eehahLeinoh}
 \frac{\lambda^{s} \delta^{1 - s}}{r^{s}},
\]
with \(\Cr{cst_booyakie9jae4eehahLeinoh}^q \Cr{cst_Ueghaik5eduzoi9yee4pi3hu}=  \frac{1}{2}\), for each \(x \in E^r_\lambda\), we have by \eqref{eq_oosahjoe6oph4ziu8yaiS2mo} 
\begin{equation}
\label{eq_eefoe3onei8mohf4eeKuZa2c}
  \lambda^{s q}
  \le \Cl{cst_wei7vax9eih0xet7Tingai8i} \frac{r^{s q - m}}{\delta^{(1 - s)q}} \smashoperator{\iint_{(\Omega \cap B_{r} (x))^2}}
 \brk[\bigg]{\frac{d_{\manifold{N}} (u (y), u (z))}{\abs{y - z}^{s}} - \mu_\lambda^r}_+^q \frac{\dif y \dif z}{\abs{y - z}^m}.
\end{equation}
Hence, we have by \eqref{eq_eefoe3onei8mohf4eeKuZa2c}
\begin{equation}
\label{eq_Co6jooko5uu2Aey6ail6chee}
\begin{split}
  \mathcal{L}^m (E^r_\lambda)
  &\le \frac{\Cr{cst_wei7vax9eih0xet7Tingai8i} r^{s q  - m}}{\lambda^{s q} \delta^{(1 - s)q}}
  \int_{\Omega} 
  \smashoperator[r]{\iint_{(\Omega \cap B_{r} (x))^2}}
 \brk[\bigg]{\frac{d_{\manifold{N}} (u (y), u (z))}{\abs{y - z}^{s}} - \mu_\lambda^r}_+^q \frac{\dif y \dif z}{\abs{y - z}^m}
   \dif x \\
  &= \frac{\Cr{cst_wei7vax9eih0xet7Tingai8i} r^{s q  - m}}{\lambda^{s q} \delta^{(1 - s)q}}
  \smashoperator{\iint_{\Omega \times \Omega}} 
   \brk[\bigg]{\frac{d_{\manifold{N}} (u (y), u (z))}{\abs{y - z}^{s}} - \mu_\lambda^r}_+^q 
   \mathcal{L}^m (\Omega \cap B_r (y) \cap B_r (z))
   \frac{\dif y \dif z}{\abs{y - z}^m}\\
  &\le \frac{\C r^{s q}}{\lambda^{s q} \delta^{(1 - s)q}}
    \smashoperator{\iint_{\Omega \times \Omega}}
 \brk[\bigg]{\frac{d_{\manifold{N}} (u (y), u (z))}{\abs{y - z}^{s}} - \mu_\lambda^r}_+^q \frac{\dif y \dif z}{\abs{y - z}^m}.
 \end{split}
\end{equation}

On the other hand, since \(s p > 1\), if \(x \in E^r_\lambda\), we have by Jensen's inequality and by \cref{lemma_oscillation_estimate} with \(\mu = 0\),  
\begin{equation}
\label{eq_Di2thaevaibosofiediezeiH}
\begin{split}
 \lambda^{s p}
 &\le \brk[\bigg]{\frac{1}{r^{2m }} \smashoperator{\iint_{(\Omega \cap B_{r} (x))^2}}
 \brk[\big]{d_{\lifting{\manifold{N}}} (\lifting{u} (y), \lifting{u} (z))- \delta}_+ \dif y \dif z}^{s p} \\
 & \le 
 \C 
 \frac{1}{r^{2m }} 
 \smashoperator{\iint_{(\Omega \cap B_{r} (x))^2}}
 \brk[\big]{d_{\lifting{\manifold{N}}} (\lifting{u} (y), \lifting{u} (z)) - \delta }_+^{s p} 
 \dif y \dif z\\
 &\le  
 \Cl{cst_ZiaLa3choohokiungaePaoCo} 
 \frac{r^{s p - m}}{\delta^{(1 - s) p}} 
 \smashoperator{\iint_{(\Omega \cap B_{r} (x))^2}}
 \frac{d_{\manifold{N}} (u (y), u (z))^p}{\abs{y - z}^{m + s p}} 
 \dif y 
 \dif z
 ;
 \end{split}
\end{equation}
it follows then from \eqref{eq_Di2thaevaibosofiediezeiH} that 
\begin{multline}
\label{eq_oc1ooFei2quahsahngureach}
\set{(r, \lambda)\in \intvo{0}{\infty}^2 \st E_r^\lambda \ne \emptyset}\\
\subseteq
 H \defeq \set[\bigg]{ (r, \lambda) \in \intvo{0}{\infty}^2\st 
 \lambda^{s p} r^{m - s p}
 \le \frac{\Cr{cst_ZiaLa3choohokiungaePaoCo}^{sp}}{\delta^{(1 - s) p}} 
 \smashoperator{\iint_{\Omega \times \Omega}}
 \frac{d_{\manifold{N}} (u (y), u (z))^p}{\abs{y - z}^{m + s p}} \dif y \dif z}.
\end{multline}

By \eqref{eq_baeWu5quoh6ziexaiph0ieb3}, \eqref{eq_Co6jooko5uu2Aey6ail6chee} and \eqref{eq_oc1ooFei2quahsahngureach}, we have 
\begin{multline}
\label{eq_gahz3shooth5Eshohquuriuc}
 \int_{\Omega} 
 \int_0^{\diam \Omega}
 \brk[\bigg]{\frac{1}{r^{2m}}
 \smashoperator{\iint_{(\Omega \cap B_{r} (x))^2}}
 \brk[\big]{d_{\lifting{\manifold{N}}} (\lifting{u} (y), \lifting{u} (z)) - \delta}_+  \dif y \dif z
 }^{p_*} \frac{\dif r}{r^{1 + s_* p_*}}\dif x\\
 \le \frac{\C}{ \delta^{(1 - s)q}}
 \iint\limits_{H}\smashoperator[r]{\iint_{\Omega \times \Omega}}
 \brk[\bigg]{\frac{d_{\manifold{N}} (u (y), u (z))}{\abs{y - z}^s} - \mu_\lambda^r}_+^q \frac{\dif y \dif z}{\abs{y - z}^m}
\frac{r^{s q}\lambda^{p_*}}{r^{1 + s_*p_*}\lambda^{1+s q}}\dif \lambda \dif r.
\end{multline}
Applying the change of variable 
\begin{align*}
 \mu &=  \Cr{cst_booyakie9jae4eehahLeinoh}
 \frac{\lambda^{s} \delta^{1 - s}}{r^{s}}&
 &\text{ and }&
 t =
 \lambda^{s p}r^{m - s p}
\end{align*}
in \eqref{eq_gahz3shooth5Eshohquuriuc},
we infer from  \eqref{eq_roo9Eep4EeXookohxahk6ieW} that 
\begin{multline}
\label{eq_maiZuch5eiS0wohd1ohlohGe}
\int_{\Omega}
 \int_{0}^{\diam (\Omega)} 
 \brk[\bigg]{\frac{1}{r^{2m}}
 \smashoperator{\iint_{(\Omega \cap B_{r} (x))^2}}
 \brk[\big]{d_{\lifting{\manifold{N}}} (\lifting{u} (y), \lifting{u} (z))- \delta}_+  \dif y \dif z
 }^{p_*} \frac{\dif r}{r^{1 + s_* p_*}}\dif x\\
 \le \frac{\C}{ \delta^{(1 - s)p}}
 \smashoperator{
 \iint_{\Omega \times \Omega}
 }
 \int_0^{\Bar{t}} 
 \int_0^\infty
 \brk[\bigg]{\frac{d_{\manifold{N}} (u (y), u (z))}{\abs{y - z}^s} - \mu}_+^q 
\mu^{p - q - 1} t^{\frac{p^*}{sp} - 2} \dif \mu \dif t\frac{\dif y \dif z}{\abs{y - z}^m},
\end{multline}
with
\[
 \Bar{t} \defeq \frac{\Cr{cst_ZiaLa3choohokiungaePaoCo}^{sp}}{\delta^{(1 - s)p}} \smashoperator{\iint_{\Omega \times \Omega}}
 \frac{d_{\manifold{N}} (u (y), u (z))^{p}}{\abs{y - z}^{m + s p}} \dif y \dif z.
\]
The conclusion follows by the integration in \(\mu\) and \(t\) of the right-hand side of the inequality \eqref{eq_maiZuch5eiS0wohd1ohlohGe}, since \(q < p \) and \( p^* > sp\).
\end{proof}

\subsection{Proof of the large scale estimate}

We now use \cref{lemma_single_scale_large_oscillation_estimate} to prove \cref{proposition_large_oscillation_estimate}.
The main idea consists in applying \cref{lemma_single_scale_large_oscillation_estimate} with the triangle inequality;  because of the truncation in the left-hand side we need to rely on \cref{lemma_single_scale_large_oscillation_estimate} with values of \(\delta\) arbitrarily close to \(0\).

\begin{proof}[Proof of \cref{proposition_large_oscillation_estimate}]
\resetconstant
By a comparison argument, we have 
\begin{equation}
\label{eq_Supeingeequaehei2aeQuev4}
\smashoperator[r]{\iint_{\substack{(x, y) \in \Omega \times \Omega\\
 d_{\lifting{\manifold{N}}}(\lifting{u} (y), \lifting{u} (x))\ge \delta}}} \frac{d_{\lifting{\manifold{N}}} (\lifting{u} (y), \lifting{u} (x))^{p_*}}{\abs{y - x}^{m + s_*p_*}} \dif y \dif x 
 \le 
 2^{p_*}
 \smashoperator{\iint_{\Omega \times \Omega}} \frac{\brk[\big]{d_{\lifting{\manifold{N}}} (\lifting{u} (y), \lifting{u} (x)) - \frac{\delta}{2}}_+^{p_*}}{\abs{y - x}^{m + s_*p_*}} \dif y \dif x 
 .
\end{equation}
By the triangle inequality and by symmetry, we have then
\begin{multline}
\label{eq_ieThiiTohw0vaivo7eiSh3tu}
  \smashoperator{
    \iint_{\Omega \times \Omega}
  } 
    \frac
      {\brk[\big]{d_{\lifting{\manifold{N}}} (\lifting{u} (y), \lifting{u} (x)) - \frac{\delta}{2}}_+^{p_*}}
      {\abs{y - x}^{m + s_*p_*}} 
    \dif y 
    \dif x 
\\
\le 2^{p_* - 1}
\smashoperator{
\iint_{\Omega \times \Omega} }
\brk[\bigg]{
\;
\smashoperator[r]{
\fint_{\Omega \cap B_{\abs{y -x}/2}(\frac{x + y}{2})}}
\brk[\big]{d_{\lifting{\manifold{N}}} (\lifting{u} (y), \lifting{u} (z)) - \tfrac{\delta}{4}}_+ 
\dif z}^{p_*} \\
\shoveright{+
\brk[\bigg]{
\;
\smashoperator[r]{\fint_{\Omega \cap B_{\abs{y -x}/2}(\frac{x + y}{2})}}
\brk[\big]{d_{\lifting{\manifold{N}}} (\lifting{u} (z), \lifting{u} (x)) - \tfrac{\delta}{4}}_+ 
\dif z}^{p_*} \frac{\dif y \dif x}{\abs{y - x}^{m + s_*p_*}}}
\\
= 
2^{p_*}
\smashoperator{
\iint_{\Omega \times \Omega} 
}
\brk[\bigg]{\;
\smashoperator[r]{
\fint_{\Omega \cap B_{\abs{y -x}/2}(\frac{x + y}{2})}
}
\brk[\big]{d_{\lifting{\manifold{N}}} (\lifting{u} (z), \lifting{u} (x)) - \tfrac{\delta}{4}}_+ 
\dif z
}^{p_*} 
\frac
  {\dif y \dif x}
  {\abs{y - x}^{m + s_*p_*}}
  .
\end{multline}
By \eqref{eq_Supeingeequaehei2aeQuev4} and \eqref{eq_ieThiiTohw0vaivo7eiSh3tu} we have by integration in spherical coordinates.
\begin{multline}
\label{eq_hooph5pheegee9Zaeyii0uRu}
\smashoperator[r]{\iint_{\substack{(x, y) \in \Omega \times \Omega\\
 d_{\lifting{\manifold{N}}}(\lifting{u} (y), \lifting{u} (x))\ge \delta}}} \frac{d_{\lifting{\manifold{N}}} (\lifting{u} (y), \lifting{u} (x))^{p_*}}{\abs{y - x}^{m + s_*p_*}} \dif y \dif x 
\\
\le \C 
\smashoperator{
\iint_{\Omega \times \Omega} }
\brk[\bigg]{\;
\smashoperator[r]{
\fint_{\Omega \cap B_{\abs{y -x}}(x)}
}
\brk[\big]{d_{\lifting{\manifold{N}}} (\lifting{u} (z), \lifting{u} (x)) - \tfrac{\delta}{4}}_+ \dif z}^{p_*}  \frac{\dif y \dif x}{\abs{y - x}^{m + s_*p_*}}\\
\le \C
\int_{\Omega} \int_0^{\diam (\Omega)}
\brk[\bigg]{\;
\smashoperator[r]{
\fint_{\Omega \cap B_{r}(x)} }
\brk[\big]{d_{\lifting{\manifold{N}}} (\lifting{u} (z), \lifting{u} (x) - \tfrac{\delta}{4}}_+ \dif z}^{p_*}  \frac{\dif r}{r^{1 + s_*p_*}} \dif x.
\end{multline}
By the triangle inequality, similarly to the proof of \cref{lemma_Minkowski_average}, we have for almost every \(x \in \Omega\) and every \(r \in \intvo{0}{\diam (\Omega)}\), 
\begin{multline}
\label{eq_hee6oxe0Shoepohgaexie1oh}
  \smashoperator[r]{
    \fint_{\Omega \cap B_r (x)}
    } 
    \brk[\big]{d_{\lifting{\manifold{N}}} (\lifting{u} (z), \lifting{u} (x)) - \tfrac{\delta}{4}}_+\dif z \\
  \le 
  \sum_{j \in \Nset} 
  \;
  \fint\limits_{\Omega \cap B_{2^{-j} r} (x)}
  \smashoperator[r]{\fint_{\Omega \cap B_{2^{-j - 1} r} (x)}}
    \brk[\big]{
      d_{\lifting{\manifold{N}}} (\lifting{u} (y), \lifting{u} (z)) - \delta_j
      }_+ 
    \dif y 
    \dif z
 \\
  \le 
    \Cl{cst_icubibusheFei8aimuith8se}
    \frac
      {r^{2m}}
      {2^{2mj}}  
    \sum_{j \in \Nset} \;\int\limits_{\Omega \cap B_{2^{-j} r} (x)}
    \smashoperator[r]{\int_{\Omega \cap B_{2^{-j} r} (x)}}
    \brk[\big]{
      d_{\lifting{\manifold{N}}} (\lifting{u} (y), \lifting{u} (z)) - \delta_j
      }_+
  \dif y
  \dif z
  ,
\end{multline}
where we have set for each \(j \in \Nset\)
\begin{equation}
\label{eq_iesh4bai0OoN1mie3jeos7Po}
 \delta_j \defeq \frac{\delta\kappa^j}{4(1 - \kappa)},
\end{equation}
with a constant \(\kappa \in \intvo{0}{1}\) to be determined later,
since \(
 \sum_{j \in \Nset} \delta_j 
 = \frac{\delta}{4}
\).
We have then by \eqref{eq_hee6oxe0Shoepohgaexie1oh} and Minkowski's inequality
\begin{multline}
\label{eq_iesh8eir2daith8NeeM8ohXi}
  \brk[\Bigg]{
    \int_{\Omega} 
    \int_0^{\diam \Omega}
      \brk[\bigg]
        {
          \smashoperator[r]{
            \fint_{\Omega \cap B_r (x)}} 
              \brk[\big]
                {d_{\lifting{\manifold{N}}} (\lifting{u} (y), \lifting{u} (x))-\tfrac{\delta}{4}}_+
              \dif y
        }^{p_*} 
        \frac
          {\dif r}
          {r^{1 + s_*p_*}} 
        \dif x
      }^\frac{1}{p_*}
 \\
 \le 
 \Cr{cst_icubibusheFei8aimuith8se}
 \sum_{j \in \Nset} 
 \brk[\Bigg]{
    \int_{\Omega} 
    \int_0^{\diam \Omega}
      \brk[\bigg]
        {
          \frac
            {2^{2mj}}
            {r^{2m}} 
          \smashoperator{
            \iint_{(\Omega \cap B_{2^{-j} r} (x))^2}
            } 
            \brk[\big]{d_{\lifting{\manifold{N}}} (\lifting{u} (y), \lifting{u} (z)) -\delta_j}_+
            \dif y \dif z}^{p_*} 
            \frac
              {\dif r}
              {r^{1 + s_*p_*}} 
            \dif x
    }^\frac{1}{p_*}
    .
\end{multline}
For every \(j \in \Nset\), we have by a change of variable in the outer integral
\begin{multline}
\label{eq_leef9eezee2keo8uD5zeicai}
  \int_0^{\diam \Omega}
    \brk[\bigg]{
      \frac
        {2^{2mj}}
        {r^{2m}} 
        \smashoperator{
          \iint_{(\Omega \cap B_{2^{-j} r} (x))^2}
        }
          \brk[\big]{d_{\lifting{\manifold{N}}} (\lifting{u} (y), \lifting{u} (z)) -\delta_j}_+ 
          \dif y 
          \dif z
        }^{p_*} 
        \frac
          {\dif r}
          {r^{1 + s_*p_*}} 
          \\
    =  
      \frac
        {1}
        {2^{s_* p_* j}} 
      \int_0^{2^{-j} \diam \Omega} 
        \brk[\bigg]{\frac{1}{r^{2m}} \smashoperator{\iint_{(\Omega \cap B_{r} (x))^2}}  \brk[\big]{d_{\lifting{\manifold{N}}}
          (\lifting{u} (y), \lifting{u} (z)) -\delta_j}_+
          \dif y 
          \dif z
        }^{p_*} \frac{\dif r}{r^{1 + s_*p_*}},
\end{multline}
whereas by \cref{lemma_single_scale_large_oscillation_estimate},
\begin{multline}
\label{eq_ca7eiw2muf3Jie0geiw9Ufia}
\int_0^{\diam \Omega} 
  \brk[\bigg]{
    \frac{1}{r^{2m}} 
  \smashoperator{
    \iint_{(\Omega \cap B_{r} (x))^2}
  }
  \brk[\big]{d_{\lifting{\manifold{N}}} (\lifting{u} (y), \lifting{u} (z)) -\delta_j}_+\dif y 
  \dif z
}^{p_*} 
\frac
{\dif r}
{r^{1 + s_*p_*}}
\\[-1em]
\le  
\C
 \brk[\Bigg]{\frac{1}{\delta_j^{(1 - s)p}}
 \smashoperator{\iint_{\Omega\times \Omega}} 
  \frac
    {d_{\manifold{N}} (u (y), u (z))^{p}}
    {\abs{y - z}^{m + s p}}  
  \dif y 
  \dif z}
 .
\end{multline}
Combining \eqref{eq_iesh8eir2daith8NeeM8ohXi}, \eqref{eq_leef9eezee2keo8uD5zeicai} and \eqref{eq_ca7eiw2muf3Jie0geiw9Ufia}, we obtain in view of \eqref{eq_iesh4bai0OoN1mie3jeos7Po}
\begin{multline}
  \brk[\Bigg]{
    \int_{\Omega} 
    \int_0^{\diam \Omega}
    \brk[\bigg]{
      \smashoperator[r]{
        \fint_{\Omega \cap B_r (x)}} 
          \brk[\big]{d_{\lifting{\manifold{N}}} (\lifting{u} (y), \lifting{u} (x))-\tfrac{\delta}{4}}_+
          \dif y}^{p_*} 
          \frac
          {\dif r}
          {r^{1 + s_*p_*}} 
          \dif x}^\frac{1}{p_*}
          \\[-1em]
\le \C 
\sum_{j \in \Nset} 
  \frac
    {1}
    {\brk[\big]{2^{s_*}\kappa^{\frac{1 - s}{s}}}^j}
    \brk[\Bigg]{
      \frac
        {1}
        {\delta^{(1 - s)p}}
        \smashoperator{
          \iint_{\Omega^2}
        } 
        \frac
          {d_{\manifold{N}} (u (y), u (z))^{p}}
          {\abs{y - z}^{m + s p}}  \dif y \dif z}^\frac{1}{sp},
\end{multline}
and the conclusion follows provided \(\kappa \in \intvo{0}{1}\) is chosen in such a way that \(\kappa  > 2^{-\frac{s_* s}{1-s}}\).
\end{proof}

\subsection{Conclusion and further estimate}
We now deduce \cref{theorem_estimate_lifting_noncompact} from \cref{proposition_large_oscillation_estimate}.

\begin{proof}[Proof of \cref{theorem_estimate_lifting_noncompact}]
\resetconstant
We first assume that \(\manifold{M} = \Omega\), where the set \(\Omega \subset \Rset^m\) is open, bounded and convex. 
By \cref{proposition_X_lifting_of_Wsp_in_Y}, we have \(\lifting{u} \in Y (\Omega, \lifting{\manifold{N}})\).
Letting \(p_* \defeq  p\),
we have
\(
  s_* = s + \brk{1 - s}\brk{1 - \frac{m}{sp}}\ge s
\).
We get since \(\Omega\) is bounded and \(s_* p_* \ge s p\),
\begin{equation}
\label{eq_cai1rivai5daiyaThauzoh8o}
\begin{split}
\smashoperator[r]{\iint_{\substack{(x, y) \in \Omega \times \Omega\\
 d_{\lifting{\manifold{N}}}(\lifting{u} (y), \lifting{u} (x))\ge \inj(\manifold{N})}}} \frac{d_{\lifting{\manifold{N}}} (\lifting{u} (y), \lifting{u} (x))^{p}}{\abs{y - x}^{m + sp}} \dif y \dif x 
 &\le\C
 \smashoperator{\iint_{\substack{(x, y) \in \Omega \times \Omega\\
 d_{\lifting{\manifold{N}}}(\lifting{u} (y), \lifting{u} (x))\ge \inj{\manifold{N}}}}} \frac{d_{\lifting{\manifold{N}}} (\lifting{u} (y), \lifting{u} (x))^{p_*}}{\abs{y - x}^{m + s_*p_*}} \dif y \dif x \\
 &\le 
 \C
 \brk[\bigg]{\;\smashoperator[r]{\iint_{\Omega\times \Omega}} \frac{d_{\manifold{N}} (u (y), u (x))^{p}}{\abs{y - x}^{m + s p}}  \dif y \dif x}^\frac{1}{s}
 ,
 \end{split}
\end{equation}
by \cref{proposition_large_oscillation_estimate} with \(\delta = \inj\brk{\manifold{N}}\).
Combining the estimate \eqref{eq_cai1rivai5daiyaThauzoh8o} with \cref{proposition_spgt1_estimate}, we get
\begin{equation}
\label{eq_yo4ooYai8ahp3heeMeifie1s}
\begin{split}
\smashoperator[r]{\iint_{\Omega \times \Omega}}& \frac{d_{\lifting{\manifold{N}}} (\lifting{u} (y), \lifting{u} (x))^{p}}{\abs{y - x}^{m + sp}} \dif y \dif x \\
&\le\C \brk[\Bigg]{\;\smashoperator[r]{  \iint_{\Omega\times \Omega}} \frac{d_{\manifold{N}} (u (y), u (x))^{p}}{\abs{y - x}^{m + s p}}  \dif z \dif y + \brk[\bigg]{\;\smashoperator[r]{\iint_{\Omega\times \Omega}} \frac{d_{\manifold{N}} (u (y), u (x))^{p}}{\abs{y - x}^{m + s p}}  \dif y \dif x}^{\frac{1}{s}} }
.
 \end{split}
\end{equation}
We reach the conclusion \eqref{eq_Pei1Chu7Amie6ohz0phi6cai} on a general compact manifold \(\manifold{M}\) thanks to the estimate \eqref{eq_yo4ooYai8ahp3heeMeifie1s} and the covering of  \cref{lemma_manifold_diagonal_covering}.
\end{proof}

\begin{remark}
\resetconstant
The exponent \(\frac{1}{s}\) in \eqref{eq_Pei1Chu7Amie6ohz0phi6cai} is optimal.
Indeed,  assuming that 
\begin{equation}
\label{eq_jaiGeejeiCeish7wieyahnae}
\norm{\lifting{u}}_{\homog{W}^{s, p}}
\le \C \brk[\big]{\norm{u}_{\homog{W}^{s, p}} + \norm{u}_{\homog{W}^{s, p}}^\gamma}  
\end{equation}
holds and 
taking \(\pi : \Rset \to \Sset^1\) to be the universal covering of the unit circle and \(\lifting{u} = t \varphi\), for some \(\varphi \in C^\infty (\manifold{M}, \Rset)\) and every \(t \in \Rset\), one gets from \eqref{eq_Pei1Chu7Amie6ohz0phi6cai} that 
\(
  \abs{t} \le \C (\abs{t}^s + \abs{t}^{\gamma s})\),
which can only hold if \(\gamma \ge \frac{1}{s}\).
\end{remark}

\Cref{proposition_large_oscillation_estimate} can also be applied to obtain a result in which the nonlinear part in the estimate contains a critical fractional Sobolev energy.

\begin{theorem}
\label{theorem_Wsp_Wmpp}
Let \(\manifold{M}\) and \(\manifold{N}\) be compact Riemannian manifolds, let \(m \defeq  \dim \manifold{M}\), let \(\pi : \lifting{\manifold{N}} \to \manifold{N}\) be a Riemannian covering map, let \(s \in \intvo{0}{1}\) and let \(p \in \intvo{1}{\infty}\). 
If \(sp > 1\), then there exists a constant \(C \in \intvo{0}{\infty}\) such that for every \(\lifting{u} \in X (\manifold{M}, \lifting{\manifold{N}})\), we have \(\lifting{u} \in \homog{W}^{s, p} (\manifold{M}, \lifting{\manifold{N}})\) and 
\begin{multline}
\label{eq_lazodahXa9of7ieJaephe5no}
 \smashoperator[r]{\iint_{\manifold{M} \times \manifold{M}}} \frac{d_{\lifting{\manifold{N}}} (\lifting{u} (y), \lifting{u} (x))^p}{d_{\manifold{M}} (y, x)^{m + sp}} \dif y \dif x\\[-1em]
 \le 
 \C \brk[\Bigg]{1 + \smashoperator{\iint_{\manifold{M}\times \manifold{M}}} \frac{d_{\manifold{N}} (u (y), u (x))^p}{d_{\manifold{M}} (y, x)^{2 m}} \dif y \dif x}^{\frac{(1 - s)p}{m}}
 \smashoperator{\iint_{\manifold{M} \times \manifold{M}}} \frac{d_{\manifold{N}} (u (y), u (x))^p}{d_{\manifold{M}} (y, x)^{m + sp}} \dif y \dif x.
\end{multline}

\end{theorem}

Although the no restriction is put on the exponent, in practice the first integral in the right-hand side will be finite for some nonconstant function \(u\) if and only if \(p > m\).

\begin{proof}
[Proof of \cref{theorem_Wsp_Wmpp}]
\resetconstant
We proceed as in the proof of \cref{theorem_estimate_lifting_noncompact}, 
applying now \cref{proposition_large_oscillation_estimate} 
with \(s\) being given by \(s_0 = \frac{1}{1 + (1-s)p/m}\), \(p_* = p\) and
so that \(s_*\) is then given by \(s\) 
in \eqref{eq_Foreigie4zaiw2ze0dah8nah}. Since \(sp > 1\), we have \(s_0 p = \frac{p}{1 + (1-s)p/m} > 1\),
and thus by  \cref{proposition_large_oscillation_estimate}
\begin{equation}
\label{eq_se4inamoh1oe9ooy8eiYohxo}
 \smashoperator{\iint_{\substack{(x, y) \in \Omega \times \Omega\\
 d_{\lifting{\manifold{N}}}(\lifting{u} (y), \lifting{u} (x))\ge \inj(\manifold{N})}}} \frac{d_{\lifting{\manifold{N}}} (\lifting{u} (y), \lifting{u} (x))^p}{\abs{y - x}^{m + sp}} \dif y \dif x \le 
 \C \brk[\Bigg]{\;\smashoperator{\iint_{\Omega \times \Omega}} \frac{d_{\manifold{N}} (u (y), u (x))^p}{\abs{y - x}^{m + \frac{mp}{m + (1 - s)p}}} \dif y \dif x}^\frac{m + (1 - s)p}{m}.
\end{equation}
If \(sp \ge m\), we have \(
  m \le \frac{mp}{m + (1 - s)p}\le sp,
\) whereas if \(sp \le m\) we have \(
  sp \le \frac{mp}{m + (1 - s)p}\le m
\),
and thus by Hölder's inequality and \eqref{eq_se4inamoh1oe9ooy8eiYohxo} we get
\begin{multline}
\label{eq_kui1zielaiCh3cou9aeTae1y} \smashoperator{\iint_{\substack{(x, y) \in \Omega \times \Omega\\
 d_{\lifting{\manifold{N}}}(\lifting{u} (y), \lifting{u} (x))\ge \inj(\manifold{N})}}}
  \frac{d_{\lifting{\manifold{N}}} (\lifting{u} (y), \lifting{u} (x))^p}{\abs{y - x}^{m + sp}} \dif y \dif x\\[-1em] \le 
 \C \brk[\bigg]{1 + \smashoperator{\iint_{\Omega\times \Omega}} \frac{d_{\manifold{N}} (u (y), u (x))^p}{\abs{y - x}^{2 m}} \dif y \dif x}^{\frac{(1 - s)p}{m}}
 \smashoperator{\iint_{\Omega \times \Omega}} \frac{d_{\manifold{N}} (u (y), u (x))^p}{\abs{y - x}^{m + sp}} \dif y \dif x.
\end{multline}
Hence combining \eqref{eq_kui1zielaiCh3cou9aeTae1y} with \cref{proposition_spgt1_estimate}, we get
\begin{multline}
\label{eq_eeth3ay4loos7kaliuDo1ohZ} \smashoperator{\iint_{\Omega \times \Omega}}
  \frac{d_{\lifting{\manifold{N}}} (\lifting{u} (y), \lifting{u} (x))^p}{\abs{y - x}^{m + sp}} \dif y \dif x\\[-1em] \le 
 \C \brk[\bigg]{1 + \smashoperator{\iint_{\Omega\times \Omega}} \frac{d_{\manifold{N}} (u (y), u (x))^p}{\abs{y - x}^{2 m}} \dif y \dif x}^{\frac{(1 - s)p}{m}}
 \smashoperator{\iint_{\Omega \times \Omega}} \frac{d_{\manifold{N}} (u (y), u (x))^p}{\abs{y - x}^{m + sp}} \dif y \dif x.
\end{multline}
Combining \eqref{eq_eeth3ay4loos7kaliuDo1ohZ} with the covering of \cref{lemma_manifold_diagonal_covering}, we conclude.
\end{proof}

\begin{remark}
\resetconstant
Again the exponent \(\frac{1}{s}\) in \eqref{eq_Pei1Chu7Amie6ohz0phi6cai} is optimal.
Indeed,  assuming that we have
\begin{equation}
\label{eq_nae1Mai4eix0waishae6ihin}
\norm{\lifting{u}}_{\homog{W}^{s, p}}
\le \C \brk[\big]{1 + \norm{u}_{\homog{W}^{m/p, p}}^\gamma}  \norm{u}_{\homog{W}^{s, p}}
\end{equation}
and
taking \(\pi : \Rset \to \Sset^1\) to be the universal covering of the unit circle and \(\lifting{u} = t \varphi\), for some \(\varphi \in C^\infty (\manifold{M}, \Rset)\) and every \(t \in \Rset\), one gets from \eqref{eq_nae1Mai4eix0waishae6ihin} that \(\abs{t} \le \C (1 + \abs{t}^{\gamma m/p})\abs{t}^s\),
which can only hold if \(\gamma \ge \frac{(1-s)p}{m}\).
\end{remark}
Finally, the same methods can be used to get an estimate on a  lower order fractional Sobolev energy when the dimension is supercritical.

\begin{theorem}
\label{theorem_supercritical_estimate}
Let \(\manifold{M}\) and \(\manifold{N}\) be compact Riemannian manifold, let \(m \defeq  \dim \manifold{M}\), let \(\pi : \lifting{\manifold{N}} \to \manifold{N}\) be a Riemannian covering map, let \(s \in \intvo{0}{1}\) and let \(p \in \intvo{1}{\infty}\). 
If
\begin{equation}
\label{eq_NiiPeiJoo4Sing6ahyu4ua6U}
1 - s < \frac{sp}{m} < 1
\end{equation}
and if \(\lifting{u} \in X (\manifold{M}, \lifting{\manifold{N}})\), then \(\lifting{u} \in \homog{W}^{s_\flat, p} (\manifold{M}, \lifting{\manifold{N}})\) and 
\begin{equation}
\begin{split}
\smashoperator[r]{\iint_{\manifold{M} \times \manifold{M}}} &\frac{d_{\lifting{\manifold{N}}} (\lifting{u} (y), \lifting{u} (x))^{p}}{d_{\manifold{M}} (y, x)^{m + s_\flat p}} \dif y \dif x \\
 &\le 
 \C
 \brk[\Bigg]{\;
 \smashoperator[r]{\iint_{\manifold{M}\times \manifold{M}}} \frac{d_{\manifold{N}} (u (y), u (x))^{p}}{d_{\manifold{M}} (y, x)^{m + s p}}  \dif y \dif x
 +
 \brk[\bigg]{\;
 \smashoperator[r]{\iint_{\manifold{M}\times \manifold{M}}} \frac{d_{\manifold{N}} (u (y), u (x))^{p}}{d_{\manifold{M}} (y, x)^{m + s p}}  \dif y \dif x}^\frac{1}{s}
 \,
 }. 
\end{split}
\end{equation}
with 
\begin{equation}
\label{eq_Ua7heJairahte6ohWeighoo8}
s_\flat \defeq s - (1 - s)\brk[\Big]{\frac{m}{sp} - 1}.
\end{equation}
\end{theorem}

\begin{proof}
\resetconstant
We follow the structure of the proof of  \cref{theorem_estimate_lifting_noncompact}.
Considering \(\lifting{u} \in Y (\Omega, \lifting{\manifold{N}})\), we apply \cref{proposition_large_oscillation_estimate} with \(p_* =  p\) so that \(s_* = s_{\flat}\) in view of \eqref{eq_Ua7heJairahte6ohWeighoo8} since by \eqref{eq_NiiPeiJoo4Sing6ahyu4ua6U}
\begin{equation*}
  s_\flat p = p + m \brk[\Big]{\frac{1}{s} - 1}>1
\end{equation*}
and we get 
\begin{equation*}
\smashoperator[r]{\iint_{\substack{(x, y) \in \Omega \times \Omega\\
 d_{\lifting{\manifold{N}}}(\lifting{u} (y), \lifting{u} (x))\ge \inj(\manifold{N})}}} \frac{d_{\lifting{\manifold{N}}} (\lifting{u} (y), \lifting{u} (x))^{p}}{\abs{y - x}^{m + s_\flat p}} \dif y \dif x 
 \le 
 \C
 \brk[\bigg]{\;\smashoperator[r]{\iint_{\Omega \times \Omega}} \frac{d_{\manifold{N}} (u (y), u (x))^{p}}{\abs{y - x}^{m + s p}}  \dif y \dif x}^\frac{1}{s}.
\end{equation*}
On the other hand by \cref{proposition_spgt1_estimate}, since \(s_\flat <s\) and since the set \(\Omega\) is bounded, we get 
\begin{equation}
\label{eq_chahxi4AW6theiThoGhohpie}
\begin{split}
\smashoperator[r]{\iint_{\Omega \times \Omega}} &\frac{d_{\lifting{\manifold{N}}} (\lifting{u} (y), \lifting{u} (x))^{p}}{\abs{y - x}^{m + s_\flat p}} \dif y \dif x \\
 &\le 
 \C
 \brk[\Bigg]{
 \; \smashoperator[r]{\iint_{\Omega\times \Omega}} \frac{d_{\manifold{N}} (u (y), u (x))^{p}}{\abs{y - x}^{m + s p}}  \dif y \dif x
 +
 \brk[\bigg]{\;\smashoperator[r]{\iint_{\Omega \times \Omega}} \frac{d_{\manifold{N}} (u (y), u (x))^{p}}{\abs{y - x}^{m + s p}}  \dif y \dif x}^\frac{1}{s}\,
 }. 
\end{split}
\end{equation}
The conclusion follows from \eqref{eq_chahxi4AW6theiThoGhohpie} and \cref{lemma_manifold_diagonal_covering}.
\end{proof}

\begin{remark}
The value \(s_\flat\) in the statement of \cref{theorem_supercritical_estimate} is optimal: taking \(\pi : \Rset \to \Sset^1\) to be the universal covering of the unit circle and defining \(\lifting{u} (x) \defeq \abs{x}^{-\alpha}\), then \(u \in \homog{W}^{1, sp} (\Bset^m, \Rset)\) if and only if \((\alpha + 1)sp < m\).
By the fractional Gagliardo--Nirenberg interpolation inequality, one has then \(\pi \compose \lifting{u} \in \homog{W}^{s, p} (\Bset^m, \Sset^1)\).
We also have \(u \in \homog{W}^{s_*, p} (\Bset^m, \Rset)\) if and only if \((\alpha + s_*) p < m\).
This implies that we can have \(\lifting{u} \not \in \homog{W}^{s_*, p} (\Bset^m, \Rset)\) and \(u \in \homog{W}^{s, p} (\Bset^m, \Sset^1)\), when \(\frac{m}{p} - s_* < \frac{m}{sp} - 1\), or equivalently \(s_* > s_\flat\).
\end{remark}

\end{document}